\documentclass[a4paper,11pt,reqno]{article}
\usepackage[margin=1.2in]{geometry}
\usepackage{dsfont}
\usepackage{amsmath}
\usepackage{amsfonts}
\usepackage{amsthm}
\usepackage{amssymb}
\usepackage[colorlinks]{hyperref}
\usepackage[matrix,arrow,curve,frame]{xy}
\usepackage{enumerate}
\usepackage{mathtools}
\usepackage{setspace}
\usepackage{pgfplots}
\pgfplotsset{compat=1.13}
\usepackage{braids}
\usepackage{tikz}
\usepackage{tcolorbox}

\usepackage{xcolor}
\definecolor{darkgreen}{rgb}{0,0.5,0}
\definecolor{darkblue}{rgb}{0,0.1,0.5}

\usepackage{tikz-cd}

\hypersetup{
    colorlinks=true, 
    linkcolor=darkblue, 
    urlcolor=blue, 
    linktoc=all, 
    citecolor=darkgreen
    }

\usepackage[symbol]{footmisc}

\usepackage{thmtools}
\newtheoremstyle{introTheorems}
  {\topsep}
  {\topsep}
  {\itshape}
  {0pt}
  {\bfseries}
  {}
  { }
  {\thmname{#1}
  \textnormal{\thmnote{#3}.}
  }

\newtheorem{theorem}{Theorem}[section]
\newtheorem{assumption}[theorem]{Assumption}

\newtheorem{conjecture}[theorem]{Conjecture}
\newtheorem{corollary}[theorem]{Corollary}
\newtheorem{definition}[theorem]{Definition}
\newtheorem{example}[theorem]{Example}
\newtheorem{lemma}[theorem]{Lemma}
\newtheorem{problem}[theorem]{Problem}

\newtheorem{remark}[theorem]{Remark}
\theoremstyle{introTheorems}
\newtheorem{introTheorem}{Theorem}
\newtheorem{introLemma}{Lemma}
\newtheorem{introExample}{Example}

\usepackage{calc}
\newcommand{\grau}{gray!60}
\newenvironment{grform}{\begin{tikzpicture}[intext]}{\end{tikzpicture}}
\tikzset{
norm/.style = {ultra thick, color = #1},
norm/.default = black,
intext/.style = {baseline = (current bounding box.center)},
}

\newcommand{\vLine}[5]{
  \draw[ultra thick, color = #5, rounded corners] (#1 , #2) -- (#1 , #2 * 0.7 + #4 * 0.3 ) -- ( #3 , #2 * 0.3 + #4 * 0.7 ) -- (#3 ,#4);
}
\newcommand{\vLineO}[5]{
  \draw[line width = 6pt , color = white, rounded corners] (#1 , #2) -- (#1 , #2 * 0.7 + #4 * 0.3 ) -- ( #3 , #2 * 0.3 + #4 * 0.7 ) -- (#3 ,#4);
  \draw[ultra thick, color = #5, rounded corners] (#1 , #2) -- (#1 , #2 * 0.7 + #4 * 0.3 ) -- ( #3 , #2 * 0.3 + #4 * 0.7 ) -- (#3 ,#4);
}

\newcommand{\dMult}[5]{
\draw[ultra thick, color = #5] (#1 , #2) -- (#1 , #2 + #4 * 0.2) .. controls (#1
, #2 + #4 * 0.5) .. (#1 + #3 *0.25 , #2 + #4 * 0.5) 
-- (#1 + #3 *0.5 , #2 + #4 * 0.5) -- (#1 + #3 *0.75 , #2 + #4 * 0.5) .. controls
(#1 + #3 , #2 + #4 * 0.5) .. (#1 + #3 , #2 + #4 * 0.2) -- (#1 + #3 , #2);
\draw[ultra thick, color = #5] (#1 + #3 *0.5 , #2 + #4 * 0.5) -- (#1 + #3 *0.5 ,
#2 + #4);
}

\newcommand{\dMultO}[5]{
\draw[line width = 6pt , color = white] (#1 , #2) -- (#1 , #2 + #4 * 0.2) ..
controls (#1 , #2 + #4 * 0.5) .. (#1 + #3 *0.25 , #2 + #4 * 0.5) 
-- (#1 + #3 *0.5 , #2 + #4 * 0.5) -- (#1 + #3 *0.75 , #2 + #4 * 0.5) .. controls
(#1 + #3 , #2 + #4 * 0.5) .. (#1 + #3 , #2 + #4 * 0.2) -- (#1 + #3 , #2);
\draw[ultra thick, color = #5] (#1 , #2) -- (#1 , #2 + #4 * 0.2) .. controls (#1
, #2 + #4 * 0.5) .. (#1 + #3 *0.25 , #2 + #4 * 0.5) 
-- (#1 + #3 *0.5 , #2 + #4 * 0.5) -- (#1 + #3 *0.75 , #2 + #4 * 0.5) .. controls
(#1 + #3 , #2 + #4 * 0.5) .. (#1 + #3 , #2 + #4 * 0.2) -- (#1 + #3 , #2);
\draw[ultra thick, color = #5] (#1 + #3 *0.5 , #2 + #4 * 0.5) -- (#1 + #3 *0.5 ,
#2 + #4);
}

\newcommand{\dAction}[6]{
\draw[ultra thick, color = #5] (#1 , #2) -- (#1 , #2 + #4 * 0.2) .. controls (#1
, #2 + #4 * 0.5) .. (#1 + #3 * 0.8 , #2 + #4 * 0.5) -- (#1 + #3 , #2 + #4 *
0.5);
\draw[ultra thick, color = #6] (#1 + #3 , #2) -- (#1 + #3 , #2 + #4);
}

\makeatletter
\renewcommand*\env@matrix[1][*\c@MaxMatrixCols c]{%
  \hskip -\arraycolsep
  \let\@ifnextchar\new@ifnextchar
  \array{#1}}

\newcommand{\id}{{\rm id}}
\newcommand{\Id}{{\rm Id}}

\newcommand{\Rep}{{\rm Rep}}
\renewcommand{\dim}{{\rm dim}}
\newcommand{\Hom}{{\rm Hom}}
\newcommand{\End}{{\rm End}}
\newcommand{\Ext}{{\rm Ext}}
\newcommand{\ses}[3]{0 \ra #1 \ra #2 \ra #3 \ra 0} 
\newcommand{\ra}{\rightarrow}

\newcommand{\catM}{\mathcal O^T_{\mathcal M(p)}}
\newcommand{\YD}[1]{{}_{#1}^{#1}\mathcal{YD}}

\newcommand{\cF}{\mathcal{F}}
\newcommand{\cG}{\mathcal{G}}

\newcommand{\cM}{\mathcal{M}}
\newcommand{\cZ}{\mathcal{Z}}

\newcommand{\cW}{\mathcal{W}}
\newcommand{\cU}{\mathcal{U}}

\newcommand{\CC}{\mathcal{C}}
\newcommand{\DD}{\mathcal{D}}
\newcommand{\BB}{\mathcal{B}}
\newcommand{\UU}{\mathcal{U}}

\newcommand{\g}{\mathfrak{g}}
\newcommand{\h}{\mathfrak{h}}
\newcommand{\zem}{\mathfrak{Z}}

\newcommand{\C}{\mathbb{C}}

\newcommand{\Z}{\mathbb{Z}}

\newcommand{\one}{\mathbf{1}}

\newcommand{\sF}{\mathsf{F}}

\newcommand{\Vect}{\mathrm{Vect}}

\usepackage{lineno}
\newcommand{\marginparIf}[1]{  }

\renewcommand{\i}{\mathrm{i}}

\newcommand{\cat}{\UU}
\newcommand{\catA}{{\cat}_A}
\newcommand{\catAloc}{{\cat}_A^0}
\newcommand{\forgetA}{\mathrm{forget}_A}
\newcommand{\induceA}{\mathrm{ind}_A}
\newcommand{\bV}{\mathbb{V}}
\newcommand{\coVerma}{\mathrm{coV}}

\renewcommand{\sl}{\mathfrak{sl}}
\newcommand{\V}{\mathcal{V}}
\newcommand{\W}{\mathcal{W}}

\newcommand{\Y}{\mathcal{Y}}

\begin{document}
\title{An algebraic theory for logarithmic Kazhdan-Lusztig correspondences}
\author{Thomas Creutzig, Simon Lentner, and Matthew Rupert}

\maketitle









\begin{abstract}
Let $\UU$ be a braided tensor category, typically unknown, complicated and in particular non-semisimple. We characterize $\UU$ under the assumption that there exists a commutative algebra $A$ in $\UU$ with certain properties: 
Let $\CC$ be the category of local $A$-modules in $\UU$ and $\BB$ the category of $A$-modules in $\UU$, which are in our set-up usually much simpler categories than $\UU$. Then we can characterize $\UU$ as a relative Drinfeld center $\mathcal Z_\CC(\BB)$ and $\BB$  as representations of a certain Hopf algebra inside $\CC$.

In particular this allows us to reduce braided tensor equivalences to the knowledge of abelian equivalences, e.g.~if we already know that $\UU$ is abelian equivalent to the category of modules of some quantum group $U_q$ or some generalization thereof, and if $\CC$ is braided equivalent to a category of graded vector spaces, and if $A$ has a certain form, then we already obtain a braided tensor equivalence between $\UU$ and $\Rep(U_q)$.

A main application of our theory is to prove logarithmic Kazhdan-Lusztig correspondences, that is, equivalences of braided tensor categories of representations of vertex algebras and of quantum groups. Here, the algebra $A$ and the corresponding category $\CC$ are a-priori given by a free-field realization of the vertex algebra and by a Nichols algebra. We illustrate this in those examples where the representation theory of the vertex algebra is well enough understood. In particular we prove the conjectured correspondences between the singlet vertex algebra $\cM(p)$ and the unrolled small quantum groups of $\mathfrak{sl}_2$ at $2p$-th root of unity. Another new example is the Kazhdan-Lusztig correspondence between the affine vertex algebra of $\mathfrak{gl}_{1|1}$ and an unrolled quantum group of $\mathfrak{gl}_{1|1}$.
\end{abstract}

\quad\\

\setcounter{tocdepth}{2}
\tableofcontents

\setlength\parskip{1em}
\setlength\parindent{0pt}

\newcommand{\Localization}{\mathfrak{l}}

\newcommand{\Schauenburg}{\mathcal{S}}
\newcommand{\SchauenburgBar}{\overline{\mathcal{S}}}

\newcommand{\Nichols}{\mathfrak{N}}
\newcommand{\NicholsOf}{\mathfrak{B}}

\newcommand{\assoz}{a}

\section{Introduction}

To a simple Lie algebra $\g$ over $\C$ one has an associated quantum group $U_q(\g)$ as well as the affinization $\widehat{\g}$ of $\g$. 
Both algebras allow for categories that are similar to integrable $\g$-modules. For the quantum group this is 
the category of weight modules and for the affine Lie algebra this is the Kazhdan-Lusztig category of level $k$, i.e.~smooth modules at level $k$ whose conformal weight subspaces  are integrable as modules for the horizontal subalgebra. For generic $q$ and $k$, these are abelian equivalent to the category of integrable $\g$-modules and this abelian equivalence is a quite trivial observation. The category of weight modules of the quantum group is also naturally a braided tensor category, while smooth modules of an affine Lie algebra at level $k$ are in particular modules of the affine vertex operator algebra $V^k(\g)$ of $\g$ at level $k$. 
A vertex operator algebra is a rigorous notion of chiral algebra of two-dimensional conformal field theory and in particular it is widely believed that suitable categories of vertex operator algebra modules are braided tensor categories as well. In a series of highly impressive works in the 1990's Kazhdan and Lusztig were able to construct such an expected braided tensor category structure on the Kazhdan-Lusztig category of level $k$ and even better they proved equivalences as braided tensor categories with quantum group categories \cite{KL1, KL2, KL3, KL4, KL5}. A general expectation is that braided tensor categories of modules of quantum groups (and maybe more general quasi Hopf algebras) can also be realized as braided tensor categories of modules of suitable vertex operator algebras. This expectation is broadly referred to as the Kazhdan-Lusztig correspondence and modernly one often adds the adjective logarithmic. The logarithmic Kazhdan-Lusztig correspondences are concerned with non-semisimple categories and the word logarithmic has its origin from logarithmic conformal field theories, which are conformal field theories with non-semisimple representations. 

Non-semisimple braided tensor categories play a prominent role in the mathematics arising from quantum field theories. In particular seemingly very different categories, as quantum group categories, vertex operator algebra categories and categories of line operators are predicted to be braided tensor equivalent, see e.g. \cite{CDGG}. Categories of line operators are formulated as certain categories of, for example, D-modules or coherent sheaves on certain loop spaces, see e.g. \cite{BF18, HR21, W19, BFN21}, and are currently mostly understood on the abelian level. 

We noted that in the original work of Kazhdan and Lusztig the very deep statement is the braided tensor equivalence, while the abelian equivalence has been clear. Our objective in the present article is to develop a theory of (logarithmic) Kazhdan-Lusztig correspondences in the sense that if one can prove an abelian equivalence of a category $\UU$
with some quantum group category together with suitable braided tensor equivalences of much simpler categories, then one already must get a braided tensor equivalence as well. 


\subsection{Characterizing braided tensor categories}

Algebraically, we develop tools to describe and classify braided tensor categories in the following situation, which is set up in Section \ref{sec_extension}.
We work over the base field $\C$. 
Let $\UU$ be a braided tensor category, typically unknown and complicated, usually non-semisimple. Suppose that $A$ is a {commutative algebra} inside $\UU$, then there is an associated category of representations $\UU_A$ inside $\UU$ with a tensor product $\otimes_A$, and a subcategory of \emph{local modules}~$\UU_A^0$, i.e.~modules on which the action of $A$ is not changed when precomposed with a double-braiding. The category of local modules admits again a braiding, and we imagine this category $\UU_A^0$ to be known and easier. We summarize the situation with the standard induction tensor functor, restriction functor and embedding tensor functor in the following diagram
\vspace{-1cm}
\begin{center}
$$
 \begin{tikzcd}[row sep=10ex, column sep=15ex]
   \mathbf{\UU} \arrow{r}{\induceA}  
   & \arrow[shift left=2]{l}{\forgetA}   \mathbf{\BB} = & \hspace{-4.75cm} \catA \\
   &\arrow[hookrightarrow]{u}{\iota}
   \mathbf{\CC} = & \hspace{-4.75cm}\catAloc
    \end{tikzcd}
$$
\end{center}
There are two main classes of examples for this situation we have in mind: 
\begin{itemize}
    \item To any semisimple finite-dimensional complex Lie algebra $\g$ there is associated a \emph{quantum group} $U_q(\g)$ for a generic complex parameter $q$, a Hopf algebra whose category of representations is the same as for $\g$ with nontrivial associativity and braiding isomorphisms depending on $q$. For $q$ a root of unity there is associated a non-semisimple small quantum group $u_q(\g)$ with category of representations resembling $\g$ in finite characteristic.

    If $\UU$ is the category of representations of a quantum group and $A$ is the dual Verma module of weight zero, which is an algebra, then $\BB$ is the category of representations of the  Borel part and $\CC$ is the category of representations of the Cartan part.
    
    \item A  \emph{vertex operator algebra} (VOA) $\mathcal{V}$ is a rigorous formulation of the chiral algebra of a two-dimensional conformal field theory. Its algebra structure depends on a formal variable $z$ and one might like to call it a quantum analogue of a commutative algebra. 
    In good cases the category of suitable modules form a braided tensor category $\Rep(\mathcal{V})$, with braiding given by the monodromy around singular points $z$.

    For vertex algebras $\mathcal{V}\supset \mathcal{W}$, the larger vertex algebra $\mathcal{V}$ becomes a commutative algebra in the braided tensor category $\UU=\Rep(\cW)$, and the representation category of the larger is recovered as the category of local modules $\CC=\Rep(\V)$, with $\BB$ being the category of twisted modules. This is in particular studied if $\cW=\V^G$ is the fixed subalgebra under some group action ($G$-orbifold), or if $\V$ is a free-field theory and $\cW$ is the kernel of so-called screening operators, which facilitate the action of a Lie algebra or Nichols algebra on the vertex operator algebra. The \emph{logarithmic Kazhdan Lusztig conjecture} links in such cases $\Rep(\cW)$ with the representations of a small quantum group, as discussed further below. 
\end{itemize}

In this setting, and with these applications in mind, our structural results are as follows: In Section \ref{sec_Schauenburg} we prove in a non-semisimple setting based on work by Schauenburg \cite{Sch} under some technical Assumptions \ref{assumption} ($\UU$ rigid, braided, locally finite with trivial M\"{u}ger center, and  $A$ commutative, haploid, and $\cU_A$ rigid) that the category $\UU$ embeds into a relative center:

\begin{introTheorem}[\ref{relcent}]
Let $\UU,\, A,\, \UU_A$ be as in \textup{Assumption \ref{assumption}}.

Then there is a fully faithful functor to the relative Drinfeld center
\[ \Schauenburg:  \mathcal{U} \hookrightarrow \mathcal{Z}_{\mathcal{U}_A^0}(\mathcal{U}_A),  \qquad
X \mapsto (\induceA(X), b_{\bullet, X}). \]
\end{introTheorem}

Here, the Drinfeld center $\cZ(\BB)$ is a braided tensor category associated to any tensor category $\BB$, and there is a notion of a relative Drinfeld center $\cZ_\CC(\BB)$ with respect to a suitable braided subcategory $\CC$, see \cite{LW1} and Definition \ref{def_relcenter} below. This functor is in fact an equivalence for finite categories, due to dimension arguments, see Corollary \ref{cor_SchauenburgFinite}. However we can also prove this  for infinite tensor categories where in a certain sense $\UU$ is finite over $\CC$, see Corollary \ref{cor_SchauenburgInfinite}. So the braided tensor category $\UU$ is completely determined by the typically less complicated tensor category $\BB = \UU_A$.\\

Our next focus is hence to understand the tensor category $\BB$ in relation to the braided tensor subcategory $\CC$. In Section \ref{sec_splitting} we study the particular case, where the inclusion $\CC\hookrightarrow \BB$ admits a left-inverse \emph{splitting tensor functor} $\Localization:\CC\twoheadleftarrow \BB$. While this is typically not true for $G$-orbifolds, it seems to be a general feature of the Kazhdan-Lusztig setting. The following technical lemma shows why  this is true, due to the ``nilpotent" situation:

\begin{introLemma}[\ref{lm_socle}]
Let $\iota:\CC \hookrightarrow \BB$ be a full Serre tensor subcategory, where $\CC$ is semisimple and $\BB$ consists of objects with finite Jordan-H\"older length. Assume that all simple objects of $\BB$ lie in $\CC$, then there exists a splitting tensor functor $\Localization:\BB\to \CC$.
\end{introLemma}

The typical setting in which there is a splitting functor  is if $\BB=\Rep(\Nichols)(\CC)$ for some Hopf algebra $\Nichols$ inside the braided tensor category $\CC$. In Lemma \ref{lm_infiniteSplittingIsHopfAlgB} we prove the converse statement, again in a version for infinite tensor categories where reconstruction statements may not be available: If $\BB$ over $\CC$ admits a splitting tensor functor $\Localization$ and if $\BB=\Rep(\Nichols)(\CC)$ for a given algebra $\Nichols$ in $\CC$, and if certain $\CC$-module category isomorphisms hold, then $\Nichols$ is already a Hopf algebra in $\CC$ that realizes $\BB$ as a tensor category. 

Our next focus in Section \ref{sec_Nichols} is hence to characterize certain Hopf algebras $\Nichols$ in a braided tensor category $\CC$. A generic construction for $\Nichols$ is the \emph{Nichols algebra} $\NicholsOf(X)$ of an object $X\in \CC$, a quotient of the tensor algebra $\mathfrak{T}(X)$ fulfilling several universal properties. For example, it is the smallest Hopf algebra in $\CC$ generated by $X$ where elements $x\in X$ have a derivation--type coproduct $\Delta(x)=1\otimes x + x\otimes 1$. The simplest example is 
\begin{introExample}[\ref{exm_NicholsRank1}]
The Nichols algebra of a $1$-dim.~object with braiding $x\otimes x\mapsto q(x\otimes x)$ is
$$\NicholsOf(X)=\begin{cases}
\C[x]/x^p,\qquad &q\text{ a primitive root of unity of order $p>1$}\\ 
\C[x],\qquad &\text{else}
\end{cases}$$
\end{introExample}
A standard textbook on the subject is \cite{HS20}. Finite dimensional Nichols algebras with a braiding of diagonal type $x_i\otimes x_j\mapsto q_{ij}(x_j\otimes x_i)$ have been classified by Heckenberger \cite{Heck09}, and they include quantum Borel parts $u_q(\g)^+$ of quantum groups, and quantum super groups and some additional cases. A key observation is that Nichols algebras under suitable conditions admit a generalized root system (with integrality conditions present, but not necessarily isomorphic Weyl chambers), which can be classified by themselves \cite{CH09}. The existence of a generalized root system is a feature that continues in more general situations \cite{HS10, AHS10}. A prominent appearance of Nichols algebras is the Andruskiewitsch-Schneider program, which proposes to classify Hopf algebras $H$ with a fixed (co-)semisimple part $H_0$ in terms of  Nichols algebras over $H_0$. In  \cite{AS10} they classify all $H$ for $H_0$ the group ring of an abelian group, if no small prime divisors are present and hence the only Nichols algebras present are $u_q(\g)^+$, and Angiono completes this for general abelian groups and Nichols algebras. The results in our present paper can be viewed as taking this program further to a categorical setting, and to a braided setting.\\

We now return to results in our present article and assume the braided tensor category $\CC$ is the category $\Vect_\Gamma^Q$ of vector spaces graded by an abelian group $\Gamma$ with braiding and associator determined by a quadratic form $Q$ \cite{McL}. If the situation is ``sufficiently unrolled" in the sense of Definition \ref{def_sufficientlyUnrolled}, then the Nichols algebra is a Hopf algebra already uniquely characterized by its $\Gamma$-grading resp.~structure as an object in $\CC$. Moreover in this case Lemma \ref{lm_isRadfordNichols} gives a second argument for the existence of a splitting functor. Note that the first statement holds without assuming sufficiently unrolled, due to results in \cite{AKM15}, but the second statement may fail in rare examples, see Remark \ref{rem_lifting}.\\

In Section \ref{sec:RelativeDrinfeldCentersGiveLocalization} we return to the category $\UU$, which can by our results so far be realized as $\Nichols$-$\Nichols$-Yetter-Drinfeld modules over $\CC$, see Definition \ref{def_YD}. In the case $\CC=\Vect_\Gamma^Q$ and trivial associator we also give in Lemma \ref{lm_realizedQuantumGroup} a concrete realization of this category as representations of a generalized quantum group $U_q$, compare again \cite{LW1}. For example the Nichols algebra in Example \ref{exm_NicholsRank1} produces the small quantum group $U_q=u_q(\sl_2)$ for $\CC=\Vect_{\Z_p}$ for $q$ of odd order~$p$, respectively the quasi-Hopf algebra $U_q=\tilde{u}_q(\sl_2)$ for $\CC=\Vect_{\Z_{2p}}$ for $q$ of even order~$2p$ \cite{CGR20,GLO18}, respectively the unrolled quantum group  $u_q^H(\sl_2)$ for $\CC=\Vect_{\C}$ \cite{CGP}.\\

The theory developed at this point can now be used to prove characterization theorems in our setting and apply them to prove equivalences of braided tensor categories in the context of the logarithmic Kazhdan-Lusztig correspondence, discussed below. The following two statements are the algebraic side of our main results in this article:

\begin{introTheorem}[\ref{thm_CharacterizingUsingB}]
Suppose $\UU$ is a braided tensor category and $A\in \UU$ a commutative algebra  fulfilling \textup{Assumption \ref{assumption}}. Let $\CC=\Vect_\Gamma^Q=\Rep^{wt}(C)$ and let $X\in \CC$ with finite-dimensional Nichols algebra $\Nichols$ satisfying the mild degree conditions in \textup{Lemma \ref{lm_isNichols}}
and $U_q$ the corresponding quantum group in \textup{Lemma \ref{lm_realizedQuantumGroup}}. Assume that 
\begin{enumerate}[a)]
    \item $\UU\cong \Rep^{wt}(U_q)$ as abelian categories.
    \item $\UU_A^0\cong \Vect_\Gamma^Q$ as braided tensor categories.
    \item $\UU_A\cong \Rep(\Nichols)(\Vect_\Gamma^Q)$ as abelian categories, compatible with the respective $\CC$-module structures as described in the assumptions of \textup{Lemma \ref{lm_infiniteSplittingIsHopfAlgB}}.
\end{enumerate}
Then in fact these are equivalence of tensor categories and braided tensor categories.
\end{introTheorem}

The application of this result depends on the knowledge of $\UU_A$ as an abelian category. On the other hand the following result depends on the knowledge of a tensor functor to $\CC$, which means effectively that $\UU$ is realized a-priori as a (quasi-)Hopf algebra. It uses deep results by Takeuchi \cite{Tak79} and Skryabin \cite{Skry06} on relative Hopf modules.

\begin{introTheorem}[\ref{thm_characterizeQG}]
  Let $U\supset C$ be a Hopf algebra with $\UU=\Rep^{wt}(U)$ a braided tensor category with trivial M\"uger center and $A$ is a finite-dimensional commutative simple algebra object in $\UU$. Assume that 
  \begin{enumerate}[a)]
  \item $U=U_q$ as algebra.
  \item $A$ as an algebra in $\Rep^{wt}(U)$ is the dual Verma module $\bV_0^*$, see \textup{Definition \ref{def_Verma}}. (this condition can be relaxed to $A$ as an object, see \textup{Corollary \ref{cor_fixingtwist}}) 
  \item $\NicholsOf(X)\in\CC$ is sufficiently unrolled in the sense of \textup{Definition \ref{def_sufficientlyUnrolled}}.
  \item the category $\UU_A^0$ of local $A$-modules is equivalent as a braided tensor category to $\CC$, compatible with the inclusion $C\subset U$ in the sense that the following  commutes
  \begin{center}
$$
 \begin{tikzcd}[row sep=10ex, column sep=15ex]
   \Rep^{wt}(U) \arrow{r}{\induceA}  
   \arrow{rd}{\mathrm{res}_C}
   & 
   \Rep(A)(\Rep^{wt}(U))  \arrow{d}{A\text{-}\mathrm{inv}}\\
   & 
   \Rep^{wt}(C)
    \end{tikzcd}
$$
\end{center}
  
  \item the categories $\UU$ and $\UU_A$ are rigid.
  \end{enumerate}
  Then we have $B\cong \NicholsOf(X)\rtimes C$ and $U\cong U_q$ as Hopf algebras, and accordingly we have an equivalence of braided tensor categories $\UU=\UU_q$.
\end{introTheorem}

In Section \ref{sec_uprolling} we prove that extending the category $\UU$ by another commutative algebra~$S$ (typically a direct sum of invertible objects) commutes with the constructions above. In particular Theorem \ref{thm_uprolling} states that if $\UU=\YD{\Nichols}(\CC)$ and then the extensions  $\UU_S^0$ for suitable $S$ are again of the form $\YD{\tilde{\Nichols}}(\tilde{\CC})$. This means for example, that it is sufficient to prove a Kazhdan-Lusztig correspondence in the completely unrolled case $\Vect_{\C^n}^Q$.\\


The main aim of the previous results is to characterize the braided tensor categories of representations of the small quantum groups $u_q(\g)$ in terms of commutative algebras, but it also applies to basic super Lie algebras and other Nichols algebras in Heckenberger's list. Beyond $\CC=\Vect_\Gamma^Q$, interesting cases for our application are modules over non-abelian groups \cite{MS00, AFGV10, Len13, HV14}, parabolic situations \cite{CL17,AA20} and affine Lie algebras at positive integer level.

\begin{problem}
    Classify all finite-dimensional Nichols algebras in the modular tensor category associated to an affine Lie algebra at positive integer level. Examples can be produced from the parabolic situations, but it is in no way clear that this is exhausting.   
\end{problem}

\subsection{Correspondences of VOA and quantum group categories}

 A  \emph{vertex operator algebra} (VOA) $\mathcal{V}$ is a rigorous formulation of the chiral algebra of a two-dimensional conformal field theory. The algebra structure of $\mathcal V$ is given by fields 
 \[
Y(-, z): \; \mathcal V \otimes \mathcal V \rightarrow \mathcal V((z)), 
\qquad v \otimes w \mapsto Y(v, z)w
 \]
with $Y(v, z)w$ a formal Laurent series in $z$ with coefficient in the typically infinite dimensional $\mathbb C$-vectorspace $\mathcal V$. This structure is in a certain sense commutative and associative, e.g. the commutativity property is called locality and says that 
\[
(Y(v_1, z_1)Y(v_2, z_2)-Y(v_2, z_2)Y(v_1, z_1))w
\]
vanishes if multiplied by $(z-w)^N$ for sufficiently large $N \in \mathbb Z_{>0}$ depending on $v_1, v_2$.
There is a similar notion for modules and intertwining operators. Under suitable conditions intertwining operators realize a tensor category structure on categories of $\mathcal V$-modules \cite{HLZ0}-\cite{HLZ8}. This is  a braided tensor category.

Logarithmic conformal field theory first appeared around three decades ago \cite{Ka91, Gu93, GK96} and roughly a decade later it got translated to non-semisimple representation theory of VOAs. Much of the work was focused on accessible examples and in fact on their conjectural relations to quantum group categories \cite{FGST06, FGST06b, FGST07}. Since studying abelian and tensor categories of quantum groups is much clearer than for VOAs the putative connection of quantum groups and VOAs has been a huge guide in the development of the field of logarithmic VOAs (VOAs with non-semisimple representation categories). In particular, making the connection clear is a major open problem of the area for the last two decades.

Our theory is tailored to prove correspondences between categories of modules of vertex operator algebras (VOAs) and quantum groups. So we now explain how the previous discussion naturally appears in the VOA setting. 

\subsubsection{Commutative algebras via free field realizations}

Given a VOA $\V$, there are several standard constructions to construct a new VOA, e.g. cosets, orbifolds, cohomologies, tensor products of VOAs, VOA extensions. A basic class of VOAs are affine vertex algebras, which are associated to a Lie superalgebra with an invariant, consistent  and supersymmetric bilinear form and all VOAs that we are aware of  can be realized from affine VOAs via iterating standard constructions. For example $\W$-algebras are constructed via quantum Hamiltonian reduction, that is as certain cohomologies \cite{KW04}.

Except for the nicest classes of VOAs, such as rational and $C_2$-cofinite ones, the representation theory of VOAs is largely open at present. Another class of VOAs whose representation theory is under control and in fact rather trivial is free field algebras. In particular the category of modules of a simple Heisenberg VOA (the free boson of conformal field theory) is nothing but $\C$-graded vector spaces with some quadratic form, $\Vect_\C^Q$.
The point is that most VOAs allow for conformal embeddings into VOAs with much simpler representation theory and a free field realization is such a conformal embedding into a free field algebra. Standard examples are the Wakimoto free field realizations of affine VOAs. Here the free field algebra is the Heisenberg VOA associated to the Cartan subalgebra times a $\beta\gamma$-VOA associated to the root spaces \cite{FF90}. Another standard example are principal $\W$-algebras that embed conformally into the Heisenberg VOA associated to the Cartan subalgebra \cite{FF90b} and there are also free field realizations for general $\W$-algebras \cite{Gen}.
Sometimes one also has certain chains of conformal embeddings, e.g. the affine VOA of $\mathfrak{sl}_2$ at level $k$ embeds into the Virasoro vertex algebra at central charge $13-6(k+2)-6/(k+2)$ times a free field algebra, while the latter embeds itself in a certain free field algebra, see \cite{Ad1} for details. Note, that the Virasoro algebra is the principal $\W$-algebra of $\mathfrak{sl}_2$. In higher rank these chain of embeddings are even richer and quite useful to understand abelian representation theory, see \cite{ACG} for the case of $\mathfrak{sl}_3$.

We are interested in conformal embeddings $\W\subset \V$ of VOAs, since it translates to the following categorical statement. Let $\W$ be a VOA and $\UU$ a braided tensor category of $\W$-modules in the sense of Huang-Lepowsky-Zhang \cite{HLZ0}-\cite{HLZ8}. Let $\V$ be another VOA, that contains $\W$ as a subVOA and such that $\V$ is an object in $\CC$. Then $\V$ is identified with a commutative algebra $A$ in $\UU$ \cite{HKL} and the braided vertex tensor category of $\V$-modules $\CC$, that lie in $\UU$ is braided tensor equivalent to $\UU_A^0$, the category of local  $A$-modules in $\UU$ \cite{CKM}.

\subsubsection{Nichols algebras of screening operators}

Let $\V$ be a VOA, $M$ be a $\V$-module and $\mathcal Y$ an intertwining operator of type  $\binom{M}{M  \ \V}$, and $m\in M$ a suitable vector. Then the screening operator $\zem_m$  is a linear map from $\V$ to $M$  obtained as the zero-mode of $\mathcal Y(m, z)$ with the property that its kernel is automatically a subVOA of $\V$.  It turns out that in all the free field realizations that we mentioned $\W$ can be characterized as the joint kernel of a set of screening charges on $\V$. Screening charges in the case $M=\V$ generate actions of Lie algebras on $\V$ and for $M \neq \V$ it has been conjectured to give actions of quantum Borel parts and more generally of Nichols algebras $\Nichols$ in the category of $\V$-modules $\CC$. Unfortunately the connection of this Nichols algebra to the representation theory of $\W$ is open, see \cite{S11, S12, ST12, ST13, Len21, FL22} for work done in this direction. It has been conjectured there, with certain caveats in place

\begin{conjecture}
Let $\V$ be a VOA with a vertex tensor category $\CC$ of $\V$-modules $M_1,\ldots,M_n$. Fix suitable elements $m_1,\ldots,m_n$ therein and let $\zem_1, \dots, \zem_n$ be their screening operators.
\begin{enumerate}[a)]
    \item The screening operators generate a corresponding Nichols algebra $\Nichols$ in the braided tensor category $\CC$.
    \item The joint kernel of these screening operators $\W$ is a subVOA, and if $\Nichols$ is finite-dimensional then $\W$ has a representation theory similarly good to $\V$, for example if $\V$ is $C_2$-cofinite then $\W$ is also $C_2$-cofinite. 
    \item The braided tensor category of representations of $\W$ is equivalent to the category of $\Nichols$-Yetter-Drinfeld modules in $\CC$, see Definition \ref{def_YD}. Note that in cases without proper rigidity such as $\W_{p,q}$ \cite{GRW09} one expects a weaker relations between these categories.  
\end{enumerate}
\end{conjecture}
Assertion a) was proven in \cite{Len21} if $\V$ is a lattice vertex algebra, i.e.~a free-field theory. The results of the present article give a strong indication why  assertion c) might hold, in particular it proves assertion c) if there is already an equivalence of abelian categories or if the category of twisted modules is as expected.

\subsubsection{Logarithmic Kazhdan-Lusztig correspondences}

Our theory tells us that if we have a VOA $\W$ allowing for a suitable embedding into a simpler VOA $\V$, then in order to prove a braided tensor equivalence to some quantum group category one needs to do the following steps:
\begin{enumerate}
\item Understand a suitable abelian category $\UU$ of $\W$-modules so that one can establish its abelian equivalence to the quantum group category.
\item Show that $\UU$ is a vertex tensor category and that every object is rigid.
\item Show that $\UU_A$ is rigid and that   $\UU_A\cong \Rep(\Nichols)(\CC)$ for a suitable Nichols algebra $\Nichols$ of screenings in the category $\CC \cong \UU_A^0$ of $\V$-modules.
\end{enumerate}
Performing these steps is usually a tedious problem in VOA theory and we will now explain the state of the art in this direction.
For the existence of vertex tensor category one needs to establish good finiteness conditions. 
$C_2$-cofinite or lisse VOAs always allow for a vertex tensor category with finitely many inequivalent simple modules \cite{H09}, with the most prominent example the triplet VOAs $\cW(p)$ \cite{FGST06, FGST06b, FGST07, AM08, TW, McRY}. However most cases are not of such a nice finite type and a much larger class of categories is covered by 
a good criterion that emerged from \cite{CY, CJORY, McR3}. It says that the category of $C_1$-cofinite modules of a simple VOA is a vertex tensor category if it is of finite Jordan-H\"older length and if it is closed under the contragredient dual functor. Effectively this means one needs to classify $C_1$-cofinite modules and hope that the classification result implies the existence of vertex tensor category. This is unfortunately usually not an easy task and it has been successfully done in the following list of cases: The Virasoro and $N=1$ super Virasoro VOA at any central charge \cite{CJORY, CMOY}, affine VOAs at admissible levels and beyond \cite{CHY, CY},  the singlet VOAs $\cM(p)$ \cite{CMY, CMY4} and the affine VOA of $\mathfrak{gl}_{1|1}$ \cite{CMY5}. This list is currently rather short, but we note that $\W$-algebras in generality as well as affine vertex superalgebras are under current investigations.
Rigidity is usually quite unclear for vertex tensor categories and it has been proven in the above mentioned works with many different and new methods; and the case for most affine VOAs at admissible levels is \cite{C1}.
There is also a small list of vertex tensor categories of modules that are not $C_1$-cofinite, but it is inherited from certain $C_1$-cofinite categories. The most prominent example is the $\beta\gamma$-VOA \cite{AW} and this is in fact the first example of the $\mathcal B(p)$-algebras, which are certain subregular $\W$-algebras of type $\mathfrak{sl}_{p-1}$. Its Feign-Semikhatov dual algebras $\mathcal S(p)$ are certain principal $\W$-superalgebras of $\mathfrak{sl}_{p-1|1}$ \cite{CMY2}.

The third step, that is showing that $\UU_A$ is rigid and that   $\UU_A\cong \Rep(\Nichols)(\CC)$, should usually be possible once one knows the abelian structure and fusion products of $\UU$ well enough. 
The main application of the developed theory of this article is to perform this third step for the singlet VOAs $\cM(p)$ and then to derive the braided tensor equivalence to a small unrolled quantum group of $\mathfrak{sl}_2$. This is done in all detail in Section \ref{sec:KLsinglet}.

In the above mentioned examples, only the case of the triplet VOAs $\cW(p)$, the singlet VOAs $\cM(p)$, the $\mathcal B(p)$-algebras, the $\mathcal S(p)$-algebras and of the affine VOA of $\mathfrak{gl}_{1|1}$ the categories of $C_1$-cofinite modules are not semisimple. These are thus the candidates 
for which we want to establish correspondences to quantum groups. Note that the triplet correspondence has been solved in general in \cite{GN} and for $p=2$ in \cite{CLR}. The drawbacks are that \cite{GN} uses a classification result of tensor categories generated by a standard representation \cite{O08} and is thus quite limited to the triplet problem, while \cite{CLR} is our first approach which is too computational and the present work is the much more efficient theory. 

Our logarithmic Kazhdan-Lusztig correspondences are 
\begin{theorem}\label{ref_logKL}
The following are equivalent as braided tensor categories 
\begin{enumerate}
    \item Let $\mathcal O^T_{\mathcal M(p)}$ the category of weight modules of the singlet algebra of {\textup{\cite{CMY4}}} and $\Rep_{\mathrm{wt}}u_q^H(\mathfrak{sl}_2)$ 
    the category of weight modules of the small unrolled quantum group of $\mathfrak{sl}_2$ at $2p$-th root of unity of \textup{\cite{CGP}}. Then
    \begin{align*}
\mathcal O^T_{\mathcal M(p)} &\cong \Rep_{\mathrm{wt}}(u_q^ H(\sl_2))
\end{align*}
as braided tensor categories \textup{(Theorem \ref{thm_singletEquivalence} as well as Example \ref{ex_singlet})}.
\item The analogous result for triplet VOA and quasi Hopf modification of the small quantum group
\textup{(Remark \ref{rem_triplet} and Corollary \ref{cor_triplet})}
    $$\mathcal O^T_{\mathcal W(p)} \cong \Rep(u_q(\sl_2)).$$
    \item The Hopf algebra $U(S(p))$ is described in \eqref{usp}.
 The vertex tensor category of representations of the vertex superalgebra $\mathcal S(p)$ studied in \textup{\cite{CMY2}} is braided equivalent to the category of weight modules of $U(S(p))$ \textup{(Lemma \ref{lem_sp})}.
 \item 
The braided tensor category of weight representations of $V^k(\mathfrak{gl}_{1|1})$ is equivalent to $\YD{\tilde{\Nichols}}(\tilde{\CC})$ \textup{(Corollary \ref{cor_gl11})}. The realizing quantum group for this braided tensor category is $u_q^B({\mathfrak{gl}_{1|1}})$ constructed in \textup{Example \ref{exm_gl11QG}}.
\end{enumerate}
\end{theorem}
We note that the quantum group corresponding to the $\mathcal B(p)$-algebras is realized similarly and that will appear in \cite{ALSW23}.

Since we want to illustrate the usefulness of our main Theorems we prove the singlet correspondence twice, see Theorem \ref{thm_singletEquivalence} as well as Example \ref{ex_singlet}, as applications of each of our main characterization theorems. The necessary computations are done in Section \ref{sec:KLsinglet} and use the explicit knowledge of the singlet category obtained in  \cite{CMY, CMY4}.

We could in principal have continued in this way for the other stated correspondences and we sketch this for the triplet in Remark \ref{rem_triplet}. The computations would have been very similar to the ones done for the singlet algebra.
The other cases are however either simple current extensions of the singlet algebra (the triplet algebra) or simple current extensions of the singlet algebra times a rank one or two Heisenberg algebra. We thus work out how categories of Yetter-Drinfeld modules and their realizing Hopf algebras interplay with such extensions (we call this uprolling) in Section \ref{sec_uprolling} so that the remaining parts of the Theorem are just examples of this uprolling procedure.

\subsubsection{Outlook}

Our logarithmic Kazhdan-Lusztig correspondences, Theorem \ref{ref_logKL}, follow from our main structural Theorems together with the good understanding of the tensor category of the singlet algebra. 
In order to prove future logarithmic Kazhdan-Lusztig correspondences one thus needs to develop the representation theory of the involved VOAs well enough. We have a few series of examples in mind, see also Example \ref{ex_VOAfreefield}. There is one relatively straight forward example related to physics, namely certain categories of line operators in twists of three dimensional abelian $\mathcal N=4$ supersymmetric gauge theories are realized by braided tensor categories of VOAs that are extensions of several singlet and Heisenberg algebras \cite{BCDN} and so via uprolling the corresponding quantum groups can and will be constructed and their connection to the gauge theory should then also further be investigated. The other examples are ambitious programs that we are truely excited about.

Given the immense impact of the Kazhdan-Lusztig correspondence for Lie algebras one surely wants to have analogous results for Lie superalgebras. Except for $\mathfrak{osp}_{1|2n}$ the category of finite dimensional weight modules of a simple, classical, basic Lie superalgebra is not semisimple. It turns out that affine VOAs associated to type I Lie superalgebras allow for very nice realizations inside the affine VOA of its even subalgebra times a certain free fermion VOA \cite{QS07}. Our theory thus suggests to show that the Kazhdan-Lusztig category of the affine VOA of the type I Lie superalgera $\g$ at some level $k$ is a relative Drinfeld center of the tensor category of modules of some Nichols algebra in the  Kazhdan-Lusztig category of the even subalgebra times the free fermion category (which is just $\text{sVect}$). This translates via the original Kazhdan-Lusztig correspondence to Yetter-Drinfeld modules over the category of weight modules of the quantum group of the even subalgebra, see Example \ref{exm_parabolic} and Lemma \ref{lem_parabolic}. In order to prove now an equivalence of braided tensor categories of the Kazhdan-Lusztig category of the affine Lie algebra of $\g$ at a given level $k$ to the category of weight modules at corresponding $q$ \footnote{this should be $q = \exp\left(\frac{\pi i }{r^\vee (k+h^\vee)} \right)$ with $h^\vee$ the dual Coxeter number of $\g$ and $r^\vee$ the lacing number of the even subalgebra.}
one needs to show that all assumptions of Theorem \ref{thm_characterizeQG} hold. This is work in progress and we are optimistic that at least for generic $k$ and type I Lie superalgebras this can be proven and the rigidity seems to be the most difficult problem. 

The next series of examples should be called generalized Feigin-Tipunin algebras or just Feigin-Tipunin algebras. 
The triplet algebras are constructed as global sections of a certain VOA bundle over the flag manifold of $SL_2$ \cite{FT10, Su21, Su22}. By construction the algebra has the group $PSL_2$ acting by outer automorphisms and the singlet algebra is the orbifold that is invariant under the action of the maximal torus. In fact this construction generalizes vastly, though not much is known about the representation theory, see \cite{AM22, CRR} for some results in the $\mathfrak{sl}_3$-case. The generalization of \cite{FT10, Su21, Su22} is  as global sections of a certain VOA bundle over the flag manifold of $G$ for simply laced $G$. The VOA bundle is naturally associated to the principal nilpotent orbit in $G$ and there exists a natural generalization to any nilpotent orbit whose investigation is initiated in \cite{ACGY, CNS}. For the principal nilpotent element one expects a correspondence  of the orbifold by the maximal torus  to the small unrolled quantum group of the Lie algebra of $G$ at some root of unity \cite{CM17}. For non principal nilpotent orbits one expects a correspondence to quantum supergroups, but the dictionary yet needs to be developed. In \cite{CNS} the case of the zero-orbit (it is also called the affine case) for $SL_2$ is studied and the corresponding quantum group is a small unrolled quantum group of $\mathfrak{sl}_{2|1}$. Using our methods, the conjecture is proven there in the very first case. 
We hope that one can study the abelian representation categories of these Feigin-Tipunin type algebras using geometric representation theory in the spirit of \cite{AG, ABBGM}.

The third series of examples is going beyond the Kazhdan-Lusztig category. The category of integrable modules of a finite dimensional simple Lie algebra is just a very small subcategory of the category of the category $\mathcal O$ which itself is a subcategory of the category of weight modules. Similary the Kazhdan-Lusztig category of an affine VOA is a subcategory of the category $\mathcal O$ at level $k$ which is a subcategory of lower bounded weight modules which in turn is a subcategory of the category of weight modules, see \cite{ACK}. This last category is expected to be a very interesting and rich vertex tensor category. Given that it lacks many finiteness conditions this is not easy to establish and will appear for $\mathfrak{sl}_2$ at admissible level in \cite{C2}.
The affine VOA of $\mathfrak{sl}_2$ enjoys Gaiotto-Rap\v{c}\'ak triality \cite{GR} with the affine VOA of $\mathfrak{sl}_{2|1}$ and the $\mathcal N=2$ super Virasoro algebra (the princial $\W$-algebra of $\mathfrak{sl}_{2|1}$). This triality are large families of isomorphism of certain coset subVOAs, proven in \cite{CL22a, CL22b}. On the other hand the affine VOA of $\mathfrak{sl}_2$ allows for a realization inside the Virasoro algebra times a free field algebra \cite{Ad1}. The category of $C_1$-cofinite modules of the Virasoro algebra at generic central charge is the Deligne product of two Kazhdan-Lusztig categories of $\mathfrak{sl}_2$ and so this realization in fact suggest that the affine VOA of $\mathfrak{sl}_2$ at generic level has a braided tensor category of modules equivalent to modules of a quantum group of $\mathfrak{sl}_2 \oplus \mathfrak{sl}_{2|1}$. At admissible level one then would get a quotient category that one might like to call a partial semisimplification and we will explain this in a coming work. Higher rank is currently completely out of reach, but we note that the expectation from iterating trialities is that the category of weight modules of the affine VOA of $\mathfrak{sl}_n$ is related to a quantum group of $\mathfrak{sl}_n \oplus \mathfrak{sl}_{n|n-1}\oplus \mathfrak{sl}_{n-1|n-2} \oplus \mathfrak{sl}_{n-2|n-3} \oplus  \dots  \oplus \mathfrak{sl}_{2|1}$.

\subsection{Outline}

\subsection*{Acknowledgement}
Thanks to D. Nikshych for suggesting Corollary \ref{cor_Wittclass} and Problem \ref{prob_Wittclass} on the Witt class, to A. Riesen for suggesting splitting functors in the orbifold context are trivializations Example \ref{exm_localization_examples} b), to I. Angiono for suggesting \cite{AKM15} which is crucial for the Recognition Lemma \ref{lm_isNichols} and to M. Mombelli for suggesting the relative coend as a solution for Problem~\ref{prob_relativecoend}. 

S. Lentner thanks T. Creutzig and T. Gannon for hospitality at the University of Alberta, where much of this work has been done, and thanks for the support of the Feyodor Lynen fellowship of the Humboldt foundation.

\section{Extensions of a braided tensor category by a commutative algebra}\label{sec_extension}

In this Section $\cat$ is a braided tensor category,  $\one$ is the tensor unit and $\otimes$ the tensor bifunctor and $c_{\bullet, \bullet}$ and $\assoz_{\bullet, \bullet, \bullet}$ denote the braiding and associativity constraint.

\subsection{Definitions}

The usual references are \cite{KO, CKM}, the textbook is \cite{EGNO}.

\begin{definition}\label{def:alg}\textup{\cite[Def. 7.8.1]{EGNO}}
Let $\cat$ be a braided tensor category. An {\bf algebra} in $\cat$ is an object $A$ in $\cat$ together with a multiplication map
\[
m: A \otimes A \rightarrow A
\]
and a unit
\[
u : \one \rightarrow A
\]
such that the multiplication is associative and compatible with left and right multiplication, i.e.~the following three diagrams commute:
\begin{equation}
\begin{split}
\xymatrix{
\left(A \otimes A \right)\otimes A  \ar[rr]^(0.5){ \assoz^{-1}_{A, A, A} }\ar[d]_{m \otimes \Id_A}
&& A \otimes \left(A\otimes A\right)  \ar[d]^{\Id_A \otimes m} &
\\
A\otimes A \ar[rd]_{m}
&& A\otimes A \ar[ld]^{m} &
\\
& A &&
\\
} \\ \\
\xymatrix{
\one \otimes A \ar[r]^{\ell_A}\ar[d]_{u\otimes \Id_A} & A \ar[d]^{\Id_A} && A\otimes \one \ar[r]^{r_A}\ar[d]_{\Id_A \otimes u} & A \ar[d]^{\Id_A}
\\  
A\otimes A \ar[r]^m & A && A\otimes A \ar[r]^m & A
\\
}
\end{split}
\end{equation}
The algebra $A$ is called {\bf commutative} if the diagram 
\begin{equation}
\xymatrix{
A\otimes A \ar[rd]_{m}\ar[rr]^{c_{A, A}}
&& A\otimes A \ar[ld]^{m} &
\\
& A &&
\\
}
\end{equation}
commutes. The algebra $A$ is called {\bf haploid} if the dimension of $\Hom_\cat(\one, A)$ is one.
\end{definition}
\begin{definition}\label{def:RepA}
Let $\cat$ be a tensor category and $A$ an algebra in $\cat$. Then the category $\catA$ has objects $(X, m_X)$ with $X$ an object of $\cat$ and $m_X \in \Hom_\cat(A \otimes X, X)$ a multiplication morphism that is
\begin{enumerate}
\item Associative, i.e.~the diagram commutes:
\begin{equation}
\begin{split}
\xymatrix{
A \otimes (A\otimes X)    \ar[rr]^{  \assoz_{A, A, X} } \ar[d]_{\Id_A \otimes m_X } 
&&  (A \otimes A) \otimes X  \ar[d]^{m \otimes \Id_X}
\\
A\otimes X \ar[rd]_{m_X} && A\otimes X \ar[ld]^{m_X} \\
& X &
\\
} 
\end{split}
\end{equation}
\item Unital: The following composition is the identity on $X$
\[
X \xrightarrow{\ell_X} \one \otimes X \xrightarrow{u \otimes \Id_X} A \otimes X \xrightarrow{m_X} X.
\]
\end{enumerate}
Morphisms of $\catA$ are all $\cat$-morphisms $f \in \Hom_\cat(X, Y)$ such that the diagram commutes
\begin{equation}\label{morph}
\begin{split}
\xymatrix{
A \otimes  X    \ar[rr]^{  \Id_A \otimes f   } \ar[d]_{ m_X } 
&&  A \otimes Y  \ar[d]^{m_Y }
\\
X \ar[rr]^{f} && Y \\
} 
\end{split}
\end{equation}
\end{definition}
If $A$ is commutative then the category $\catA$ is tensor but usually not braided. 
For the tensor product one first defines the two morphisms
\begin{equation}
    \begin{split}
        \mu_1: A \otimes (X \otimes Y) &\xrightarrow{\assoz_{A, X, Y}} (A \otimes X) \otimes Y \xrightarrow{m_X \otimes \Id_Y} X \otimes Y \\
        \mu_2: A \otimes (X \otimes Y) &\xrightarrow{\assoz_{A, X, Y}} (A \otimes X) \otimes Y \xrightarrow{c_{A, X} \otimes \Id_Y} (X \otimes A) \otimes Y \\ 
        &\xrightarrow{\assoz^{-1}_{X, A, Y}} X \otimes (A \otimes Y) \xrightarrow{\Id_X \otimes m_Y} X \otimes Y
    \end{split}
\end{equation}
and then $X \otimes_A Y$ as the coequalizer of $\mu_1$ and $\mu_2$.
The projection of $X \otimes Y$ onto $X \otimes_A Y$ will be denoted by $\eta$.
Let $f$ be a morphism from $X \otimes Y$ to $X' \otimes Y'$. It induces a morphism from $X \otimes_A Y$ to $X' \otimes_A Y'$ if $\eta \circ f \circ \mu_1 = \eta\circ f \circ \mu_2$. In this case we denote the resulting morphism by $\overline{f}$. It is a morphism in $\catA$ if \eqref{morph} holds.

\begin{definition}
Let $\cat$ be a braided tensor category and $A$ a commutative algebra in $\cat$. Then the category $\catAloc\subset\catA$ of local modules is the  subcategory whose objects are local with respect to $A$, i.e.~the following diagram commutes
\begin{equation}
\begin{split}
\xymatrix{
A \otimes  X    \ar[rd]_{  m_X }  \ar[rr]^{ c_{X, A} \circ  c_{A, X}} && A\otimes X \ar[ld]^{m_X} \\
& X&  \\
} 
\end{split}
\end{equation}
\end{definition}
\begin{theorem} \textup{\cite{KO, CKM}}
 Let $\cat$ be a braided tensor category and $A$ a commutative algebra in $\cat$.   Then
 \begin{enumerate}
 \item For any pair of objects $(X, m_X), (Y, m_Y)$ in $\catA$ there is a unique morphism in $\cat$, $m_{X \otimes_A Y}: A \otimes (X \otimes_A Y) \rightarrow X \otimes_A Y$ such that the diagram
 \begin{equation}
     \begin{split}
         \xymatrix{
A \otimes (X \otimes Y) \ar[d]^{\Id_A \otimes \eta}\ar[rr]^{\mu_i} && X \otimes Y \ar[d]^{\eta} \\
A \otimes (X \otimes_A Y) \ar[rr]^{m_{X \otimes_A Y}} && X \otimes_A Y\\
         }
     \end{split}
 \end{equation}
 commutes for $i=1,2$. The pair $(X \otimes_A Y, m_{X \otimes_AY})$ is an object in $\catA$ \textup{\cite[Prop.2.38]{CKM}}.
     \item $\catA$ is a tensor category \textup{\cite[Thm. 2.53]{CKM}}.
     \item $\catAloc$ is a braided tensor subcategory of $\catA$ and for $(X, m_X), (Y, m_Y)$ objects in $\catAloc$ the braiding is denoted by $\overline{c}_{X, Y}$ and is characterized by the commutative diagram
     \begin{equation}
         \begin{split}
          \xymatrix{   
          X \otimes Y \ar[rr]^{c_{X, Y}} \ar[d]^\eta && Y\otimes X \ar[d]^\eta \\
             X \otimes_A Y \ar[rr]^{\overline{c}_{X, Y}} && Y \otimes_A X\\
             }
          \end{split}
     \end{equation}
     \textup{\cite[Thm. 2.55]{CKM}}
 \end{enumerate}
\end{theorem}
\begin{remark}\label{rem:halfbraiding}
    The proof of Theorem 2.55 of \textup{\cite{CKM}} only uses that $X$ is in $\catAloc$. This means that in fact the following more general statement has been proven there:

    For $(X, m_X)$ an object in $\catAloc$  there is a natural family of isomorphisms denoted by $\overline{c}_{X, \bullet}$, such that for  $(Y, m_Y)$ an object in $\catA$, $\overline{c}_{X, Y}$
    is characterized by the commutative diagram
     \begin{equation}
         \begin{split}
          \xymatrix{   
          X \otimes Y \ar[rr]^{c_{X, Y}} \ar[d]^\eta && Y\otimes X \ar[d]^\eta \\
             X \otimes_A Y \ar[rr]^{\overline{c}_{X, Y}} && Y \otimes_A X\\
             }
          \end{split}
     \end{equation}
\end{remark}

\begin{definition}
The {\bf forgetful functor} 
$$\forgetA:\;\cat \leftarrow \catA$$
sends $(M,m_M)\in \catA$ to the underlying object $M\in \cat$. The forgetful functor has a left adjoint, the {\bf induction functor}
$$\induceA:\;\cat\rightarrow \catA$$
$$X\mapsto (A\otimes X, m \otimes \Id_X)$$
If $A$ is a commutative algebra, then the induction functor is a tensor functor to $(\catA,\otimes_A)$ \textup{\cite{KO, CKM}}, and the forgetful functor is an {\bf oplax} tensor functor, since we have a canonical surjective map
$$\forgetA(M)\otimes\forgetA(N)\to \forgetA(M\otimes_A N)$$
\end{definition}

We summarize the main categories and functors in the following diagram 
\begin{center}
\begin{tcolorbox}[colback=white, width=0.38\linewidth]
$$
 \begin{tikzcd}[row sep=10ex, column sep=15ex]
   \mathbf{\UU} \arrow{r}{\induceA}  
   & \arrow[shift left=2]{l}{\forgetA}   \mathbf{\BB} = & \hspace{-4.75cm} \catA \\
   &\arrow[hookrightarrow]{u}{\iota}\arrow[below, dotted, shift left=2]{ul}{} 
   \mathbf{\CC} = & \hspace{-4.75cm}\catAloc
    \end{tikzcd}
$$
\end{tcolorbox}
\end{center}

To summarize the notation appearing in this diagram: $\UU$ is a braided tensor category, $A$ is a commutative algebra therein, $\BB$ is the tensor category of $A$-modules in $\UU$, $\CC$ is the  braided tensor category of local modules therein, with $\iota$ the embedding functor. The functors $\induceA$ and the inclusion of $\CC$ in $\BB$ are tensor functors, while $\forgetA$ and the composition $\forgetA\circ\iota$ are oplax tensor functors.

\subsection{Tensoring with trivial \texorpdfstring{$A$}{A}-modules}\label{sec_TrivialActionOnOneSide}

For an ordinary algebra $A$, we can tensor an $A$-module $N$ with a vector space $X$ and obtain new $A$-modules $N\otimes X$ and $X\otimes N$ where $A$ only acts on one tensor factor. We have a similar notion for algebras in braided tensor categories:

\begin{lemma}
Given an object $(N, m_N)$ in $\catA$ and an object $X$ in $\cat$, then the objects $N\otimes X$ and $X\otimes N$ in $\cat$ become objects in $\catA$ with the following actions:
\begin{align}
m_{N \otimes X}: A \otimes (N \otimes X) &\xrightarrow{\assoz_{A, N, X}} (A \otimes N) \otimes X \xrightarrow{m_N \otimes \Id_X} N \otimes X\label{NX}\\
m_{X \otimes N}: A \otimes (X \otimes N) &\xrightarrow{\assoz_{A, X, N}} (A \otimes X) \otimes N 
\xrightarrow{c_{A, X} \otimes \Id_N}(X \otimes A) \otimes N \notag\\
&\xrightarrow{\assoz^{-1}_{X, A, N}} X \otimes (A \otimes N) \xrightarrow{\Id_X \otimes m_N} X \otimes N. \notag
\end{align}
In particular there is also a right induction functor $\induceA^R(X) := (m_{X \otimes A}, X \otimes A)$.

Moreover,the braiding provides a natural family of $A$-module isomorphism 
$$c_{N,X}:\;(N\otimes X,m_{N \otimes X}) \stackrel{\sim}{\longrightarrow} (X\otimes N,m_{X \otimes N}) $$
\end{lemma}
\begin{proof}
The associativity of multiplication follows from commutativity of
\begin{equation}
\begin{split}
\xymatrix{
A \otimes (A\otimes (X \otimes N))    \ar[r]^{  c_{A, X} } \ar[d]_{} & A \otimes (X \otimes (A \otimes N)) \ar[r]^{\ \ m_N}\ar[d]^{c_{A, X}} & A \otimes (X \otimes N) \ar[d]^{c_{A, X}} \\  
  (A \otimes A) \otimes (X \otimes N) \ar[r]^{c_{A \otimes A, X}} \ar[d]^{m} & X \otimes (A \otimes (A \otimes N)) \ar[r]^{m_N}\ar[d]^{m}& X \otimes (A \otimes N) \ar[d]^{\ \ m_N} \\ 
A\otimes (X \otimes N) \ar[r]_{c_{A,X}} & X \otimes (A\otimes N) \ar[r]^{\ \ m_N} & X \otimes N
\\
} 
\end{split}
\end{equation}
Here we omit the associativity and identity morphisms on factors. The upper right diagram commutes due to the hexagon axiom and naturality of associativity; the upper left diagram commutes due to naturality of associativity, the lower right diagram is naturality of braiding and associativity and the lower right one is naturality of associativity together with associativity of the multiplication of $N$.  
The unit property of $m_{X \otimes N}$ follows from the one of $m_N$.
Let $f \in \Hom_{\catA}(N, N')$  and $g \in \Hom_\cat(X, X')$, then $f \otimes g \in \Hom_{\catA}(N \otimes X, N'\otimes X')$ and $g \otimes f \in \Hom_{\catA}(X \otimes N, X'\otimes N')$, since \eqref{morph} clearly holds using that it holds for $f$ together with naturality of braiding and associativity.\\

For the last claim, we consider the commutative diagram
\[
\xymatrix{
A \otimes (N \otimes X) \ar[r]\ar[d]_{\Id_A \otimes c_{N, X}} &(A \otimes N) \otimes X \ar[d]^{c_{A \otimes N, X}} \ar[r]^{\ \ \ m_N} & N \otimes X \ar[d]^{c_{N, X}} \\
A \otimes (X \otimes N) \ar[r]^{c_{A, X}} & X \otimes (A \otimes N) \ar[r]^{\ \ \ m_N} & X \otimes N \\
}
\]
The left diagram commutes by the hexagon and the right one by naturality of the braiding (both plus naturality of associativity). 
Hence $c_{N, X}$ is a morphism in $\catA$. In fact, since $c_{\bullet,  X}$ is a natural family of isomorphisms in $\cat$, it is so in $\catA$ as well.
\end{proof}

Finally, we want to compare these modules with induced modules:
\begin{lemma}\label{lm_trivialActionToInducedModule}
Assume now that $A$ is commutative. 
Given an object $(N, m_N)$ in $\catA$ and an object $X$ in $\cat$, there are natural families of $A$-module isomorphisms
    $$X \otimes N \cong A \otimes_A(X \otimes N) \cong \induceA(X) \otimes_A N$$
    $$N \otimes X  \cong (N \otimes X) \otimes_A A \cong N \otimes_A \induceA(X)$$
\end{lemma}
\begin{proof}
For this we first have the commutative diagrams for $i = 1, 2$
\begin{equation}\label{diag1}
\begin{split}
\xymatrix{
A \otimes (A \otimes (X \otimes N)) \ar[rr]^{\Id_A \otimes \assoz_{A, X, N}}\ar[d]^{\mu_i} && A \otimes ((A \otimes X) \otimes N) \ar[d]^{\mu_i} \\
A \otimes (X \otimes N) \ar[rr]^{\assoz_{A, X, N}} \ar[d]^\eta && (A \otimes X) \otimes N \ar[d]^\eta \\
A \otimes_A (X \otimes N) \ar@{.>}[rr]^{\overline{\assoz}_{A, X, N}} &&\induceA(X) \otimes_A N \\
}
\end{split}
\end{equation}
Here the two upper horizontal arrows are just associativity and the upper diagram commutes trivially (i.e.~by coherence of associator) for $i=1$ and for $i=2$ due to the hexagon axiom. 
The lower dashed arrow then exists by the universal property of the coequalizer.
It remains to show that this is a morphism in $\catA$. This follows from 
\[
\xymatrix{
& A \otimes (X \otimes N) \ar[rr]^{\assoz_{A, X, N}} \ar[dd]^(.3){\eta} && (A \otimes X) \otimes N \ar[dd]^\eta \\
A \otimes (A \otimes (X \otimes N)) \ar[ur]^{\mu_i}\ar[dd]^{\Id_A \otimes \eta} \ar[rr]^(.37){\Id_A \otimes \assoz_{A, X, N}}  && \ar[ur]^{\mu_i} A \otimes ((A \otimes X) \otimes N)\ar[dd]^(.25){\Id_A \otimes \eta} \\
& A \otimes_A (X \otimes N) \ar[rr]^(.4){\overline{\assoz}_{A, X, N}} && \induceA(X) \otimes_A N \\
A \otimes (A \otimes_A (X \otimes N)) \ar[ur]^{m}\ar[rr]^{\Id_A \otimes \overline{\assoz}_{A, X, N}}  && \ar[ur]_{m} A \otimes (\induceA(X) \otimes_A N) \\
}
\]
Here we omitted the subscripts in the action of $A$ on modules, that is we just wrote $m$  for $m_{A \otimes_A (X \otimes N)}$ and  $m_{\induceA(X) \otimes_A N}$.
The left and right faces of the diagram commute due to Proposition 2.38 of \cite{CKM}. All others, except for the bottom, commute by \eqref{diag1} and so the bottom diagram commutes as well, but this is exactly the statement that $\overline{\assoz}_{A, X, N}$ is a morphism in $\catA$.
The same argument applies to ${\assoz}_{A, X, N}^{-1}$ inducing a morphism $\overline{\assoz}_{A, X, N}^{-1}$ in $\catA$, which is clearly the inverse of  $\overline{\assoz}_{A, X, N}$.

Hence $X \otimes N \cong A \otimes_A(X \otimes N) \cong \induceA(X) \otimes_A N$.
It is easy to verify that naturality of associator implies that $\overline{\assoz}_{A, X, \bullet}$ is a natural family of isomorphisms. 

Similarly we have the commutative diagrams for $i = 1, 2$
\[
\xymatrix{
A \otimes ((N \otimes X)\otimes A) \ar[rr]\ar[d]^{\mu_i} && A \otimes (N  \otimes (X \otimes A)) \ar[d]^{\mu_i} \\
(N \otimes X) \otimes A \ar[rr] \ar[d]^\eta && N \otimes (X \otimes A) \ar[d]^\eta \\
(N \otimes X) \otimes_A A \ar@{.>}[rr]^{\overline{\assoz}_{N, X, A}} && N \otimes_A \induceA^R(X)  \\
}
\]
Here the two upper horizontal arrows are just associativity and the upper diagram commutes trivially (i.e.~by coherence of associator) for $i=1$ and for $i=2$ due to the hexagon axiom. 
The lower dashed arrow then exists by the universal property of the coequalizer. That this is an ismorphism in  $\catA$ follows in the same way as the previous case. 
In other words, $N \otimes X  \cong (N \otimes X) \otimes_A A \cong N \otimes_A \induceA(X)$ and as before it in fact gives a natural family of isomorphisms. 
\end{proof}

In summary, we have the following commutative diagram
\begin{equation}\label{b-diag} 
\begin{split}
\xymatrix{
N \otimes (A \otimes X) \ar[rr]^\eta \ar[d]^{\Id_N \otimes c_{A, X}} && N \otimes_A \induceA(X) \ar[d]^{\Id_N \otimes_A c_{A, X}} \\
N \otimes (X \otimes A) \ar[rr]^\eta  \ar[d]^{\assoz_{N, X, A}} && N \otimes_A \induceA^R(X) \ar[d]^{\overline{\assoz}_{N, X, A}}\\
 (N \otimes X) \otimes A \ar[rr]^\eta  \ar[d]^{c_{N, X} \otimes \Id_A} && N \otimes X \ar[d]^{c_{N, X} \otimes_A \Id_A} \\
(X \otimes N) \otimes A \ar[rr]^\eta  && X \otimes N  \\
  \ar[u]_{c_{A,X \otimes N}} A \otimes (X \otimes N) \ar[rr]^\eta \ar[d]^{\assoz_{A, X, N}} && X \otimes N \ar[d]^{\overline{\assoz}_{A, X, N}} \ar[u]_{\overline{c}_{A, X \otimes N}} \\
  (A \otimes X) \otimes N \ar[rr]^\eta && \induceA(X) \otimes_A N \\
}
\end{split}
\end{equation}
The second diagram from the bottom is a special case of Remark \ref{rem:halfbraiding}. As a composition of natural families of isomorphism  the family $b_{\bullet, X}$ defined by
\begin{equation}
\begin{split}
    b_{N, X}: &N \otimes_A \induceA(X) \xrightarrow{\Id_N \otimes_A c_{A, X}}
    N \otimes_A \induceA^R(X) \xrightarrow{\overline{\assoz}_{N, X, A}} 
    N \otimes X \xrightarrow{c_{N, X} \otimes_A \Id_A} \\
    &X \otimes N \xrightarrow{\overline{c}^{-1}_{A,X \otimes N} }
    X \otimes N \xrightarrow{\overline{\assoz}_{A, X, N}} \induceA(X) \otimes_A N
\end{split}
\end{equation}
is natural. 
We  note that the braid diagram of the vertical morphism of the left-hand side
of \eqref{b-diag}, with morphisms to be read from bottom to top, is

\begin{center}
        \begin{grform}\label{b-braid}
           \begin{scope}[scale = 1]
                \def\y{-0.3}\def\yn{-0.3+1}
                \draw (0 , \y) node {$N$};
                \draw (1 , \y) node {$A$};
                \draw (2 , \y) node {$X$};

                \def\y{0}\def\yn{1}
                \vLine{0}{\y}{0}{\yn}{black};
                \vLine{1}{\y}{2}{\yn}{black};
                \vLineO{2}{\y}{1}{\yn}{black};

                \def\y{1}\def\yn{2}
                \vLine{0}{\y}{1}{\yn}{black};
                \vLineO{1}{\y}{0}{\yn}{black};
                \vLine{2}{\y}{2}{\yn}{black};

                \def\y{2}\def\yn{3}
                \vLine{0}{\y}{0}{\yn}{black};
                \vLine{2}{\y}{1}{\yn}{black};
                \vLineO{1}{\y}{2}{\yn}{black};

                \def\y{3}\def\yn{4}
                \vLine{1}{\y}{0}{\yn}{black};
                \vLineO{0}{\y}{1}{\yn}{black};
                \vLine{2}{\y}{2}{\yn}{black};

                \def\y{4.3}\def\yn{4.3}
                \draw (0 , \y) node {$A$};
                \draw (1 , \y) node {$X$};
                \draw (2 , \y) node {$N$};
            \end{scope}  
        \end{grform}
                \quad=\quad
        \begin{grform}
           \begin{scope}[scale = 1]
                \def\y{-0.3}\def\yn{-0.3+1}
                \draw (0 , \y) node {$N$};
                \draw (1 , \y) node {$A$};
                \draw (2 , \y) node {$X$};

                \def\y{0}\def\yn{1}
                \vLine{1}{\y}{0}{\yn}{black};
                \vLineO{0}{\y}{1}{\yn}{black};
                \vLine{2}{\y}{2}{\yn}{black};

                \def\y{1}\def\yn{2}
                \vLine{0}{\y}{0}{\yn}{black};
                \vLine{1}{\y}{2}{\yn}{black};
                \vLineO{2}{\y}{1}{\yn}{black};

                \def\y{2.3}\def\yn{3.3}
                \draw (0 , \y) node {$A$};
                \draw (1 , \y) node {$X$};
                \draw (2 , \y) node {$N$};
            \end{scope}  
        \end{grform}
\end{center}
so that $b_{N, X}$ can be characterized by 
\begin{equation}\label{bNX}
\begin{split}
    \xymatrix{
N \otimes (A \otimes X) \ar[rr]^\eta  \ar[d]^{\assoz_{N, A, X}} && N \otimes_A \induceA(X) \ar[ddddd]^{b_{N, X}} \\ 
(N \otimes A) \otimes X \ar[d]^{c^{-1}_{A, N}\otimes \Id_A} &&  \\
(A \otimes N) \otimes X \ar[d]^{\assoz^{-1}_{A, N, X}} && \\
A \otimes (N \otimes X) \ar[d]^{\Id_A \otimes c_{N, X}} && \\
A \otimes (X \otimes N) \ar[d]^{\assoz_{A, X, N}} \\
\ar[rr]^\eta (A \otimes X) \otimes N && \induceA(X)\otimes_A N  \\
    }
    \end{split}
\end{equation}

\subsection{First examples}

The easiest type of examples are the {simple current extensions} \cite{CKL}, where  $A$ is a semisimple algebra decomposing into invertible ($1$-dimensional) objects in $\UU$. A particular class of examples, which will be fundamental in what follows, is the following:

\begin{example}[Lattice]\label{exm_lattice}
Let $\Gamma$ be a vector space with inner product $(-,-)$ and consider the tensor category $\UU=\Vect_\Gamma$ with braiding $\sigma(\lambda,\mu)=e^{\pi\i(\lambda,\mu)}$. Let $\Lambda\subset \Gamma$ be an even integral lattice, then its span is a commutative  algebra $A$ in $\Vect_\Gamma$ (with multiplication modified by a $2$-cocycle $u(\lambda,\mu)$ with $u(\lambda,\mu)/u(\mu,\lambda)=\sigma(\lambda,\mu)$). In this case we have 
\begin{align*}
\UU&=\Vect_\Gamma\\
\UU_A&=\Vect_{\Gamma/\Lambda}\\
\UU_A^0 &=\Vect_{\Lambda^*/\Lambda}
\end{align*}
with the dual lattice $\Lambda^*=\{\lambda\in \Gamma\,|\, \forall_{\alpha\in \Lambda}(\lambda,\alpha)\in\Z\}$. This example is associated to a lattice vertex algebra, see Section \ref{sec_exampleVOA}.
\end{example}

Another common source of examples are $G$-crossed extensions resp. $G$-orbifolds 

\begin{example}[$G$-Orbifolds]\label{exm_orbifold}
    Let $G$ be a finite group acting on a modular tensor category $\CC$, then we can ask for a $G$-graded extension 
    $$\BB=\bigoplus_{g\in G} \CC_g,\qquad \CC_e=\CC$$
    with a $G$-action and a $G$-crossed braiding. If $\CC$ is finite and semisimple with $G$-action, then by \textup{\cite{ENOM09}} a $G$-crossed extension of $\CC$ exists and is unique up to cohomological obstructions and choices. The $G$-equivariantization is then again a modular tensor category 
    $$\UU=\BB /\!\!/ G$$

    Conversely there exists a commutative algebra $A$ in $\UU$, s.t. $\UU_A$ is $\BB$ and $\UU_A^0$ is $\CC$. This example is associated to a $G$-orbifold $\V^G$ of a vertex algebra $\V$ with category of representations $\CC$, see Section \ref{sec_exampleVOA} and e.g. \textup{\cite{McR2, McR1}}.
\end{example}

\subsection{Examples from quantum groups}\label{sec_exmQG}

The main example we have in mind are the small quantum groups $u_q(\g)$ resp. their quasi-versions for $q$ a root of unity of even order. We now summarize this situation, proofs in the context of relative Drinfeld centers are given in Section \ref{sec:RelativeDrinfeldCentersGiveLocalization}.

For $\g$ a finite-dimensional semisimple complex Lie algebra of rank $n$ with a choice of simple roots $\alpha_1,\ldots,\alpha_n$, and a Killing form $(-,-)$ normalized to $(\alpha,\alpha)=2$ for short roots. Let $q$ be a root of unity of order $\ell$, then the small quantum group $u_q(\g)$ is a a finite-dimensional non-semisimple Hopf algebra \textup{\cite{Lusz93}}. It contains Hopf subalgebras called Borel parts $u_q(\g)^{\leq}, u_q(\g)^{\geq}$. 
Inside the Borels are (coideal) subalgebras $u_q(\g)^{+}, u_q(\g)^{-}$ and a common Hopf subalgebra called the Cartan part $u_q(\g)^0$ which is a quotient of $\C[\Z^n]$ with generators $K_{\alpha_i}$. We denote the dual group by $\Gamma$, so $\Rep(u_q(\g)^0)=\Vect_\Gamma$.

We assume $\ell$ larger than $(\alpha,\alpha)/2$ for all roots $\alpha$, otherwise the quantum group degenerates. For $\ell$ prime to $2$ (resp. for $\g=G_2$ prime to $3$) the category $\UU=\Rep(u_q(\g))$ is a modular tensor category. For $2|\ell$ there is a quasi-Hopf algebra with a nontrivial associator that produces again a modular tensor category \textup{\cite{GLO18}}, as discussed in Section \ref{sec:RelativeDrinfeldCentersGiveLocalization}.

For a weight $\lambda\in \h^*$ and corresponding $1$-dimensional representation of the Cartan part by $K_\alpha\mapsto q^{(\alpha,\lambda)}$, we extend the representation trivially to a $1$-dimensional representation $\chi_\lambda: u_q(\g)^\geq\to \C$ and define the \emph{Verma module} $\bV_\lambda$ as the induced module
$$\bV_\lambda=u_q(\g)\otimes_{u_q(\g)^\geq} \C_{\chi_\lambda}$$
In particular the Verma module $\bV_0$ covering the tensor unit is automatically a commutative coalgebra in $\UU$, understood as a coalgebra quotient of $u_q(\g)$, and can be identified with $u_q(\g)^-$, and the dual Verma module $\bV_0^* = \text{Hom}_{\mathbb C}(\bV_0, \mathbb C)$ is a commutative algebra \cite{GLR23}. The local modules consist of the modules with trivial action of $u_q(\g)^+$ and coincides as a braided tensor category with the category of $\Gamma$-graded vector spaces $\Vect_{{\Gamma}}^Q$ with quadratic form $Q(\lambda)=q^{(\lambda,\lambda)/2}$. In this case our general situation applies as follows
    $$\begin{tikzcd}[row sep=10ex, column sep=15ex]
   {\Rep(u_q(\g))} \arrow{r}{\text{restriction}}  
   & \arrow[shift left=2]{l}{\text{induction}}   \Rep(u_q(\g)^{\geq 0})&\hspace{-3cm} \quad\textcolor{darkgreen}{\cong}\;\UU_A,\otimes_A \newline
   \\
   &\arrow[hookrightarrow]{u}{}\arrow[below]{ul}{\text{dual Verma module}}
   \Rep(u_q^0)\cong\Vect_{{\Gamma}}^Q&\hspace{-2.8cm}\textcolor{darkgreen}{\cong}\;\UU^0_A,\otimes_A
    \end{tikzcd}$$\\
Here, $\textcolor{darkgreen}{\cong}$ is a nontrivial equivalence of tensor categories from $\UU_A$ (or relative $u_q(\g)$-$u_q(\g)^{-}$-Hopf modules) to modules over $u_q(\g)^{\geq 0}$, which we will discuss in Section \ref{sec_quasiquantum}. While the tensor product on the former is $\otimes_A$, the tensor product of the latter is simple $\otimes_\C$ with appropriate action.
Under this identification, the induction functor $\induceA$ is actually the restriction of a $u_q(\g)$-module to a $u_q(\g)^\geq$-module, for example $\bV_0$ is mapped to a large indecomposable with $1$-dimensional composition factors. All simple $u_q(\g)^\geq$-modules are $1$-dimensional and local modules $\C_\lambda$ with trivial action of $u_q(\g)^+$. On the other hand, the functor $\forgetA$ becomes the Verma induction functor composed with duality from $u_q(\g)^\geq$ to $u_q(\g)$, and the composition  becomes the functor sending a weight $\C_\lambda$ to the dual Verma module $\bV_\lambda^*= \text{Hom}_{\mathbb C}(\bV_\lambda, \mathbb C)$.   \\

\begin{remark}[Exotic Borel] One can ask whether the choice $A=\bV_0^*$ is the only choice of a commutative algebra over $U=u_q(\g)$. Indeed via Theorem \ref{thm_skry} we can describe such an algebra in terms of the right coidal subalgebra of coinvariants $U^{coin(A)}$ in our standard example $u_q(\g)^+$. If such a coideal subalgebra contains the entire Cartan part $u_q(\g)^0$ then the main result of \textup{\cite{HS13}} is that it is necessarily $u_q(\g)^+$ or some Weyl reflection thereof. In the physical application this means that if the category of local modules (resp. the free field realization) should be $\Vect_\Gamma$, then this is the only possibility.

However there are more examples of right coideal subalgebras, in particular \textup{\cite{LV21}} studies those which retain the properties that all simple modules are $1$-dimensional, and it can be checked which of these give rise to (co)commutative (co)algebras. For example the Example~2.12 in loc.~cit.~leads to the cocommutative coalgebra $A^*=\C[K,K^{-1}]$ with $u_q(\sl_2)^0$ acting by left multiplication and $E,F$ by some difference operators. In this case, the category $\UU_A$ is the category of representations of the quantum Weyl algebra, and it would be an interesting physical problem to search for such free-field realizations.  
\end{remark}

\subsection{Examples from vertex algebras}\label{sec_exampleVOA}

Associated to every vertex algebra $\V$ fulfilling good finiteness conditions is a braided tensor category $\Rep(\V),\otimes_\V$ by \cite{HLZ0}-\cite{HLZ8}, see Section 3 of \cite{CY} for the criteria that imply existence of vertex tensor categories in the examples relevant for this work. Suppose that $\W\subset \V$ is a vertex subalgebra, then $\V$ can be regarded as a $\W$-module and is a commutative algebra in the category $\Rep(\W)$ \cite{HKL}. Altogether we have the situation

    $$\begin{tikzcd}[row sep=10ex, column sep=15ex]
   {\Rep(\W)} \arrow{r}{\text{induction}}  
   & \arrow[shift left=2]{l}{\text{restriction}}   \mathrm{twRep}_\W(\V)&\hspace{-2.5cm}=\UU_A,\otimes_A \\
   &\arrow[hookrightarrow]{u}{}\arrow[below]{ul}{\mathrm{restriction}}
   \Rep(\V)&\hspace{-2.5cm}=\UU_A^0, \otimes_A
    \end{tikzcd}$$\\

Here the situation is much closer to modules over a commutative ring and in some sense opposite to the Hopf algebra case: The functors $\induceA,\;\forgetA$ are now indeed induction and restriction of vertex algebra modules and the restriction functor is oplax, because $M\otimes_\V N$ is defined as a universal object with respect to $\V$-intertwining operators, and since these give (a subset of) $\W$-intertwining operators for the restricted modules, we have a canonical morphism 

$$\forgetA(M)\otimes_\W\forgetA(N)\to \forgetA(M\otimes_\V N)$$

The tensor category of {twisted modules} $\mathrm{twRep}_\W(\V)$ consists of $\W$-modules $M$ together with a suitable morphism of $\W$-modules $\V\otimes_\W M\to M$, but not all of these are $\V$-modules due to multivalued functions in the intertwining operators, violating locality, which is directly visible on the nontrivial braiding $c_{\V,M}$. The local modules $M$ are precisely the modules with single-valued intertwining operators, hence these are precisely the $\V$-modules.\\


\begin{example}\label{ex_VOAfreefield}
Many vertex algebras $\cW$ allow for realizations as subalgebras of well-known vertex algebras $\V$, these are often characterized as a kernel of some set of screening charges acting on the well-known vertex algebra $\V$. 
The probably best known example is the universal Virasoro algebra at any central charge as a subalgebra of the Heisenberg VOA. Here are some examples that are relevant for us, see also Section \ref{catMA}.
\begin{enumerate}
    \item Let $p \in \Z_{\geq 2}$ and let $L = \sqrt{p}A_1$ with $A_1= \sqrt{2}\Z$ the root lattice of $\sl_2$. Then the singlet algebra $\mathcal M(p)$ is a subalgebra of the rank one Heisenberg VOA $\pi$ and the triplet algebra is a subalgebra of the lattice VOA $V_L$. We review these quickly in 
    Section \ref{catMA} and early references are \textup{\cite{Ka91, FGST06, AM07}}. These algebras are first examples of large families:
    \begin{enumerate}
    \item Let $p \in \Z_{\geq 2}$. The triplet algebras are the Feigin-Tipunin algebras associated to $\sl_2$ and the general Feigin-Tipunin algebra $\mathcal{FT}(\g, p)$ \textup{\cite{FT10, Su21}} is constructed as global sections of some vertex algebra bundle over the flag variety of a compact simply-laced Lie group $G$ with Lie algebra $\g$. It is a subalgebra of the lattice VOA $V_{\sqrt{p}Q}$. It is expected that the representation category of $\mathcal{FT}(\g, p)$ is equivalent to a quasi Hopf modification of the small quantum group of $\g$ at $2p$-th root of unity, however almost nothing is known about the representation theory of  $\mathcal{FT}(\g, p)$.
\item   The Feigin-Tipunin algebras have subalgebras $\mathcal{FT}^0(\g, p)$ that are themselves subalgebras of the Heisenberg VOA of rank the rank of $\g$ \textup{\cite{CM17}}. Their representation category is expected to be equivalent to the one of the unrolled quantum group of $\g$ at  $2p$-th root of unity.
  \item The Feigin-Tipunin algebras are large extensions of principal $W$-algebras of $\g$ at level $k = - h^\vee + \frac{1}{p}$. In particular in the case $\g = \sl_2$ the principal W-algebra is just the Virasoro algebra and at these levels it has central charge $1- 6(p-1)^2/p$.
 \end{enumerate} 
  A similar (but more complicated) construction works for $W$-algebras corresponding to any nilpotent element, including affine VOAs. This has been first explored in the $\sl_2$-case \textup{\cite{ACGY}} and is currently developed in generality \textup{\cite{CNS}}. Their representation categories should be related to the ones of certain quantum supergroups, e.g. we expect in the $\sl_2$-case a correspondence with a variant of the quantum supergroup of $\sl_{2|1}$ at $2p$-th root of unity.

    \item The affine vertex algebra of $\mathfrak{gl}_{1|1}$ at non-degenerate level is a subalgebra of a pair of free fermions times a rank two Heisenberg VOA \textup{\cite{SS06}}. This  is the first one of two large series of further examples:
    \begin{enumerate}
        \item The universal principal W-superalgebra of $\mathfrak{sl_{n|1}}$ and $\mathfrak{osp}_{2|2n}$ can be embedded into a pair of free fermions times a Heisenberg VOA of rank $n+1$ \textup{\cite{CGN21}}. This indicates a potential connection to the quantum groups of $\mathfrak{sl_{n|1}}$ and $\mathfrak{osp}_{2|2n}$. Note that these W-superalgebras are Feigin-Frenkel dual \textup{\cite{CGN21, CL22a, CL22b}} to the subregular $W$-algebras of $\mathfrak{sl}_n$ and $\mathfrak{so}_{2n+1}$, so that in particular the latter allow for a similar free field realization.
        \item Let $\g$ be a Lie superalgebra of type I that allows for a non-degenerate bilinear form. Let $\g_0$ be its even subalgebra and let $d_{\text{odd}}$ be the dimension of the odd subspace of $\g$. Then there is an embedding of the universal affine VOA of $\g$ at level $k$, $V^k(\g)$ into $V^\ell(\g_0)$ times $d_{\text{odd}}$ free fermions. Here the level $\ell$ depends on $k$ in a certain way, see \textup{\cite{QS07}}. At least for generic level $k$, the Kazhdan-Lusztig category of the universal affine VOA of $\g$ at level $k$ is expected to be braided equivalent to the one of the quantum supergroup of $\g$ at $q= \text{exp}(\frac{2\pi i }{2 (k+h^\vee})$ with $h^\vee$ the dual Coxeter number of $\g$.
\end{enumerate}
\end{enumerate}
\end{example}


\section{The category \texorpdfstring{$\UU$}{U} as relative Drinfeld center of \texorpdfstring{$\UU_A$}{U\_A} }\label{sec_Schauenburg}

\subsection{Relative Drinfeld center}

Recall that the center $\mathcal Z(\UU)$ of a tensor category $\UU$ consists of objects $(Z, \gamma)$ with $Z$ an object in $\UU$ and $\gamma$  a natural family of isomorphisms $\{ \gamma_X : X \otimes Z \rightarrow Z \otimes X \; |\;  X \ \in \ \text{Obj}(\UU)\}$, satisfying the hexagon diagram
\begin{equation*}
\xymatrix{
&   \left(X \otimes Y \right) \otimes Z \ar[rr]^{\phantom{aa} \gamma_{X \otimes Y}  \phantom{aa}} && Z \otimes \left(X \otimes Y\right)  \ar[rd]^{\phantom{aa} \assoz_{Z, X, Y}} &  \\  
X \otimes \left( Y  \otimes Z\right) \ar[ru]^{\assoz_{X, Y, Z}} \ar[rd]_{\Id_X \otimes \gamma_{Y} }
&&&&  \left(Z \otimes X\right)\otimes Y
\\
&   X\otimes \left(Z \otimes Y\right)   \ar[rr]_{\phantom{aa} \assoz_{X, Z, Y}\phantom{aa} } &&  \left(X \otimes Z\right) \otimes Y\ar[ru]_{\ \ \gamma_{X}}\otimes \Id_Y &\\
}
\end{equation*}
A morphism $f$ is a morphism in $\UU$, s.t. the diagram 
\begin{equation}\label{centermor}
    \begin{split}
        \xymatrix{
  X \otimes Z \ar[rr]^{\gamma_X}\ar[dd]^{\text{Id}_X \otimes f} &&  Z \otimes X \ar[dd]^{f \otimes \text{Id}_X }   \\ \\
   X \otimes Z' \ar[rr]^{\gamma'_X} &&  Z' \otimes X  \\
   }
    \end{split}
\end{equation}
commutes for any object $X$.
The center is a tensor category with tensor product
\[
(Z, \gamma) \otimes_{\mathcal Z(\UU)} (Z', \gamma') = (Z \otimes Z', \gamma \cdot \gamma')
\]
with 
\begin{equation}\label{centerproduct}
    \begin{split}
(\gamma \cdot \gamma')_X:\ &X \otimes (Z \otimes Z') \xrightarrow{\assoz_{X, Z, Z'}} 
(X \otimes Z) \otimes Z' \xrightarrow{\gamma_X \otimes \Id_{Z'}} (Z \otimes X) \otimes Z' \\
&\xrightarrow{\assoz^{-1}_{Z, X, Z'}} Z \otimes (X \otimes Z') \xrightarrow{\Id_Z \otimes \gamma'_X}
Z \otimes (Z' \otimes X) \xrightarrow{\assoz_{Z, Z', X}} (Z \otimes Z') \otimes X 
\end{split}
\end{equation}

If $\UU$ is braided then there are two fully faithful functors
\begin{align*}
\mathcal F : \UU \rightarrow \mathcal Z(\UU), \qquad &X \mapsto (X, c_{\bullet, X}) \\
\bar{\mathcal{F}} : \bar\UU \rightarrow \mathcal Z(\UU), \qquad &X \mapsto (X, c^{-1}_{X, \bullet}).
\end{align*}
Here $\bar \UU$ denotes the category $\UU$ with reversed braiding. The M\"uger center of $\UU$ is the subcategory consisting of objects $X$ that satisfy $c_{X, Y} \circ c_{Y, X}= \text{Id}_{X \otimes Y}$ for all objects $Y$.

\begin{lemma}\label{Lem:Muger}
Let $\UU$ be a rigid, braided, locally finite  abelian category with trivial M\"{u}ger center and the property that for any object $X$, there exists a projective object $P_X$ and an injective object $I_X$ with
a surjection $\pi_X : P_X \twoheadrightarrow X$ and an embedding $\iota_X:  X \hookrightarrow I_X$. Then there is a fully faithful braided functor $\overline{\UU} \boxtimes \UU \to \mathcal{Z}(\UU)$.
\end{lemma}
\begin{proof}
$\mathcal F, \bar{\mathcal{F}}$ combine into a functor (that we denote by $\sF$)
\[
\sF : \UU \boxtimes \overline\UU \rightarrow \mathcal Z(\UU).
\]
Let $f : Z \rightarrow Z'$ be a morphism in $\UU$ and assume that $f$ is also a morphism from $(Z, c_{\bullet, Z})$ to $ (Z', c^{-1}_{Z, \bullet})$, that is 
$$\begin{tikzcd}[row sep=10ex, column sep=15ex]
  X \otimes Z \arrow{r}{c_{X, Z}}\arrow{d}{\text{Id}_X \otimes f} &  Z \otimes X \arrow[]{d}{f \otimes \text{Id}_X }   \\
   X \otimes Z' \arrow{r}{c^{-1}_{Z, X}} &  Z' \otimes X 
    \end{tikzcd}$$
commutes for any object $X$. By naturality of the braiding it follows that
$$\begin{tikzcd}[row sep=10ex, column sep=15ex]
  X \otimes Z \arrow{r}{\text{Id}_X \otimes f}  \arrow{d}{\text{Id}_X \otimes f} &  X \otimes Z' \arrow[]{d}{c_{X, Z'} }   \\
   X \otimes Z' \arrow{r}{c^{-1}_{Z, X}} &  Z' \otimes X 
    \end{tikzcd}$$
commutes, 
i.e.~$f$ factors through the tensor identity as $\UU$ has trivial M\"uger center. It follows that 
\begin{equation}\label{eqZU}
\text{Hom}_{\mathcal Z(\UU)}\left(\sF(X \boxtimes \mathbf 1),  \sF(\mathbf 1 \boxtimes Y)\right)
\cong  \text{Hom}_{\UU}\left(X, \mathbf 1\right) \otimes_{\mathbb C} \text{Hom}_{\bar\UU}\left(\mathbf 1, Y\right) 
\end{equation}
and thus
\begin{equation}\label{ffcent2}
\begin{split}
\text{Hom}_{\mathcal Z(\UU)}\left(\sF(X \boxtimes Y),  \mathbf 1 \right) &\cong  
\text{Hom}_{\mathcal Z(\UU)}\left(\sF(X \boxtimes \mathbf 1) \otimes \sF(\mathbf 1 \boxtimes Y),  \mathbf 1 \right)\\
&\cong  \text{Hom}_{\mathcal Z(\UU)}\left(\sF(X \boxtimes \mathbf 1),  \sF(\mathbf 1 \boxtimes Y)^* \right)\\
&\cong  \text{Hom}_{\UU}\left(X, \mathbf 1\right) \otimes_{\mathbb C} \text{Hom}_{\bar\UU}\left(\mathbf 1, Y^*\right) \\
&\cong  \text{Hom}_{\UU}\left(X, \bf 1\right) \otimes_{\mathbb C} \text{Hom}_{\bar \UU}\left(Y, \bf 1\right) \\
&\cong  \text{Hom}_{\UU \boxtimes \bar\UU}\left(X \boxtimes Y, \bf 1 \boxtimes \bf 1\right).
\end{split}
\end{equation}
Here we used that the dual satisfies $\sF(\mathbf 1 \boxtimes Y)^* \cong \sF(\mathbf 1 \boxtimes Y^*)$.
Hence for any objects $X, X', Y, Y'$ we have
\begin{equation}\label{ffcent}
\begin{split}
\text{Hom}_{\mathcal Z(\UU)}\left(\sF(X \boxtimes Y),  \sF(X' \boxtimes Y') \right) &\cong  \text{Hom}_{\mathcal Z(\UU)}\left(\sF(X' \boxtimes Y')^* \otimes \sF(X \boxtimes Y), \mathbf 1 \right)\\ 
&\cong   \text{Hom}_{\mathcal Z(\UU)}\left(\sF((X'^* \otimes X) \boxtimes \sF(Y'^* \otimes Y)), \mathbf 1 \right)\\ 
&\cong  \text{Hom}_{\UU \boxtimes \bar\UU}\left((X'^* \otimes X) \boxtimes (Y'^* \otimes Y), \bf 1 \boxtimes \bf 1\right) \\
&\cong  \text{Hom}_{\UU}\left(X'^* \otimes X , \bf 1\right) \otimes_{\mathbb C} \text{Hom}_{\bar \UU}\left(Y'^* \otimes Y, \bf 1\right) \\
&\cong  \text{Hom}_{\UU}\left(X,  X'\right) \otimes_{\mathbb C} \text{Hom}_{\bar \UU}\left(Y, Y'\right) \\
&\cong  \text{Hom}_{\UU \boxtimes \bar\UU}\left( X \boxtimes Y, X' \boxtimes Y'\right).
\end{split}
\end{equation}
By \cite[Prop.1.11.2]{EGNO} and Lemma \ref{From Nikolaus for Thomas} the assumptions of Lemma \ref{ff} are satisfied and so $\sF$ is fully faithful. 
\end{proof}

The following notion seems to have first appeared in \cite{Maj91, GNN09} and has been developed in \cite{Lau20,LW1}

\begin{definition}\label{def_relcenter} 
Given a monoidal category $\BB$ and braided category $(\CC,c)$, we say $\BB$ is $\CC$-central if there is a  faithful braided monoidal functor $\mathcal{F}: \overline{\CC} \to \mathcal{Z}(\BB)$. Given a $\CC$-central category, the $\CC$-relative monoidal center $\mathcal{Z}_{\CC}(\BB)$ is the centralizer of $\cF(\CC)$ in $\mathcal{Z}(\BB)$. That is,
\begin{equation}\label{relc} \mathcal{Z_C(B)}:=\mathrm{Cent}_{\mathcal{Z(B)}}(\mathcal{F(C)})=\{ X \in \mathcal{Z(B)} \; | \; c_{\mathcal{F}(C),X} \circ c_{X, \mathcal{F}(C)} = \mathrm{Id}_{X \otimes \mathcal{F}(C)} \; \forall C \in \CC \; \}  \end{equation}
\end{definition}

\subsection{Main theorem of this section} 


We now prove our main result of this section. It implies that for a category $\UU$ together with the existence of a commutative algebra $A$ satisfying some properties completely fixes $\UU$ as the relative Drinfeld center of $\UU_A$. 
The properties needed are
\begin{assumption}\label{assumption}
Assume that  $\UU$ is a rigid, braided, locally finite  abelian category with trivial M\"{u}ger center,  the property that for any object $X$, there exists a projective object $P_X$ and an injective object $I_X$ with
a surjection $\pi_X : P_X \twoheadrightarrow X$ and an embedding $\iota_X:  X \hookrightarrow I_X$.  and that $A \in \cU$ is a commutative and haploid algebra object, such that $\cU_A$ is  rigid.
\end{assumption}

We proceed in several steps. First, suppose we have a fully faithful braided tensor functor 
$F:\UU\to \mathcal{W}$ and a commutative algebra $A\in\UU$, then the image $F(A)$ is a commutative algebra in $\mathcal{W}$ and the functor gives rise to a fully faithful tensor functor resp. braided tensor functor
\begin{align*}
\UU_A&\to \mathcal{W}_{F(A)} \\
\UU_A^0&\to \mathcal{W}_{F(A)}^0 
\end{align*}

\begin{lemma}\label{loceq}
Let $\UU, A, \UU_A$ be as in Assumption \ref{assumption}.

Then the canonical functor $\overline{\cU} \boxtimes \cU \to \cZ(\UU)$ of Lemma \ref{Lem:Muger} induces a fully faithful tensor functor resp. braided tensor  functor
\begin{align*}
\overline{\cU}_{A} \boxtimes \cU & \to \cZ(\cU)_{\bar{\cF}(A)} \\ \overline{\cU}_{A}^0 \boxtimes \cU & \to \cZ(\cU)_{\bar{\cF}(A)}^0. 
\end{align*}
\end{lemma}
\begin{proof}
First we note that $\cU_A$ is locally finite as it is a subcategory of $\cU$ and it is also abelian \cite[Thm. 2.9]{CKM}.
The  property that for any object $X$, there exists a projective object $P_X$ and an injective object $I_X$ with
a surjection $\pi_X : P_X \twoheadrightarrow X$ and an embedding $\iota_X:  X \hookrightarrow I_X$ holds in
$\UU_A$:  Induction maps projectives to projectives \cite[Rmk. 2.64]{CKM} and so by Frobenius reciprocity we can take $P_X$ as $\induceA(P_{\forgetA(X)})$.
By Rigidity  we can take $P_{X^*}^*$ as $I_X$. 

By Lemma \ref{Lem:Muger} the functor
\[  \overline{\cU} \boxtimes \cU \to \cZ(\UU), \quad N \boxtimes X \mapsto \big(N\otimes X, c_{N,\bullet}^{-1} \cdot   c^{\phantom{}}_{\bullet, X}  \big) \]
is fully faithful and braided.

Next, let $X$ be an object in $\cat$. Then $\text{ind}_{\bar{\cF}(A)}(\cF(X))$ is an object in $\cZ(\UU)_{\bar{\cF}(A)}$.
By Frobenius reciprocity, \eqref{ffcent} and since $A$ is haploid
\begin{equation}\nonumber 
    \begin{split}
        \Hom_{\cZ(\UU)_{\bar{\cF}(A)}}(\text{ind}_{\bar{\cF}(A)}(\cF(X)), \text{ind}_{\bar{\cF}(A)}(\cF(Y))) &\cong
        \Hom_{\cZ(\UU)}(\cF(X), \bar{\cF}(A) \otimes \cF(Y))) \\
        &\cong \Hom_{\overline{\cat} \boxtimes \cat}(\one \boxtimes X, A \boxtimes Y) \\ 
        &\cong \Hom_{\cat}(X, Y)
    \end{split}
\end{equation}
so by Lemma \ref{ff} there  is also fully faithful functor from $\cat$ to $\cZ(\UU)_{\bar{\cF}(A)}$.

Next we have that $(N\otimes X, m_{N \otimes X})$ is an object in $\overline{\cat}_A$, see \eqref{NX}.
$m_{N \otimes X}$ is also a morphism in $\cZ(\cU)$, since \eqref{centermor} holds by naturality of inverse-braiding, so the pair
\[
\big(N\otimes X, c_{N,\bullet}^{-1} \cdot   c^{\phantom{}}_{\bullet, X}, m_{N \otimes X}\big)
\]
is also in $\cZ(\cU)_{\bar{\cF}(A)}$. A morphism $f \boxtimes g$ is just mapped to $f \otimes g$.
Let us denote the functor by $\cG: \overline{\cat}_A \boxtimes \cat \to \cZ(\UU)_{\bar{\cF}(A)}$.
\begin{equation}\nonumber 
    \begin{split}
        \Hom_{\cZ(\UU)_{\bar{\cF}(A)}}(\cG(A \otimes X), \cG(N \otimes \one)) 
        &\cong \Hom_{\cZ(\UU)_{\bar{\cF}(A)}}(\text{ind}_{\bar{\cF}(A)}(\cF(X)), ((N, c_{N, \bullet}^{-1}), m_N)) \\
        &\cong \Hom_{\cZ(\UU)}(\cF(X), (N, c_{N, \bullet}^{-1})) \\
        &\cong \Hom_{\cZ(\UU)}(\cF(X), \bar{\cF}(N)) \\
        &\cong \Hom_{\cZ(\UU)}(\sF(X \boxtimes \one), \sF(\one \boxtimes N)) \\
        &\cong \Hom_{\cat}(X, \one) \otimes_{\CC} \Hom_{\overline\cat}(\one, N) \\
        &\cong \Hom_{\cat}(X, \one) \otimes_{\CC} \Hom_{\overline\cat_A}((A, m), (N, m_N))
    \end{split}
\end{equation}
Here we first used the definition of $\cG$, then Frobenius reciprocity, then the definition of $\bar\cF$ and $\sF$, then \eqref{eqZU} and finally Frobenius reciprocity again. 
By the same argument as \eqref{ffcent2} and \eqref{ffcent} it follows that $\cG$ is fully faithful.

Suppose now that $(N, m_N) \in \overline{\cU}_{A}^0$ so that 
\begin{equation} \label{local1} m_N \circ (c^{-1}_{A,N} \circ c^{-1}_{N,A})= m_N\end{equation}
We need to show that 
\[
\big(\big(N\otimes X, c_{N,\bullet}^{-1} \cdot   c^{\phantom{}}_{\bullet, X}, m_{N \otimes X}\big)
 \in \cZ(\cU)_{\cF(A)}^0.
 \]
 Denote by $M_{A, N \otimes X}$
 the monodromy, i.e.~the double braiding, of $((A, c^{-1}_{N, \bullet}), m)$ with $
\big(\big(N\otimes X, c_{N,\bullet}^{-1} \cdot   c^{\phantom{}}_{\bullet, X}, m_{N \otimes X}\big)$.
It is given by
\begin{equation} \nonumber 
    \begin{split}
M_{A, N \otimes X}: A \otimes (N\otimes X) &\xrightarrow{\assoz_{A, N, X}} (A \otimes N) \otimes X
\xrightarrow{c^{-1}_{N, A}} (N \otimes A) \otimes X 
\xrightarrow{\assoz^{-1}_{N, A, X}} N \otimes (A \otimes X) \\
&\xrightarrow{c_{A, X}} N \otimes (X \otimes A) 
\xrightarrow{\assoz_{N, X, A}} (N \otimes X) \otimes A
\xrightarrow{c^{-1}_{A, N \otimes X}} A \otimes (N \otimes X)
    \end{split}
\end{equation}
By the hexagon axiom one has that 
\begin{align*}
N \otimes (A \otimes X) 
\xrightarrow{c_{A, X}} N \otimes (X \otimes A) 
&\xrightarrow{\assoz_{N, X, A}} (N \otimes X) \otimes A
\xrightarrow{c^{-1}_{A, N \otimes X}} A \otimes (N \otimes X)
\\
\intertext{is equal to} 
N \otimes (A \otimes X) 
\xrightarrow{\assoz_{N, A, X}} (N \otimes A) \otimes X 
&\xrightarrow{c^{-1}_{A, N}} (A \otimes N) \otimes X
\xrightarrow{\assoz^{-1}_{A, N, X}} A \otimes (N \otimes X)
\\
\intertext{and so}
M_{A, N \otimes X}: A \otimes (N\otimes X) &\xrightarrow{\assoz_{A, N, X}} (A \otimes N) \otimes X
\xrightarrow{c^{-1}_{N, A}} (N \otimes A) \otimes X \\
&\xrightarrow{c^{-1}_{A, N}} (A \otimes N) \otimes X
\xrightarrow{\assoz^{-1}_{A, N, X}} A \otimes (N \otimes X)
\end{align*}
But then clearly by \eqref{local1} we have \[
m_{N \otimes X} \circ M_{A, N \otimes X} = m_{N \otimes X},
\]
so we have an induced fully faithful functor
\[ \overline{\cU}_{A}^0 \boxtimes \cU \to \cZ(\UU)_{\bar{\cF}(A)}^0. \]
\end{proof}

Recall Remark \ref{rem:halfbraiding}, that is 
   for $(X, m_X)$ an object in $\catAloc$  there is a natural family of isomorphisms denoted by $\overline{c}_{X, \bullet}$, such that for  $(Y, m_Y)$ an object in $\catA$, $\overline{c}_{X, Y}$
    is characterized by the commutative diagram
     \begin{equation}
         \begin{split}
          \xymatrix{   
          X \otimes Y \ar[rr]^{c_{X, Y}} \ar[d]^\eta && Y\otimes X \ar[d]^\eta \\
             X \otimes_A Y \ar[rr]^{\overline{c}_{X, Y}} && Y \otimes_A X\\
             }
          \end{split}
     \end{equation}
    As a consequence we obtain a  braided tensor functor  
     which we call $\SchauenburgBar$:
    \begin{definition}\label{def_cent2} 
     We denote the tensor functor by 
     \[
    \SchauenburgBar:\,\bar{\mathcal{U}}_A^0 \to \mathcal{Z}(\UU_A),\qquad
    N \mapsto (N,\overline{c}_{N, \bullet}^{-1}).
    \]
\end{definition}
Also recall that there is a family of natural isomorphism $b_{\bullet, X}$ with $b_{N, X}$ characterized by \eqref{bNX}, which is using \eqref{centerproduct}
\begin{equation}
\begin{split}
    \xymatrix{
N \otimes (A \otimes X) \ar[rr]^{(c^{-1}_{A, \bullet} \cdot\, c^{\phantom{} }_{\bullet,X})_N}  \ar[d]^\eta && (A \otimes X) \otimes N  \ar[d]^\eta \\ 
N \otimes_A \induceA(X) \ar[rr]^{b_{N, X}} && \induceA(X)\otimes_A N  \\
    }
    \end{split}
\end{equation}
    As a consequence we obtain a  braided tensor functor  
     which we call $\Schauenburg$:
\begin{definition}
There is a braided tensor functor
\[
\Schauenburg:\,\UU \to \mathcal{Z}(\UU_A), \qquad 
X \mapsto (\induceA(X), b_{\bullet, X}).
\]
\end{definition}

By \cite[Corollary 4.5]{Sch} there is an equivalence of categories 
\[
\cZ(\UU)_{\bar{\cF}(A)}^0 \cong \mathcal{Z}(\mathcal{U}_A)
\]
mapping $(X, \gamma)$ to $(X, \overline{\gamma})$ with $\overline{\gamma}_Y$ the image of $\gamma_Y$ under $\eta$. 
In particular the functors $\Schauenburg,\SchauenburgBar$ are the compositions of our embeddings of Lemma \ref{loceq} with Schauenburg's equivalence. As a composition of fully faithful functors, $\Schauenburg,\SchauenburgBar$ are fully faithful.
We now show that the images of these two functors have the following centralizing property.
\begin{theorem}\label{relcent}
Let $\UU, A, \UU_A$ be as in Assumption \ref{assumption}.

Then there is a fully faithful functor to the relative Drinfeld center
\[ \Schauenburg:  \mathcal{U} \hookrightarrow \mathcal{Z}_{\mathcal{U}_A^0}(\mathcal{U}_A),  \qquad
X \mapsto (\induceA(X), b_{\bullet, X}). \]
\end{theorem}
\begin{proof}

We just discussed that both $\Schauenburg$ and $\SchauenburgBar$ are fully faithful functors and 
and we only need to show that the image of $\Schauenburg$ is in fact in $\mathcal{Z}_{\mathcal{U}_A^0}(\mathcal{U}_A)$, i.e.~that for any pair $N$ in $\UU_A^0$  and $X$ in $\UU$ the monodromy of their images $(N, \overline{c}^{-1}_{N, \bullet})$ and $(X, b_{\bullet, X})$ in $\mathcal Z(\UU_A)$ is trivial. This monodromy is  $\overline{M}_{N, X} := \overline{c}^{-1}_{N, A \otimes X} \circ  b_{N, X}$ and it is characterized by the commutative diagram
\[
\xymatrix{
N \otimes (A \otimes X) \ar[rr]^{(c^{-1}_{A, \bullet} \cdot\, c^{\phantom{} }_{\bullet,X})_N}  \ar[d]^\eta && (A \otimes X) \otimes N  \ar[d]^\eta  
\ar[rr]^{c^{-1}_{N, A \otimes X}} && N \otimes (A \otimes X) \ar[d]^\eta \\
N \otimes_A \induceA(X) \ar[rr]^{b_{N, X}} && \induceA(X)\otimes_A N \ar[rr]^{\overline{c}^{-1}_{N, A \otimes X}} && N \otimes_A \induceA(X) \\
}
\]
The braid diagram of $c^{-1}_{N, A \otimes X} \circ  c^{-1}_{A, \bullet} \cdot\, (c^{\phantom{} }_{\bullet,X})_N$ is easily obtained from \eqref{b-braid}

\begin{center}
        \begin{grform}
           \begin{scope}[scale = 1]
                \def\y{-0.3}\def\yn{-0.3+1}
                \draw (0 , \y) node {$N$};
                \draw (1 , \y) node {$A$};
                \draw (2 , \y) node {$X$};

                \def\y{0}\def\yn{1}
                \vLine{0}{\y}{0}{\yn}{black};
                \vLine{1}{\y}{2}{\yn}{black};
                \vLineO{2}{\y}{1}{\yn}{black};

                \def\y{1}\def\yn{2}
                \vLine{0}{\y}{1}{\yn}{black};
                \vLineO{1}{\y}{0}{\yn}{black};
                \vLine{2}{\y}{2}{\yn}{black};

                \def\y{2}\def\yn{3}
                \vLine{0}{\y}{0}{\yn}{black};
                \vLine{2}{\y}{1}{\yn}{black};
                \vLineO{1}{\y}{2}{\yn}{black};

                \def\y{3}\def\yn{4}
                \vLine{1}{\y}{0}{\yn}{black};
                \vLineO{0}{\y}{1}{\yn}{black};
                \vLine{2}{\y}{2}{\yn}{black};

                \def\y{4}\def\yn{5}
                \vLine{0}{\y}{0}{\yn}{black};
                \vLine{2}{\y}{1}{\yn}{black};
                \vLineO{1}{\y}{2}{\yn}{black};

                \def\y{5}\def\yn{6}
                \vLine{1}{\y}{0}{\yn}{black};
                \vLineO{0}{\y}{1}{\yn}{black};
                \vLine{2}{\y}{2}{\yn}{black};

                \def\y{6.3}\def\yn{7.3}
                \draw (0 , \y) node {$N$};
                \draw (1 , \y) node {$A$};
                \draw (2 , \y) node {$X$};
            \end{scope}  
        \end{grform}
                \quad=\quad
        \begin{grform}
           \begin{scope}[scale = 1]
                \def\y{-0.3}\def\yn{-0.3+1}
                \draw (0 , \y) node {$N$};
                \draw (1 , \y) node {$A$};
                \draw (2 , \y) node {$X$};

                \def\y{0}\def\yn{1}
                \vLine{1}{\y}{0}{\yn}{black};
                \vLineO{0}{\y}{1}{\yn}{black};
                \vLine{2}{\y}{2}{\yn}{black};

                \def\y{1}\def\yn{2}
                \vLine{1}{\y}{0}{\yn}{black};
                \vLineO{0}{\y}{1}{\yn}{black};
                \vLine{2}{\y}{2}{\yn}{black};

                \def\y{2.3}\def\yn{3.3}
                \draw (0 , \y) node {$N$};
                \draw (1 , \y) node {$A$};
                \draw (2 , \y) node {$X$};
            \end{scope}  
        \end{grform}
\end{center}

Hence the monodromy is characterized by the commutative diagram
\[
\xymatrix{
N \otimes (A \otimes X) \ar[r]^{\ \assoz_{N, A, X}}  \ar[d]^\eta & (N \otimes A) \otimes X \ar[d]^{\eta} \ar[rr]^{c_{A, N} \circ c_{N, A}} &&(N \otimes A) \otimes X \ar[d]^\eta & \ar[l]_{\ \assoz_{N, A, X}}  N \otimes (A \otimes X) \ar[d]^\eta \\
N \otimes_A \induceA(X) \ar[r]^{\ \ \overline{\assoz}_{N, A, X}} & N \otimes X \ar[rr]^{ \Id_{N \otimes X} } &&N \otimes X & \ar[l]_{\overline{\assoz}_{N, A, X}}  N \otimes_A \induceA(X) \\
}
\]
We clearly have  
$\eta \circ \assoz_{N,A, X} \circ \mu_1 = \eta \circ \assoz_{N,A, X} \circ \mu_2$, since the canonical projection $(N \otimes A) \otimes X$ onto $N \otimes X\cong (A \otimes_A N) \otimes X$ is given by $(m_N \circ c_{N, A})\otimes \Id_X$. Hence $\assoz_{N, A, X}$ induces the indicated $\UU_A$-morphism $\overline{\assoz}_{N, A, X}$ by the universal property of the coequalizer. The morphism from $N \otimes X$ to itself in the diagram must then be the identity as $N$ is in $\UU_A^0$.
It follows  that  $\overline{M}_{N, X} = \Id_{N \otimes_A \induceA(X)}$. 
\end{proof}

Notice that if $\UU$ is finite with
\begin{equation}\label{eq:BC} \BB=\UU_A \qquad \mathrm{and} \qquad \CC=\UU_A^0 \end{equation}
then by 
\cite[Remark 4.10]{LW1} and \cite[Corollary 4.21]{LW2} based on  \cite{Sh19} we have
\[ \mathrm{FPdim}(\mathcal{Z}_{\CC}(\BB))=\frac{\mathrm{FPdim}(\BB)^2}{\mathrm{FPdim}(\CC)^2}=\frac{\mathrm{FPdim}(\UU)^2/\mathrm{FPdim}_{\UU}(A)^2}{\mathrm{FPdim}(\UU)/\mathrm{FPdim}_{\UU}(A)^2}=\mathrm{FPdim}(\UU)\]

Therefore, we have the following corollary of Theorem \ref{relcent}

\begin{corollary}\label{cor_SchauenburgFinite}
Let $\UU, A, \UU_A$ be as in Assumption \ref{assumption} and also assume $\UU$ to be finite. Let $\BB,\CC$ be the categories of modules and local modules as in Equation \eqref{eq:BC}, then the functor $\Schauenburg$ gives rise to an equivalence of braided tensor categories
\begin{align*}\label{Eq:Embed} \Schauenburg:\mathcal{U} \stackrel{\sim}{\longrightarrow} \mathcal{Z}_{\CC}(\BB).
\end{align*}
\end{corollary}

We call $\Schauenburg$ the \emph{Schauenburg functor} in recognition of Peter Schauenburg's many contributions to the emergence of tensor categories, e.g. \cite{Sch} Cor.~4.5. 

Recall the notion of \emph{Witt classes} of modular tensor categories by the equivalence relation $\mathcal{C}\sim \mathcal{D}$ if and only if there exists tensor categories $\mathcal{T}, \mathcal{S}$ with $\mathcal{C} \boxtimes \mathcal Z(\mathcal S)\cong \mathcal{D}\boxtimes \mathcal{Z}(\mathcal{T})$. Since $\cZ(\BB)=\cZ_\CC(\BB)\boxtimes \bar{\CC}$ 
a consequence of Corollary \ref{cor_SchauenburgFinite} is the following result, that is well-known in the semisimple case see \cite{DMNO, DNO}:
\begin{corollary}\label{cor_Wittclass}
    Suppose a modular tensor category $\UU$ contains a commutative algebra $A$ with local modules   $\UU_A^0=\CC$ satisfying Assumption \ref{assumption}, then $\UU\sim\CC$ are Witt equivalent.
\end{corollary}
\begin{problem}\label{prob_Wittclass}
    Are there nonsemisimple modular tensor categories that are not Witt equivalent to any semisimple modular tensor category?
\end{problem}
\begin{problem}
    It may be possible to relax some assumptions of Theorem \ref{relcent}. For example, it is desirable to prove a variant of Theorem \ref{relcent} without assuming rigidity since rigidity is often very difficult to establish for vertex algebra tensor categories. 

    On the other hand vertex tensor categories are Grothendieck-Verdier categories \textup{\cite{ALSW21}} and if the vertex algebra is simple and its own contragredient dual, then the category even admits a contragredient functor as introduced in \textup{\cite{CKM2}}. This functor maps an object $X$ to its contragredient dual $X'$ and satisfies
    \[
    \Hom_{\cU}(Y \otimes X, \one) \cong \Hom_{\cU}(Y, X')
    \]
    for all objects $Y$. 
    Our Theorem holds if $\cU, \cU_A$ admit a contragredient functor such that $\induceA(X') \cong \induceA(X)'$. \\

    Let us also note: The condition on projective and injective modules is only used so that Lemma \ref{ff} applies. One can weaken the assumptions of the Lemma slightly, see Remark \ref{rem:ff} and so our main Theorem of this Section also holds under those weaker assumptions. 
\end{problem}

\subsection{Relatively finite setting}

We now want our main result of the previous section, Corollary \ref{cor_SchauenburgFinite}, to hold when $\CC$ is infinite, but $\BB$ is finite ``over'' $\CC$.  In the following we set up such a setting algebraically, and later discuss more categorical notions:

    

\begin{lemma}\label{lm:epimorphism}
Given a morphism of algebras $f:{U}\leftarrow {Z}$, let  $F:\Rep({U})\rightarrow \Rep({Z})$ be the functor given by pullback $F(M)={_f}M$, where $Z$ acts from the left by precomposition with $f$,   and let $G:\Rep({U})\leftarrow \Rep({Z})$  be its left-adjoint functor given by extension of scalars $G(V)=U_f\otimes_{{Z}}V$. Then
\begin{enumerate}
\item[a)] If $F$ is fully faithful, then $G\circ F \cong \id$. 
\item[b)] If $F$ is fully faithful and $F({U})$ is a faithfully flat ${Z}$-module, then $f$ is a surjection. Similarly, if $F$ is fully faithful and $F({U})$ is a finitely generated ${Z}$-module, then $f$ is a surjection. 
\end{enumerate}
Assume now in addition that $C$ is a common subalgebra of $U$ and $Z$ such that $f$ restricts to the identity on $C$, hence $F,G$ are compatible with the restriction functors to $C$. 
\begin{enumerate}
\item[c)] If $F$ is fully faithful and $U,Z$ are free $C$-modules of the same finite rank and $C$ is weakly finite (see e.g. \textup{\cite{Skry06}} Sec. 2, for example if $C$ is commutative) then $f$ is an isomorphism and hence $F:\Rep({U})\cong \Rep({Z})$ is an equivalence of categories. 
\end{enumerate}
 
\end{lemma}
\begin{proof}

a) Since $F$ is fully faithful we have $F:\Hom_U(M,N)\to \Hom_Z(F(M),F(N))$ is a bijection (surjectivity is the nontrivial statement). Then using the left-adjoint functor $G$
$$\Hom_U(M,N)\cong \Hom_Z(F(M),F(N))\cong \Hom_U(G(F(M)),N)$$
and the homomorphism is functorial in $M,N$. By using Yoneda's Lemma (or by applying this to the regular representation $N={U}$) we find a natural isomorphism $M \cong G(F(M))$. 

We remark that in algebraic terms this means that the canonical ${U}$-module map ${U}\to {U}_f\otimes_{Z} {_f}{U}$  is an isomorphism and such a ring morphism $f$ is called epimorphism of rings. The results given in b) are standard results on when epimorphisms of rings are actually ring surjections.

b) Now consider the exact sequence of ${Z}$-modules
$$Z\to F(U) \to Q\to 0$$
which under the right-exact functor $G$ maps to 
$${U} \xrightarrow{\cong} G(F({U})) \to G(Q) \to 0$$
where we already know the first map is an isomorphism by a), so this sequence is exact with $G(Q)=0$. We now want to argue that $Q=0$ turns the first sequence into an exact sequence and thus  $f:{Z}\to {U}$ is surjective as asserted. We do so  using two types of additional assumptions:

If we assume now $F(U)$ is a faithfully flat  $Z$-module, then exactness of the second sequence with $G(Q)=0$ implies exactness of the first sequence with $Q=0$. 

Now assume alternatively that the $Z$-module $Q$ is finitely generated. Then there is a surjection of $Z$-modules $Q\to Z/I$ for some ideal $I$. Hence $G(Q)$ surjects onto  $U_f\otimes_Z Z/I=U/f(I)\neq 0$ unless in the trivial case $U=f(Z)=f(I)$, so $f$ is surjective as asserted. 


c) Since by assumption $U$ is free $C$-module of finite rank and $f(Z) \supset C$ we have $F(U)$ finitely generated, and so by b) the algebra morphism $f:Z\to U$ is a surjection. Since by assumption $U,Z$ are free $C$-modules of equal finite rank $n$ we may identify $U=Z=C^n$ and view $f\in \End_C(C^n)$. Then the assumption on $C$ to be weakly finite (for example for $C$ commutative the Theorem of Vasconcelos \cite{V69}) precisely means that  $f$ surjective implies $f$ injective.

\end{proof}

To apply the previous Lemma, note that in many situations Hopf algebras are faithfully flat (or even free) over Hopf subalgebras, for example in the finite-dimensional case by the celebrated theorem of Nichols-Zoeller \cite{NZ}, or in the case $U$ is pointed by Radford \cite{Rad85}, or for $C$ commutative by results of Arkhipov-Gaitsgory \cite{AG}.

In our later setting of unrolled quantum groups (Section \ref{sec_quasiquantum})  we have to be careful since $\UU,\cZ,\CC$ are not the categories of all representations of Hopf algebras $U,Z,C$, but only the finite-dimensional representations with $C$ acting semisimple. However, we can consider the algebra $\dot{C}=\C^{\C^n}$ spanned by central idempotents $e_\lambda,\lambda\in\C^n$ with $\Rep(\dot{C})$ the category of $\C^n$-graded vector spaces, and correspondingly the algebras $\dot{Z}$ and $\dot{U}$, where the latter is studied in \cite[Chp. 23]{Lusz93}. Also, in this case all representations are direct sums of finite-dimensional representations, so any functor between the finite-dimensional representations can be extended additively.

\begin{corollary}\label{cor_SchauenburgInfinite}
 Assume that tensor categories $\UU,\BB=\UU_A,\CC=\UU_A^0,\cZ=\cZ_\CC(\BB)$ are realized as abelian categories by finite-dimensional representations over algebras $\dot{U},\dot{B},\dot{C}, \dot{Z}$, such that all representations are direct sums of finite-dimensional representations. If Assumption \ref{assumption} holds then the functor 

$$\Schauenburg:\UU\to \cZ_\CC(\BB) $$

can be realized as pullback of an algebra morphism $f:\dot{U}\leftarrow \dot{Z}$. 

Assume further that $\dot{U}$ is a  free $\dot{C}$-module of finite rank, which is equal to the rank of $\dot{Z}$, then by the previous Lemma we have an equivalence of categories
$$\Schauenburg:\UU\stackrel{\sim}{\longrightarrow} \cZ_\CC(\BB). $$
\end{corollary}

Note that we do not have an immediate way to realize the relative Drinfeld center as $\Rep(\dot{Z})$, however in the splitting case we have an explicit presentation in terms of a Drinfeld double by Lemma \ref{lm_relcenterSplitting}.

\begin{problem}
Ultimately, one would want to work with a more categorical definition of relative finiteness. For example, one could view $\UU$ as a $\CC$-module category and require the existence (and freeness and finiteness, respectively) of the relative coend defined in \cite{BM21}. As a second step, it would be desirable to not have to require the rank equality for $\cZ,\cU$ but to derive it from the finite dimension of the algebra $A$, as in the finite case.
\end{problem}

\section{Splitting tensor functor to local modules}\label{sec_splitting}
\subsection{Definition}

Let $\UU$ be a braided tensor category and $A\in \UU$ a commutative algebra with category of $A$-modules and local $A$-modules denoted by
\[ \BB:=\UU_A, \qquad \CC:=\UU_A^0\]
and denote by $\iota$ the inclusion of the full Serre subcategory $\iota: \CC \to \BB$. 

\begin{definition}\label{localization}
A \emph{(quasi-) splitting functor} to local modules is a 
 faithul (quasi-) tensor functor $\Localization:\BB \to \CC$ such that $\Localization\circ\iota=\id$.
\end{definition}

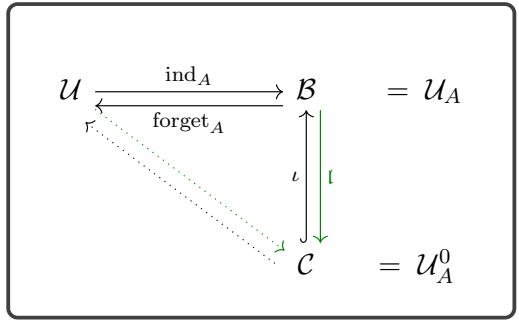
\begin{center}
\begin{tcolorbox}[colback=white, width=0.45\linewidth]
$$
 \begin{tikzcd}[row sep=10ex, column sep=15ex]
   \mathbf{\UU} \arrow{r}{\mathrm{ind}_A}   
   \arrow[darkgreen, dotted]{dr}{}
   & \arrow[shift left=2]{l}{\mathrm{forget}_A}
   \arrow[shift left=2,darkgreen]{d}{\Localization}   \mathbf{\BB}&\hspace{-2.5cm}=\;\UU_A \\
   &\arrow[hookrightarrow, above]{u}{\iota}\arrow[below, shift left=2.3, dotted]{ul}{}
   \mathbf{\CC}&\hspace{-2.7cm} =\;\UU_A^0
    \end{tikzcd}
$$
\end{tcolorbox}
\end{center}

Recall from \cite{EGNO} that a functor is faithful if it is injective on morphism spaces, and a quasi tensor functor $\Localization$ is a functor between tensor categories such that $\Localization(X\otimes Y)\cong\Localization(X)\otimes \Localization(Y)$, whereas a tensor functor is a  functor together with a family of isomorphisms $\Localization^ T:\Localization(X\otimes Y)\cong\Localization(X)\otimes \Localization(Y)$ fulfilling the hexagon coherence axiom.

\begin{remark}\label{exm_localization_examples}
In many standard examples of commutative algebras the existence of such a splitting tensor functor to local modules fails:
\begin{enumerate}
    \item In the simple current extension associated to a lattice $\Lambda\subset \Gamma$ discussed in Example~\ref{exm_lattice}, the existence of a splitting tensor functor would amount to the splitting of the following exact sequence of abelian groups.
    $$1 \to \Lambda \to \Gamma \to \Gamma / \Lambda \to 1$$
    \item In the case of an orbifold of a finite group $G$ acting on $\CC$ and $\BB=\bigoplus_{g\in G} \CC_g$ with $\CC_e=\CC$, which was discussed in Example \ref{exm_orbifold}, the existence of a splitting tensor functor implies that all $\CC_g=\CC$ and moreover that the initial action of $G$ on $\CC$ was trivial.
    \item We can take so-called nontrivial liftings of a Nichols algebra $\Nichols$, see Section \ref{sec_Nichols}. Such a lifting is a filtered Hopf algebra $B \supset C$ with associated graded Hopf algebra $\Nichols \rtimes C$, where the latter has a category of representations admitting a splitting functor. Generic cases for such liftings appear prominently in \textup{\cite{AS10}}, for example for $u_q(\mathfrak{\sl_2})^\geq$ one may  replace the relation $E^p=0$ by $E^{p}=1-K^p$. A classification program for certain classes of Nichols algebras is found in \textup{\cite{AGI19}} and references therein.
\end{enumerate}
\end{remark}

\begin{remark}
    We briefly discuss different instances of ``quasi" appearing: 
    
    Suppose $\CC$ has a tensor functor or quasi-tensor functor to the category of vector spaces, then $\CC=\Rep(C)$ for a Hopf algebra or quasi-Hopf algebra $C$. In both cases, suppose $\Localization$ is a true tensor functor, then $\BB=\Rep(B)$ for a Hopf algebra or quasi-Hopf algebra $B$, which contains $C$ as a true quasi-Hopf subalgebra (so the coassociator is the same) and $\Nichols$ in the next section is a true Hopf algebra inside $\CC$.
    
    Suppose on the other hand that $\Localization$ is only a quasi-tensor functor, then $B$ contains $C$ as an algebra but the coassociator may be different, and $\Nichols$ in the next section may be a quasi-Hopf algebra inside $\CC$. 
\end{remark}

 \subsection{A criterion giving splitting functors}

The following criterion gives a hint why the extensions by commutative algebras that appear in the logarithmic Kazhdan-Lusztig correspondence all have splitting functors. 

\begin{lemma}\label{lm_socle}
Let $\iota:\CC \hookrightarrow \BB$ be a full Serre tensor subcategory, where $\CC$ is semisimple and $\BB$ consists of objects with finite Jordan-H\"older length. Assume that all simple objects of $\BB$ lie in $\CC$, then there exists a splitting tensor functor $\Localization:\BB\to \CC$.
\end{lemma}
\begin{proof}
Let $X \in \BB$, consider the  socle filtration $0=X_0 \subset X_1 \subset \hdots \subset X_n:=X$ with $X_{k+1}$ the preimage of the socle of $X/X_{k}$, and let $\mathrm{gr}(X)$ be the associated graded object, which is semisimple. If all simple objects in $\BB$ lie in $\CC$, then $\mathrm{gr}(-)$ gives rise to a functor $\Localization:\BB \to \CC$  satisfying $\Localization \circ \iota = \mathrm{Id}$. \\

It is clear that this functor has a canonical tensor structure: Indeed, if the objects $X,Y$ have socle filtrations    
$$0=X_0 \subset X_1 \subset \hdots \subset X_n:=X$$
$$0=Y_0 \,\subset\, Y_1 \, \subset \, \hdots \, \subset \, Y_n:=Y$$
then the tensor product has an induced filtration
$$0=X_0\otimes Y_0 \,\subset \, \cdots \,
\bigoplus_{i+j=k} X_i\otimes Y_j 
\, \cdots \, \subset \, X_n\otimes Y_n := X\otimes Y 
$$
and since any filtration refines to the socle filtration we have a canonical morphism 
$$\Localization^T:\; 
\left(\bigoplus_{i+j=k+1} X_i\otimes Y_j \right) /
\left(\bigoplus_{i+j=k} X_i\otimes Y_j \right)
\longrightarrow
(X\otimes Y)_{k+1} / (X\otimes Y)_k
$$
\end{proof}

Morally, one might say that orbifolds of semisimple vertex algebras by nilpotent algebras of screenings have splitting functors. Note that the argument would also work in other cases where we have filtrations in terms of local modules that are in some sense canonical (so that they are compatible with tensor products).

\begin{problem}
    Suppose that all simple objects in $\BB$ are in the image of $\CC$, without any assumption on semisimplicity, does this imply the existence of a splitting functor? 
\end{problem}


\subsection{The Hopf algebra \texorpdfstring{$\Nichols$}{N}}

The standard situation where we have a tensor embedding $\iota:\CC\to \BB$ and a left inverse tensor functor $\Localization:\BB\to \CC$ is the category of representations of a Hopf algebra $\Nichols$ inside a braided tensor category $\CC$.

\begin{definition}
    A {bialgebra} $\Nichols$ inside a braided tensor category $\CC$ is an algebra in the tensor category $\CC$ as in Definition \ref{def:alg}
    $$\Nichols,\quad m:\Nichols\otimes \Nichols\to \Nichols,\quad u:1\to \Nichols$$
    which has at the same time the dual structure of a coalgebra
    $$\Nichols,\quad \Delta:\Nichols\to \Nichols\otimes \Nichols,\quad \epsilon:\Nichols\to 1$$
    and the two structures have the following compatibility
    involving the braiding $c_{\bullet,\bullet}$ of $\CC$
    \begin{align*}
        \Delta\circ m&=(m\otimes m)(\id\otimes c_{\Nichols,\Nichols}\otimes \id)(\Delta\otimes \Delta)\\
        \Delta\circ u&=u\otimes u \\
        \epsilon\circ m &= \epsilon\otimes \epsilon \\
        \epsilon\circ u &= \id_1
    \end{align*}
We illustrate the first identity in terms of string diagrams, which we read bottom to top 
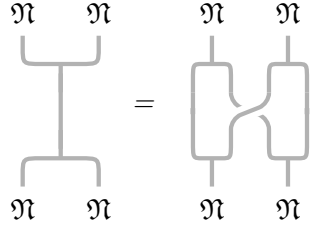
\begin{center}
            \begin{grform}
                        \begin{scope}[scale = 0.5]
                                \dMult{1}{4}{2}{-1.5}{\grau}
                                \dMult{1}{0}{2}{1.5}{\grau}
                        \end{scope}
                        \vLine{1}{0.5}{1}{1.5}{\grau}
                        
                        \draw (0.5 , -0.3) node {$\Nichols$};
                        \draw (0.5 , 2.3) node {$\Nichols$};
                        \draw (1.5 , -0.3) node {$\Nichols$};
                        \draw (1.5 , 2.3) node {$\Nichols$};
                \end{grform}
                = \hspace{.2cm}
                \begin{grform}
                        \begin{scope}[scale = 0.5]
                                \dMult{0.5}{1.5}{1}{-1.5}{\grau}
                                \dMult{0.5}{2.5}{1}{1.5}{\grau}
                                \dMult{2.5}{1.5}{1}{-1.5}{\grau}{black}
                                \dMult{2.5}{2.5}{1}{1.5}{\grau}{black}
                                \vLine{0.5}{1.5}{0.5}{2.5}{\grau}
                                \vLine{2.5}{1.5}{1.5}{2.5}{\grau}
                                \vLineO{1.5}{1.5}{2.5}{2.5}{\grau}
                                \vLine{3.5}{1.5}{3.5}{2.5}{\grau}
                        \end{scope}
                        \draw (0.5 , -0.3) node {$\Nichols$};
                        \draw (0.5 , 2.3) node {$\Nichols$};
                        \draw (1.5 , -0.3) node {$\Nichols$};
                        \draw (1.5 , 2.3) node {$\Nichols$};
                \end{grform}
\end{center}
A bialgebra is called a Hopf algebra if there exists an antipode $S:\Nichols\to \Nichols$ such that 
$$m \circ (S\otimes \id) \circ \Delta=m \circ (\id\otimes S) \circ \Delta =\epsilon$$
In the following we always assume that the antipode is bijective.     
    \end{definition}
    
    \begin{definition}\label{def_RepB}
    The category of representations $\Rep(\Nichols)(\CC)$ of a bialgebra $\Nichols$ inside $\CC$ is the category of representations of the underlying algebra in Definition \ref{def:RepA}, so objects $X\in \CC$ together with an action $m_X:\Nichols\otimes X\to X$. For a bialgebra we can endow this category with a tensor product different from the tensor product over a commutative algebra: For $(X,m_X),\;(Y,m_Y)\in\Rep(\Nichols)(\CC)$ we define $X\otimes Y$ as object in $\CC$ by the tensor product in $\CC$, and endow it with the action (again not drawing associators)
    $$m_{X\otimes Y}:\Nichols\otimes (X\otimes Y) \xrightarrow{\Delta}
    (\Nichols\otimes \Nichols)\otimes (X\otimes Y)
    \xrightarrow{c_{\Nichols,X}} (\Nichols\otimes X)\otimes (\Nichols\otimes Y)
    \xrightarrow{m_X\otimes m_Y} X\otimes Y
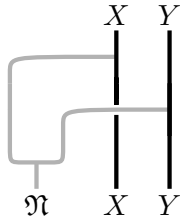
$$
\begin{center}
            \begin{grform}
                        \begin{scope}[scale = .7]
                                \dMult{-.5}{1}{1}{-1}{\grau}
                                \dAction{-0.5}{2}{2}{1}{\grau}{black}
                                \dAction{0.5}{1}{2}{1}{\grau}{black}
                                \vLine{-0.5}{1}{-0.5}{2}{\grau}
                                \vLine{1.5}{0}{1.5}{1.4}{black}
                                \vLine{1.5}{1.7}{1.5}{2}{black}
                                \vLine{2.5}{0}{2.5}{3}{black}
                                \draw (0 , -0.3) node {$\Nichols$};
                                \draw (1.5 , -0.3) node {$X$};
                                \draw (2.5 , -0.3) node {$Y$};
                                \draw (1.5 , 3.3) node {$X$};
                                \draw (2.5 , 3.3) node {$Y$};
                        \end{scope}
                        
                \end{grform}
    \end{center}
    and we endow the unit object $1\in\CC$ with the trivial action via $\epsilon$. \\

    If $\Nichols$ is a Hopf algebra and $\CC$ is rigid, then $\Rep(\Nichols)(\CC)$ is again rigid. The dual of an $\Nichols$-module is the dual object in $\CC$ and the action precomposed with the antipode. 
    \end{definition}

For the category of representations of a quasitriangular Hopf algebra $\CC=\Rep(C)$ we can write the tensor category $\Rep(\Nichols)(\CC)$ again as a category of representations of a Hopf algebra, namely the biproduct or smash product  due to Radford \cite{Rad85} and Majid \cite{Maj94}, which we denote $\Nichols\rtimes C$

    \newcommand{\forgetNichols}{\mathfrak{f}}
    \begin{lemma}\label{lm_RepBFunctors}
    In this situation we have tensor functors 
    \begin{align*}
        \Rep(\Nichols)(\CC) &\xrightarrow{\hspace{.37cm}\forgetNichols\hspace{.37cm}} \CC \\
        \cZ(\Rep(\Nichols)(\CC)) &
        \xleftarrow{(\epsilon,{c}^{-1})} \CC 
    \end{align*}
    where $\forgetNichols:\Rep(\Nichols)(\CC)\to \CC$ is the functor forgetting the the $\Nichols$-action and $\iota:\CC\to \Rep(\Nichols)(\CC)$ is the functor endowing an object $X\in \CC$ with the trivial $\Nichols$-action via $\epsilon$, and this functor extends to a functor $(\iota,{c}^{-1}):\CC\to \cZ(\Rep(\Nichols)(\CC))$ with the following half-braiding 
    $$c_{\iota(X),N}^{-1}:\,  \iota(X)\otimes N\to N\otimes \iota(X)$$
    This half-braiding $c_{\iota(X),N}^{-1}$ is $\Nichols$-linear since the actions on both sides is 
    \begin{align*}
    m_{N\otimes \iota(X)}
    &=m_N\otimes \id_{\iota(X)}\\
    m_{\iota(X)\otimes N}
    &=(\id_{\iota(X)}\otimes m_N)(c_{\Nichols,\iota(X)}\otimes \id_N)
    \end{align*}
    and hence we have
    \begin{center}
        \begin{grform}
           \begin{scope}[scale = 1]
                \vLine{0}{0}{0}{2}{\grau}
                \vLine{1}{0}{2}{2}{black}
                \vLineO{2}{0}{1}{2}{black}
                \dAction{0}{2}{1}{1}{\grau}{black}
                \vLine{2}{2}{2}{3}{black}
                \draw (0 , -0.3) node {$\Nichols$};
                \draw (1 , -0.3) node {$\iota(X)$};
                \draw (2 , -0.3) node {$N$};
                \draw (1 , 3.3) node {$N$};
                \draw (2 , 3.3) node {$\iota(X)$};
            \end{scope}  
        \end{grform}
        \quad=
        \begin{grform}
           \begin{scope}[scale = 1]
                \vLine{1}{0}{1}{.7}{black} 
                \dAction{0}{0}{2}{1.75}{\grau}{black}
                \vLine{1}{1}{2}{3}{black}
                \vLineO{2}{1}{1}{3}{black}
                \draw (0 , -0.3) node {$\Nichols$};
                \draw (1 , -0.3) node {$\iota(X)$};
                \draw (2 , -0.3) node {$N$};
                \draw (1 , 3.3) node {$N$};
                \draw (2 , 3.3) node {$\iota(X)$};
            \end{scope}  
        \end{grform} 
        \quad
    \end{center}
\end{lemma}

\begin{example}[Rank 1]\label{exm_NicholsRank1}
    In the braided tensor categories $\CC=\Vect_\C$  with braiding 
    $$\C_\lambda\otimes\C_\mu\xrightarrow{e^{\pi\i(\lambda\mu)}} \C_\mu \otimes \C_\lambda.$$ 
    Let $X=\C_\alpha$ with generator $x$ in degree $\alpha=-\sqrt{2/p}$ (the sign is mere convention for later) for $p \in \Z_{\geq 1}$, so the self-braiding is $e^{2\pi\i/p}$. Then there is a Hopf algebra
    $$\Nichols=\C[x]/x^p$$
    $$\Delta(x)=1\otimes x+x\otimes 1,\qquad \epsilon(x)=0.$$
    We remark that this gives a well-defined Hopf algebra in $\CC$ precisely because the given self-braiding has order $p$, so all nontrivial Gau{\ss} binomial coefficients in the following computation vanish 
    $$\Delta(x^p)=\sum_{i=0}^p {p\choose k}_{e^{2\pi\i/p}} (x^k\otimes x^{p-k})
    =x^p\otimes 1+1\otimes x^p$$
    and the relation $x^p=0$ is compatible with the definition of $\Delta$. For $p\geq 2$ this is the simplest example of a Nichols algebra in the next section, and the Borel part of $u_q(\sl_2)$ at $q=e^{\pi\i/p}$. \\

    We discuss in this case the category of representations  $\BB=\Rep(\Nichols)(\CC)$: It consists of $\C$-graded vector spaces with an additional action of $x$ shifting the degree by $\alpha$ such that $x^p$ acts by zero. The simple modules are precisely the simple modules $\C_{\lambda,1}=\C_{\lambda}$ with trivial $\Nichols$-action $\epsilon(x)=0$. There is clearly a unique indecomposable module $\C_{\lambda,l}$ with top $\C_\lambda$ and lenght $0\leq l< p$, namely the one with action of $x$ (and at the same time Loewy diagram) given by  
    $$\C_\lambda\stackrel{x}{\longrightarrow}\C_{\lambda+\alpha}\stackrel{x}{\longrightarrow}\;\cdots\; \stackrel{x}{\longrightarrow} \C_{\lambda+\alpha(l-1)}$$
    and in particular $\C_{\lambda,p}$ is the indecomposabele projective cover of $\C_{\lambda,1}$. The tensor product in $\CC$ with the action via $\Delta(x)=1\otimes x+x\otimes 1$ is, for example (for $p>2$)
    \begin{align*}
    \C_{\lambda,l}\otimes \C_{\mu,1} &= \C_{\lambda+\mu,l}\\
    \C_{\lambda,2}\otimes \C_{\mu,2} &= \C_{\lambda+\mu,3}+\C_{\lambda+\mu+\alpha,1}
    \end{align*}
\end{example}

We now want to have statements ensuring conversely that a splitting tensor functor implies the existence of a bialgebra or Hopf algebra $\Nichols$ in $\CC$ with 
$$\BB=\Rep(\Nichols)(\CC)$$

\begin{example}[Finite Hopf algebra case]\label{exm_splittingIsHopfAlgB}
In the case where $\CC\cong \Rep(C)$ for a finite-dimensional (quasi-) Hopf algebra and $\BB$ a finite tensor category,  the existence of $\Nichols$ follows: By the reconstruction theorem $\BB=\Rep(B)$ for some (quasi-) Hopf algebra $B\supset C$ and the existence of $\iota$ implies $B=\Nichols \rtimes C$, since we can consider the image of the regular representation $\iota(C_{reg})$ as a $B$-module, which provides us with a projection of $B\to C$, and then apply the Radford projection Theorem \textup{\cite{Rad85}}. Here, $\Nichols$ is a Hopf algebra in $\cZ(\CC)$. If $\CC$ is braided and the functor $\iota$ gives rise to a functor $(\iota,\bar{c}^{-1}):\CC\to \cZ(\BB)$, then $\Nichols$ comes from an object in the braided tensor category $\CC$ (as can be seen from the coproduct in the Radford biproduct) 
\end{example}

\begin{problem}\label{prob_relativecoend}
We would expect in general that the existence of a tensor functor with a right inverse tensor functor $\CC\leftrightarrows \BB$ for finite categories implies the existence of a Hopf algebra $\Nichols\in\CC$ with $\BB=\Rep(\Nichols)(\CC)$. A possible construction for $\Nichols$ could be the relative coend of $\BB$ as a module category over $\CC$. 
\end{problem}


 In an infinite-dimensional setting without using reconstruction statements, we will use the following setting and result. 

 \newcommand{\oneside}{{_{\Nichols}}}

\begin{lemma}\label{lm_infiniteSplittingIsHopfAlgB}
    Suppose that $\CC$ is a braided tensor category and $\BB$ is a tensor category with tensor functor $\iota:\CC\to \BB$ which extends to a braided functor $(\iota,c'): \bar{\CC}\to \cZ(\BB)$ and with a tensor functor $\Localization:\BB\to \CC$ with tensor structure $\Localization^T$ such that $\Localization\circ \iota=\id$ and $\Localization(c')=c^{-1}$.
    
    Suppose that $\Nichols$ is an algebra in $\CC$. Then the abelian category $\Rep(\Nichols)(\CC)$ is endowed with the  functor  $\forgetA:\Rep(\Nichols)(\CC)$ forgetting the $\Nichols$-action (and possibly tensor structure $\forgetNichols^T$) such that $\forgetNichols\circ \varepsilon=\id$. 
    
    Suppose we have an an equivalence $F$ of abelian categories that commute with the functors $\forgetNichols,\Localization$ to $\CC$  
    \begin{align*}\label{ModCatCondition}
          \xymatrix{   
          \Rep(\Nichols)(\CC)\ar[rr]^F \ar[dr]_{\forgetNichols}
          &&
          \BB \ar[dl]^{\Localization} 
          \ar@{}[dll]^(.25){}="a"^(.55){}="b" 
          \ar@{=>}_\eta "a";"b"\\
          &
          \CC
          &
          }
    \end{align*}
    in that we have a natural isomorphism $\eta_\bullet: \Localization\circ F\to \forgetNichols$.

    Both categories have the structure of a right $\CC$-module category: For $N\in \BB$ and $V\in \CC$ we define $N.V=N\otimes_\BB \iota(V)$. For $M\in \BB$ and $V\in \CC$ we define $M.V=\oneside\forgetNichols(M)\otimes_\CC V$, which is the $\CC$-object $\forgetNichols(M)\otimes_\CC V$ endowed with $\Nichols$ acting on the first factor only.
    Suppose now that $F$ has the additional structure of an equivalence of  $\CC$-module categories
    $$F^{T}:\, F(\oneside{\forgetNichols(M)\otimes_\CC V}) \stackrel{\sim}{\longrightarrow} F(M)\otimes_\BB \iota(V)$$
    for all $M\in \Rep(\Nichols)(\CC)$ and $V\in \CC$. Suppose that the image of this structure in $\CC$ under $\Localization$ is trivial in the sense that the following diagram involving the tensor structures $\Localization^T,\forgetNichols^T$ commutes
         \begin{equation}\label{ModCatCondition}
         \begin{split}
          \xymatrix{   
        \Localization(F(\oneside{\forgetNichols(M)\otimes_\CC V})) 
          \ar[rr]^{\Localization(F^T)} 
          \ar[d]^{\eta_{\oneside{\forgetNichols(M)\otimes_\CC V}}}
          &&
          \Localization(F(M)\otimes_\BB \iota(V))
          \ar[d]^{\Localization^T}
          \\
          \forgetNichols(\oneside{\forgetNichols(M)\otimes_\CC V})
          \ar[d]^{\forgetNichols^T}
           && 
          \Localization(F(M))\otimes_\CC \Localization(\iota(V))
          \ar[d]^{\eta_M\otimes \id}
          \\
          \forgetNichols(M)\otimes_\CC V 
          \ar[rr]^{=}
         &&
         \forgetNichols(M) \otimes_\CC V
             }
          \end{split}
          \end{equation}
              
    Then $\Nichols$ has the structure of  a bialgebra in the braided tensor category $\CC$ such that $F$ is an equivalence of tensor categories. \\
    
    If $\BB$ and $\CC$ are rigid and $\iota$ and $\Localization$ are functors between rigid tensor categories, $\Nichols$ has the structure of a Hopf algebra such that $F$ is an equivalence of rigid tensor categories.
 \end{lemma}


Note that the additional conditions (\ref{ModCatCondition}) in the Hopf algebra intuition means that $\Nichols$ is the subalgebra of left coinvariants with respect to $\Localization$. Together with the centrality it implies that the reversed order $\iota(V)\otimes M\to V\otimes \Localization(M)$ is given by the braiding in $\CC$. Note that in the Radford biproduct there are other choices of subalgebras giving the same category, for example the right coinvariants, where the roles are reversed. \\

\begin{proof}
    All tensor products are in $\CC$ unless specified otherwise.

    First, we construct a functor that provides a notion of trivial $\Nichols$-action 
    $$\varepsilon:\; \CC \stackrel{\iota}{\longrightarrow} \BB \stackrel{F^{-1}}{\longrightarrow} \Rep(\Nichols)(\CC)$$
    and is by construction a right inverse to $\forgetNichols$ with the following natural transformation
    $$\eta_{F^{-1}(\iota(\bullet))}^{-1}:\; (\forgetNichols\circ \varepsilon)=(\forgetNichols\circ F^{-1} \circ \iota)\longrightarrow (\Localization \circ \iota)=\id$$
    In particular $\varepsilon(1_\BB)$ gives a $\Nichols$-module on the object $\Localization(1_\BB)=1_\CC$. We use this to define the counit $\epsilon:\Nichols\to 1_\CC$ which is my construction a $\CC$-morphism and an algebra morphism. We may now define for $V\in\CC$ the respective object $\V_\epsilon\in \Rep(\Nichols)(\CC)$ with action given by $\epsilon$. The previous natural transformation $\eta_{F^{-1}(\iota(V))}^{-1}$ gives an isomorphism $\varepsilon(V)\cong V_\epsilon$. \\ 

    Second, we use the tensor product in $\BB$ to induce a tensor product for $M,N \in \Rep(\Nichols)(\CC)$ via
    $$M\otimes_F N := F^{-1}\big(F(M)\otimes_\BB F(N)\big)$$
    which is a $\Nichols$-module with an underlying  $\CC$-object isomorphic to $\forgetNichols(M) \otimes \forgetNichols(N)$
    \begin{align*}
    \forgetNichols(F^{-1}(F(M)\otimes_\BB F(N)))
    \xrightarrow{\eta^{-1}_{F^{-1}(F(M)\otimes_\BB F(N))}}
    &\quad\Localization(F(M)\otimes_\BB F(N)) \\
    \xrightarrow{\qquad \ \ \ \Localization^T\ \ \ \qquad }
    &\quad\Localization(F(M))\otimes_\BB \Localization(F(M)) \\
    \xrightarrow{\quad \ \ \ \eta_M\otimes \eta_N\  \ \ \quad}
    &\quad\forgetNichols(M) \otimes \forgetNichols(N) 
    \end{align*}

    For any $\Nichols$-module $M$ the action map $\rho_M:\,\Nichols\otimes M\to M$ becomes a $\Nichols$-module-map if we take the regular representation $\Nichols_{reg}$ and the action on the first tensor factor only, i.e.~in our notation $\rho_M$  gives rise to a morphism in $\Rep(\Nichols)(\CC)$
$$\hat{\rho}_M:\,\oneside{\forgetNichols(\Nichols_{reg})\otimes \forgetNichols(M)}\to M$$
    The module category structure morphism $F^T$ allows us to express this as a $\BB$-morphism from the $\BB$-tensor product
$$
F(\Nichols_{reg})\otimes_\BB \iota(\forgetNichols(M))
\xrightarrow{(F^T)^{-1}}
F(\oneside{\forgetNichols(\Nichols_{reg})\otimes \forgetNichols(M)})
\xrightarrow{F(\hat{\rho})}
F(M)
$$
 which is essentially $\rho$ on the level of $\CC$ in the following sense: From the assumed compatibility diagam for $F^T$ we get

 \begin{equation}
\begin{split}
          \xymatrix{   
        \Localization(F(M))
        \ar[d]^\eta
        &&
        \Localization(F(\oneside{\forgetNichols(\Nichols_{reg})\otimes\forgetNichols(M)})) 
          \ar[d]^{\eta}
        \ar[ll]^{\Localization(F(\hat{\rho}_M))\qquad}
          &&
          \Localization(F(\Nichols_{reg})\otimes_\BB \iota(\forgetNichols(M)))
          \ar[d]^{\Localization^T}
         \ar[ll]^{\Localization(F^T)^{-1}} 
          \\
          \forgetNichols(M)
          \ar[d]^{=}
          &&
          \forgetNichols(\oneside{\forgetNichols(M)\otimes \forgetNichols(M)})
          \ar[d]^{\forgetNichols^T}
          \ar[ll]^{\forgetNichols(\hat{\rho}_M)\qquad}
           && 
          \Localization(F(\Nichols_{reg}))\otimes \Localization(\iota(V))
          \ar[d]^{\eta\otimes \id}
          \\
          \forgetNichols(M)
          &&
          \forgetNichols(\Nichols_{reg})\otimes \forgetNichols(M)
                    \ar[ll]^{\rho_M\qquad}
         &&
         \forgetNichols(\Nichols_{reg}) \otimes \forgetNichols(M)
         \ar[ll]^{=}
             }
          \end{split}
\end{equation}
 After applying $F^{-1}$ on both sides and using $F^{-1}(\iota(\forgetNichols(M)))=\varepsilon(\forgetNichols(M))$ we get a $\Nichols$-morphism 
$$F(\Nichols_{reg})\otimes_F \varepsilon(\forgetNichols(M))\xrightarrow{\hat{\hat{\rho}}_M}  
M
$$
In particular, for $M=\Nichols_{reg}$ we get a $\Nichols$-morphism 
$$\Nichols_{reg}\otimes_F \varepsilon(\Nichols)
\xrightarrow{\hat{\hat{\rho}}_{\Nichols_{reg}}}
\Nichols_{reg}
$$
and the assumed compatibility diagram we shows that on the level of $\CC$-objects (and up to $\eta$ and tensor structures) this is the multiplication in $\Nichols$.\\

To construct the coproduct we now consider the tensor product $M=\Nichols_{reg} \otimes_F \Nichols_{reg}$ and the respective $\Nichols$-morphism 

$$\Nichols_{reg}\otimes_F \varepsilon(\Nichols\otimes \Nichols)
\xrightarrow{\hat{\hat{\rho}}_{\Nichols_{reg}\otimes_F \Nichols_{reg}}}
\Nichols_{reg}\otimes_F \Nichols_{reg}
$$
and precompose with  $e\otimes_F e$ for the unit $e:1_\CC\to \Nichols$ to define a $\Nichols$-morphism 
$$\Nichols_{reg}
\xrightarrow{\Delta}
\Nichols_{reg}\otimes_F \Nichols_{reg}
$$
The coassociativity follows from the coassociativity of the tensor product. Counitality follows from unitality of the tensor product $\otimes_F$. The braided bialgebra axiom is essentially that the map $\Delta$ above is a $\Nichols$-morphism between regular representations, as established. Let us argue more concretely:

Define the $\Nichols$-morphism 
\begin{align*}\mathrm{BiAx}:\,
(\Nichols_{reg}\otimes_F \Nichols_{reg})\otimes_F (\Nichols_\epsilon\otimes_F \Nichols_\epsilon)
&\xrightarrow{assoz}
\Nichols_{reg}\otimes_F (\Nichols_{reg}\otimes_F \Nichols_\epsilon) \otimes_F \Nichols_\epsilon\\
&\xrightarrow{\id\otimes F^{-1}(c')\otimes \id}
\Nichols_{reg}\otimes_F (\Nichols_{\epsilon}\otimes_F \Nichols_{reg}) \otimes_F \Nichols_\epsilon\\
&\xrightarrow{assoz}
(\Nichols_{reg}\otimes_F \Nichols_{\epsilon})\otimes_F (\Nichols_{reg}\otimes_F \Nichols_\epsilon)\\
&\xrightarrow{\hat{\hat{\rho}}_{\Nichols_{reg}}\otimes \hat{\hat{\rho}}_{\Nichols_{reg}}}
\Nichols_{reg}\otimes_F \Nichols_{reg}
\end{align*}
Then the bialgebra axiom on the level of elements reads:
  \begin{align*}
        \mathrm{BiAx}(\Delta(a)\otimes \forgetNichols(\Delta)(b))
        &= \mathrm{BiAx}(a.(1\otimes 1\otimes  \forgetNichols(\Delta)(b)))\\
        &= a.\mathrm{BiAx}(1\otimes 1 \otimes  \forgetNichols(\Delta)(b)) \\
        &= a.\Delta(b) \\
        &= \Delta(ab) 
    \end{align*}
If $\BB$ and $\CC$ are rigid and $\iota$ and $\Localization$ are functors between rigid tensor categories, then $\hat{\hat{\rho}}_{F^{-1}(F(\Nichols_{reg})^*)}$ again precomposed with the unit gives an antipode on $\Nichols$, so $\Nichols$ is a Hopf algebra.

\end{proof}

\begin{remark}
Now suppose we are in the situation that $\BB=\UU_A$ is an extensions of a commutative algebra. Then we have a tensor functor $\iota:\CC\to \BB$ which can be extended to a functor $(\iota,\bar{c}^{-1}_{N,\bullet}):\CC\to \cZ(\BB)$ by Definition \ref{def_cent2}. 
\end{remark}

\section{Nichols algebras}\label{sec_Nichols}

\subsection{Definition and relevant examples}

Let $\CC$ be a braided tensor category and $X\in\CC$ an object. The tensor algebra
$$\mathfrak{T}(X):=\bigoplus_{n=0}^\infty X^{\otimes n}$$
becomes a Hopf algebra in $\CC$ if it is endowed with the coproduct $\Delta:\; X \to 1\otimes X+X\otimes 1$ with the identity in $X$ and with the counit $\epsilon$ defined to be zero in degree $\geq n$.
\begin{definition}
    The Nichols algebra $\NicholsOf(X)$ is an $\mathbb{N}$-graded Hopf algebra in $\CC$, which is a surjective image of the Hopf algebra $\mathfrak{T}(X)$, with the universal property that any other such surjective image of $\mathfrak{T}(X)$ factors over $\NicholsOf(X)$.
\end{definition}
There is also an alternative definition in terms of the quantum symmetrizer  \cite{Wor}. Nichols algebras are a crucial component in the classification of pointed Hopf algebras \cite{AS10}. Nichols algebras in $\CC=\Vect_{\Gamma}$ for $\Gamma=\C^n$ (or some other abelian group) have been classified in \cite{Heck09}, for nonabelian groups see \cite{HV14}. A comprehensive standard textbook for Nichols algebras is \cite{HS20}.
In the context of vertex algebras, the non-local screening operators are described by Nichols algebras in the category of representations of the vertex algebra \cite{Len21}. For an introduction to Nichols algebras and an introduction to screening operators, see for example \cite{FL22} Section 2 and 3. 

Let $\CC=\Vect_{\Gamma}$ for $\Gamma=\C^n$ a vector space with an inner product $(\lambda,\mu)$, and define a braiding by $\sigma(\lambda,\mu)=e^{\pi\i(\lambda,\mu)}$ and a trivial associator. In a more general setting let $\Gamma$ be an abelian group endowed with a quadratic form $Q(\lambda)$ and associated bimultiplicative form $B(\lambda,\mu)$, then let $(\sigma,\omega)$  a choice of representing abelian $3$-cocycle for the abelian cohomology class defined by $Q$ \cite{Mac52}, so that a braiding is given by the $2$-cocycle $\sigma$ and the associator is given by the $3$-cocycle $\omega$, such that $Q(\lambda)=\sigma(\lambda,\lambda)$ and $B(\lambda,\mu)=\sigma(\lambda,\mu)\sigma(\mu, \lambda)$ and $\omega$ expresses to which extend $\sigma$ is not bimultiplicative.\\

Let $X=\C_{\alpha_1}\oplus \cdots \otimes \C_{\alpha_n}$ be any object in $\CC$, spanned by generators $x_i$ in degree $\alpha_i\in\Gamma$. The Nichols algebra $\NicholsOf(X)$ depends on the braiding matrix $q_{ij}=\sigma(\alpha_i,\alpha_j)$, and in fact it only depends on $q_{ii}=Q(\alpha_i)$ and $q_{ij}q_{ji}=B(\alpha_i,\alpha_j)$ up to twist equivalence.  Our main example in this article is the following:

\begin{example}[Quantum group]\label{exm_NicholsQG}
    Let $\g$ be a semisimple complex finite-dimensional Lie algebra with Cartan subalgebra $\h$ and Killing form $(-,-)$, and let $\Gamma=\h^*$, so $\CC=\Vect_\Gamma$ is the category of weight spaces, which we endow with the braiding $q^{(\lambda,\mu)}$ for some $q\in \C^ \times$. Let $\alpha_1,\ldots,\alpha_n$ be a choice of simple roots, take the object $X=\C_{\alpha_1}\oplus\cdots\oplus \C_{\alpha_n}$, which has a braiding matrix $q_{ij}=q^{(\alpha_i,\alpha_j)}$. Then the Nichols algebra $\NicholsOf(X)$ is the quantum Borel part of the (small) quantum group $u_q(\g)^+$. For $\g=\sl_2$ and $q^2 \neq 1$ this reproduces the Hopf algebra in Example \ref{exm_NicholsRank1}. For $\g=\sl_2$ and $q^2 = 1$ the Nichols algebra is $\C[x]$.
\end{example}

    Note that the Nichols algebra is independent on whether it is realized in the category $\CC=\Vect_\Gamma$ for $\Gamma$ some finite quotient of $\h^*$, corresponding to the Cartan part $u_q(\g)^0$ of the small quantum group,  or in the case $\Gamma=\h^*$ corresponding to the Cartan part $u^H_q(\g)^{0}$ of the unrolled quantum group. Note also that in Nichols algebra literature one usually considers a category of $\Gamma$-Yetter-Drinfeld modules instead of $\Vect_\Gamma^Q$ with a chosen quadratic form. The two are related by the canonical braided tensor functor
    $$\Vect_\Gamma^Q \longrightarrow \cZ(\Vect_\Gamma)=\YD{\C[\Gamma]}$$

    A second relevant class of examples comes from parabolic situations and is intimately related to the reflection theory of Nichols algebras, as discussed for example in \cite{CL17}:

    \begin{theorem}\label{thm_parabolic}
        Let $J\subset \{1,\ldots,n \}$ be a subset of simple roots and $X=X_J\bigoplus X_{\bar{J}}$. Then there is an isomorphism of Hopf algebras resp. an equivalence of tensor  categories 
        $$\NicholsOf(X)\cong \NicholsOf(\hat{X}_{\bar{J}})\rtimes \NicholsOf(X_J)$$
        $$\Rep(\NicholsOf(X))(\CC)\cong \Rep(\NicholsOf(\hat{X}_{\bar{J}}))(\Rep(\NicholsOf(X_J)(\CC))$$
        
        where $\hat{X}_{\bar{J}}$ is a Yetter-Drinfeld module over $\NicholsOf(X_J)$ (in the sense of Definition \ref{def_YD}) generated  by $X_{\bar{J}}$ under the adjoint action $\mathrm{ad}_{\NicholsOf(X_J)}(X_{\bar{J}})$. 
    \end{theorem}
    Differently spoken,  $\NicholsOf(\hat{X}_{\bar{J}})$ is a Nichols algebra in the category of representations of the quantum group associated to $J$. The root system of this Nichols algebra is given by a restriction of hyperplane arrangements, see \cite{CL17} Section 3. Note that even in the Lie algebra case the resulting root system can be a generalized root system, because the parabolic subsystem is in general not stable under reflections. See in \cite{CL17} the examples $A_2$ in Section 4.4., the example $E_7$  with a new root system in  Section 4.5 and the thoroughly discussed example $D(2,1|\alpha)$ in Section 5.3. 
    
    \begin{example}[$\sl(n|1)$ over $\sl_n$]\label{exm_parabolic}
        The Nichols algebra $\NicholsOf(X)$ of the object 
        $$X=X_{\alpha_1}\oplus \cdots \oplus X_{\alpha_{n-1}}\oplus X_{\alpha_n}$$ 
        of rank $n$ with braiding matrix 
        $$q_{ij}=\begin{pmatrix}[ccccc|c] 
        q^2 & q^{-1} & 1        &\cdots & 1 & 1 \\
        q^{-1} & q^2 & q^{-1}   &\cdots & 1 & 1 \\
        1 & q^{-1} & q^2 & \cdots & 1 & 1 \\
        \vdots & \vdots & \vdots & & \vdots & \vdots \\
        1 & 1 & 1 & \cdots & q^{2} & q^{-1} \\  
        \hline
        1 & 1 & 1 & \cdots & q^{-1}\vphantom{X^{X^X}} & -1 \\  
        \end{pmatrix}$$
        has a root system of type $A_{n-1|1}$ with ${n+1\choose 2}$ positive roots. 
        
        The roots $J=\{\alpha_1,\ldots,\alpha_{n-1}\}$ generate the root system $A_{n-1}$ and there are $n$ fermionic roots $\alpha_n,\,\alpha_n+\alpha_{n-1},\ldots$. Under the adjoint action these roots span a Yetter-Drinfeld module $\hat{X}_{\bar{J}}$ of rank $1$ (i.e.~simple) and dimension $n$ over the Nichols algebra 
        $$\NicholsOf(X_J)=u_q(\sl_n)^+,$$
        Correspondingly $\hat{X}_{\bar{J}}$  is the $n$-dimensional standard representation of $u_q(\sl_n)$. The Nichols algebra of $\hat{X}_{\bar{J}}$ is isomorphic to the exterior algebra of dimension $2^n$
        $$\NicholsOf(\hat{X}_{\bar{J}})=\bigwedge \hat{X}_{\bar{J}}.$$
    \end{example}

\subsection{Recognition statements}

Suppose $\Nichols$ is a Hopf algebra in a braided tensor category $\CC$, for now $\Vect_{\Gamma}^Q$ for $\Gamma=\C^n$ or more generally any abelian group $\Gamma$ and a quadratic form $Q$. In the context of our article, we encounter such $\Nichols$ from a splitting tensor functor $\BB\leftrightarrows \CC$. We want to give criteria that ensure that $\Nichols$ is indeed a Nichols algebra $\NicholsOf(X)$. We have a surprisingly simple and applicable argument by $\Gamma$-degrees: 

\begin{definition}[Sufficiently unrolled]\label{def_sufficientlyUnrolled}
Let $\NicholsOf(X)$ be the Nichols algebra of an object $X=\C_{\bar{\alpha}_1}\oplus\cdots \oplus \C_{\bar{\alpha}_n}$ in $\Vect_\Gamma$. The Nichols algebra is by definition $\mathbb N^n$-graded with generators in degree $\alpha_1,\ldots,\alpha_n$ which are the formal basis of $\mathbb N^n$. Denote $\mathrm{supp}(\NicholsOf(X))\subset \mathbb N^n$ the subset of degrees appearing. On the other hand the Nichols algebra as an object in $\Vect_\Gamma$ is $\Gamma$-graded, and we denote by a bar the additive map $\mathbb N^n\to \Gamma$ sending $\alpha_k$ to $\bar{\alpha}_k$.

We call the Nichols algebra $\NicholsOf(X)$  in $\Vect_\Gamma$ sufficiently unrolled, if $\bar{\alpha}+\bar{\beta}=\bar{\gamma}$ implies $\alpha+\beta=\gamma$ for all $\alpha,\beta,\gamma \in \mathrm{supp}(\NicholsOf(X))$.
\end{definition}

The main point of this definition is 

\begin{lemma}\label{lm_isNicholsFree}
Let $\Nichols$ be any Hopf algebra in $\Vect_\Gamma$ which is isomorphic to $\NicholsOf(X)$ as object in $\Vect_\Gamma$. If the latter is sufficiently unrolled then we have a Hopf algebra isomorphism $\NicholsOf(X)\to \Nichols$. 
\end{lemma}
\begin{proof}
The condition of being sufficiently unrolled implies that the $\Gamma$-grading of $\Nichols$ as algebra and coalgebra lifts to an $\mathbb N^n$-grading. Such a grading implies for the generators of $X_{\alpha_k}$ in $\Nichols$ the coproduct
$$\Delta(x_k)=x_k\otimes 1+1\otimes x_k.$$
Hence we have a Hopf algebra map $\NicholsOf(X)\to \Nichols$, which by the universal property of the Nichols algebra is injective. Since both sides are $\mathbb N^n$-graded vector spaces with finite-dimensional homogeneous components (as the tensor algebra) it is already an isomorphism.
\end{proof}

Using deeper arguments on Nichols algebras and some knowledge from the classification of finite-dimensional Nichols algebras of diagonal type we have 

\begin{theorem}[\cite{AKM15} Theorem 6.3.]
Any filtered Hopf algebra $\Nichols$ with associated graded Hopf algebra $\mathrm{gr}(\Nichols)\cong \NicholsOf(X)$ with $\NicholsOf(X)$ a finite-dimensional Nichols algebra with diagonal braiding is actually graded and $\Nichols\cong \NicholsOf(X)$.
\end{theorem}

This gives another criterion that we can apply whenever we know the category $\BB$ as abelian category. Note that the previous theorem is very strong and one could also use other types of assumptions to make it applicable.

\begin{lemma}\label{lm_isNichols}
Assume the degrees $\bar{\alpha}_1,\ldots,\bar{\alpha}_n$ are independent in $\Gamma$ in the sense that $0,\bar{\alpha}_1,\ldots,\bar{\alpha}_n$ are all different in $\Gamma$. Assume that as algebras $\Nichols\cong \NicholsOf(X)$. Then also as Hopf algebras $\Nichols\cong \NicholsOf(X)$.
\end{lemma}
\begin{proof}
As a Hopf algebra the Nichols algebra $\NicholsOf(X)$ is $\mathbb{N}^n$-graded with $n=\dim(X)$, and the total degree gives in particular a $\mathbb{N}$-grading with $1$ in degree $0$ and $X$ in degree $1$. As a Hopf algebra in $\Vect_\Gamma$ both $\NicholsOf(X)$ and $\Nichols$ are graded by $\Gamma$. Our assumption on the $\Gamma$-degrees $\alpha_i$ implies that $\Delta(x_i)=1\otimes x_i+x_i\otimes 1$ plus elements in higher total $\mathbb{N}$-degree. Since $\Delta$ is multiplicative, this implies that $\Nichols$ is a $\mathbb{N}$-filtered Hopf algebra. \\

 We consider the dual Hopf algebra $\Nichols^*$ which is then a graded coalgebra and filtered algebra with associated graded Hopf algebra $\mathrm{gr}(\Nichols^*)\cong\NicholsOf(X^*)$. The Nichols algebra has the universal property that the canonical pairing $\NicholsOf(X)\otimes\NicholsOf(X^*)\to 1$ is a nondegenerate Hopf pairing. For a finite-dimensional Nichols algebra we have hence $\NicholsOf(X)^*\cong \NicholsOf(X^*)$. Now by the preceeding theorem we have already $\Nichols^*\cong \NicholsOf(X^*)$.\\
 
\end{proof}

\begin{example}
Consider again the Nichols algebra $\NicholsOf(X)=\C[x]/x^p$ with $x$ in degree $\bar{\alpha}$ having self-braiding $Q(\bar{\alpha})=e^{2\pi\i/p}$ in $\Vect_{\Gamma}^Q$, where $\Gamma=\Z\bar{\alpha}/p\bar{\alpha}$ is the smallest possible quotient group. Let $\Nichols$ be a Hopf algebra of dimension $p$ in $\Vect_{\Gamma}^Q$ containing $x$ in degree $\bar{\alpha}$ and assume as algebras $\Nichols\cong \C[x]/x^p$. The coproduct is of the form
$$\Delta(x)=x\otimes 1 + 1\otimes x+\sum_{i+j=p+1} a_{ij} (x^i\otimes x^{j}).$$


The dual Hopf algebra $\Nichols^*$ is spanned by $(x^n)^ *,\;0\leq n<p$, has a known coproduct with $\Delta^*(x^*)=x^*\otimes 1+1\otimes x^*$ and counit $\epsilon^*(\phi)=\phi(1)$, and a convolution product with possible extra terms depending on $a_{ij}$. Hence $\Nichols^*$ is a $\Z$-filtered algebra graded by the cosets in $\Z_p$, that is a product of elements in degree $n$ and $m$ is an element in degree $n+m$ with possible extra terms in degree $n+m-p\mathbb{N}$. In particular the convolution power $(x^*)^p$ may have a term proportional to $1^*$. But this would violate multiplicativity of $\epsilon^*$. Hence the coproduct has no additional terms and $\Nichols\cong\NicholsOf(X)$.\\

\end{example}

\begin{remark}
It may seem puzzeling that the same problem is much easier over $\Vect_{\C^n}$ than $\Vect_\Gamma$. From the algebra side it is not so surprising, because an algebra is forced to be graded. From the vertex algebra side however it seems to be a trivial fact that the lattice vertex algebra contains a Heisenberg vertex algebra, and alongside the triplet algebra contains a singlet algebra. But on closer look this fact contains more structure, which might be most transparently formulated as an additional $\h$-action or $\h^*$-grading: The free field realization $\V$ carries an additional action of the Lie algebra $\h$, and orbifolding with it brings us from $\Vect_\Gamma$ to $\Vect_{\C^n}$. This action comes from the Heisenberg fields $\partial\phi_i$, which are in the kernel of the short screenings $\W$, so these are still acting on $\W$. This additional structure apparently suffices to establish the additional grading on $\Nichols$ in general.
\end{remark}

\subsection{Splitting statements}

In case we have no a-priori knowledge of a splitting functor, but in which we can realize $\BB=\Rep(B)$ for a Hopf algebra $B$, we can analyze $B$  in a similar way as the previous section. Assume $B\supset C$ are Hopf algebras and $\BB=\Rep^{wt}(B)$ and $\CC=\Rep^{wt}(C)$, where $\Rep^{wt}$ denotes representations that are finite-dimensional with $C$ acting semisimply and $C$ is cosemisimple. For example we can take $\CC=\Vect_{\C^n}$ with $C=\C[H_1,\ldots,H_n]$  or $\CC=\Vect_\Gamma$ with $C=\C[\Gamma]$ a finite group ring. 

Again the far easier case is the sufficiently unrolled case, in particular $\CC=\Vect_{\C^n}$.

\begin{corollary}\label{cor_isNicholsRadfordFree}
Assume that $B$ under the adjoint action of $C$ is sufficiently unrolled in the sense of Definition \ref{def_sufficientlyUnrolled}. Then we have a splitting tensor functor and  $B=\Nichols\rtimes C$.
\end{corollary}
\begin{proof}
$B$ is a $\mathbb{N}$-graded Hopf algebra with cosemisimple Hopf algebra $C$, so it is coradically graded. Hence $B$ and thus it is a Radford biproduct. \\

Alternatively we could argue that by grading that  all simple modules are local and we can apply Lemma \ref{lm_socle}.
\end{proof}

We can proceed to a more general case if we know the algebra structure:

\begin{lemma}\label{lm_isRadfordNichols}
Assume $B$ is a Hopf algebra, which is isomorphic as an algebra to $\NicholsOf(X)\rtimes C$ as finite-dimensional algebras in $\CC$, and which contains $C$ as a Hopf subalgebra, and assume the degrees $\bar{\alpha}_1,\ldots,\bar{\alpha}_n$ of a basis $x_1,\ldots,x_n$ of $X$ are independent in the sense that never  $\bar{\alpha}_i+\bar{\alpha}_j=\bar{\alpha}_k$ or $0$. Then $B=\NicholsOf(X)\rtimes C$ as Hopf algebras.
\end{lemma}
\begin{proof}
Consider the (Hopf-)dual Hopf algebra $B^*$. Our assumption amounts to $B^*=\Nichols^*\rtimes C^*$ as coalgebra and our additional assumption amounts to $\Delta(x_i)\in C\otimes x_i+x_i\otimes C$ plus additional terms in total degree $\geq 2$. Hence $B$ is a filtered Hopf algebra with associated graded Hopf algebra $\mathrm{gr}B=\Nichols\rtimes C$. Hence all simple modules are local and we can apply Lemma \ref{lm_socle}. 
\end{proof}

\begin{remark}\label{rem_lifting}
For not sufficiently unrolled cases the question above is precisely the question of liftings of a Nichols algebra \textup{\cite{AI19}}. For example the preceeding Lemma also follows from the result in  that all nontrivial liftings of a Radford biproduct of a finite-dimensional Nichols algebra $\NicholsOf(X^*)$ and a group ring $C^*$ are $2$-cocycle deformations of the trivial lifting, hence the associated category of comodules are equivalent, which relies on the lifting method and case-by-case analysis on the classification of finite-dimensional diagonal Nichols algebras. However, in our case we already assume that the algebra structure is still graded, which seems to be close to the assertion that $C^*$ is a $2$-cocycle deformation. The first counterexamples are found for liftings of Nichols algebras associated to Jordan planes \textup{\cite{AAH18,AAH22}}.
\end{remark}

\renewcommand{\catM}{\UU}

\section{The relative center in the splitting case}\label{sec:RelativeDrinfeldCentersGiveLocalization}

\subsection{Yetter-Drinfeld modules in a braided tensor category}

For $\BB=\Rep(B)$ the category of representations of a Hopf algebra $B$, the Drinfeld center $\cZ(\BB)$ was initially studied as the category of representations of a particular quasitriangular Hopf algebra called Drinfeld double $D(B)$, and a different formulation are the $B$-$B$-Yetter-Drinfeld modules (or crossed modules), which are modules and comodules of $C$ with a certain compatiblility condition.  They have been generalized in \cite{Besp95} to Hopf algebras $\Nichols$ inside a braided tensor category $\CC$, see also \cite{BLS14} Section 2 and 3: 
\begin{definition}\label{def_YD}
Let $\Nichols$ be a Hopf algebra in a braided category $\CC$. A $\Nichols$-$\Nichols$-\emph{Yetter-Drinfeld module} $X$ is an object $X$ in $\CC$ with the structure of a $\Nichols$-module and $\Nichols$-comodule in $\CC$, such that the following compatibility condition holds
\begin{center}
                \begin{grform}
                        \begin{scope}[scale = 0.5]
                                \dMult{0.5}{1.5}{1}{-1.5}{\grau}
                                \dMult{0.5}{2.5}{1}{1.5}{\grau}
                                \dAction{2.5}{1.5}{0.5}{-1.5}{\grau}{black}
                                \dAction{2.5}{2.5}{0.5}{1.5}{\grau}{black}
                                \vLine{0.5}{1.5}{0.5}{2.5}{\grau}
                                \vLine{2.5}{1.5}{1.5}{2.5}{\grau}
                                \vLineO{1.5}{1.5}{2.5}{2.5}{\grau}
                                \vLine{3}{1.5}{3}{2.5}{black}
                        \end{scope}
                        \draw (0.5 , -0.3) node {$A$};
                        \draw (0.5 , 2.3) node {$A$};
                        \draw (1.5 , -0.3) node {$X$};
                        \draw (1.5 , 2.3) node {$X$};
                \end{grform}
                =
                \begin{grform}
                        \begin{scope}[scale = 0.5]
                                \vLine{2}{0}{1}{1}{black}
                                \dMultO{0}{1}{1.5}{-1}{\grau}
                                \dAction{0}{1}{1}{1}{\grau}{black}
                                \dAction{0}{3}{1}{-1}{\grau}{black}
                                \dMult{0}{3}{1.5}{1}{\grau}
                                \vLineO{2}{4}{1}{3}{black}
                                \vLine{1.5}{1}{1.5}{3}{\grau}
                        \end{scope}
                        \draw (0.45 , -0.3) node {$A$};
                        \draw (0.45 , 2.3) node {$A$};
                        \draw (1 , -0.3) node {$X$};
                        \draw (1 , 2.3) node {$X$};
                \end{grform}
\end{center}
 
\end{definition}

The category $\YD{\Nichols}(\CC)$ consists of  Yetter-Drinfeld modules and of $\Nichols$-linear and $\Nichols$-colinear morphisms. It becomes a tensor category with the usual tensor product of $\Nichols$-modules and $\Nichols$-comodules $X\otimes Y$ (see Definition \ref{def_RepB}). This category also admits a braiding
        \begin{align*}
                c_{(X, m_X, \delta_X), (Y, m_Y, \delta_Y)} &: X \otimes Y \rightarrow Y \otimes X \\
\intertext{which is given on objects $(X, m_X, \delta_X)$ and $(Y, m_Y, \delta_Y)$ by}
                c_{(X, m_X, \delta_X), (Y, m_Y, \delta_Y)} &:= (\rho_Y \otimes \id_X ) \circ (\id_\Nichols \otimes
c_{X,Y} ) \circ (\delta_X \otimes \id_Y )\,\, .
        \end{align*}
and is invertible if $\Nichols$ has a bijective antipode
\[
               (c_{(X, m_X, \delta_X), (Y, m_Y, \delta_Y)} )^{-1} := c^{-1}_{X,Y} \circ (m_Y \otimes
\id_X ) \circ (c^{-1}_{\Nichols,Y} \otimes \id_X )
                \circ (\id_Y \otimes S^{-1} \otimes \id_X ) \circ (\id_Y \otimes
\delta_X ).
\]
 If $\CC$ is rigid, then the dual object in $\CC$ with the standard dual action and coaction by the antipode gives a dual object in $\YD{\Nichols}(\CC)$. The structure is summarized in the following statement proven in \cite{Besp95}:
\begin{theorem}
Let $\Nichols$ be a Hopf algebra in $\CC$. 
The Yetter-Drinfeld modules over $\Nichols$ in $\CC$ have
a natural structure of a
braided monoidal category $\YD{\Nichols}(\CC)$. If $\CC$ is rigid, then so is $\YD{\Nichols}(\CC)$.
\end{theorem}

This gives an explicit expression for the relative center in the splitting case:        

\begin{lemma}[\cite{Lau20} Prop.  3.36]\label{lm_relcenterSplitting}
    Suppose $\BB=\Rep(\Nichols)(\CC)$, then the relative center is
    $$\cZ_\CC(\BB)=\YD{\Nichols}(\CC)$$
\end{lemma}

In particular in the cases covered by Section \ref{sec_Schauenburg} we have 
\begin{corollary}\label{cor_schauenburgSplit}
    Let $A$ be a commutative algebra in $\UU$ that admits a splitting functor, which means $\UU_A=\Rep(\Nichols)(\CC)$ and $\UU_A^0=\CC$, and assume Assumption \ref{assumption}, 
    then $\UU\cong \YD{\Nichols}(\CC)$. 
\end{corollary}

\begin{center}
\begin{tcolorbox}[colback=white, width=0.72\linewidth]
$$
 \begin{tikzcd}[row sep=10ex, column sep=15ex]
  \UU=\YD{\Nichols}(\CC) \arrow{r}{\induceA}  
   & \arrow[shift left=2]{l}{\forgetA}  
      \arrow[shift left=2, dashed]{d}{} 
   \Rep(\Nichols)(\CC)&
   \hspace{-2.5cm}\textcolor{darkgreen}{\cong}\;\UU_A,\otimes_A \\
   &\arrow[hookrightarrow]{u}{}\arrow[below]{ul}{\coVerma}
   \mathbf{\CC}&\hspace{-2.8cm}\textcolor{darkgreen}{\cong}\;\UU_A^0
    \end{tikzcd}
$$
\end{tcolorbox}
\end{center}

\begin{definition}\label{def_Verma}
The Verma module $\bV_0$ in $\YD{\Nichols}(\CC)$ can be defined as $\Nichols$ with the adjoint action and the regular coaction. It is a commutative coalgebra in $\YD{\Nichols}(\CC)$. For $\CC$ rigid, the dual  Verma module $\bV_0^*$ in $\YD{\Nichols}(\CC)$ is an algebra. 
\end{definition}
This is the adjoint algebra in the sense of \cite{Sh19} in the relative version \cite{LW2,Mo22}. For highest-weight theory in the context of diagonal Nichols algberas see \cite{Vay19}.\\

In \cite{GLR23} we get in the splitting case the following converse of Corollary \ref{cor_schauenburgSplit}:
\begin{theorem}
Let $\Nichols$ be a Hopf algebra in a braided tensor category $\CC$, let $\BB=\Rep(\Nichols)(\CC)$ and $\UU=\cZ_\CC(\BB)=\YD{\Nichols}(\CC)$. Then for the commutative algebra $A=\bV_0$  in $\UU$ we have $\UU_A\cong\BB$ as tensor categories and $\UU^0_A\cong\CC$ as braided tensor categories. In particular the functor $\BB\to \CC$ forgetting the action of $\Nichols$ is a splitting tensor functor to the local modules.
\end{theorem}
More precisely, the category equivalence $\YD{\Nichols}(\CC)_A\textcolor{darkgreen}\;{\cong} \Rep(\Nichols)(\CC)$ is given by sending an $A$-module to its $A$-invariants, and reversed by inducing up a $\Nichols$-module  by a certain universal object $U$ in the left category with additional right $\Nichols$-action.\\

\subsection{Realizing Hopf algebras}

The preceeding categories can be realized as (quasi-) Hopf algebras if $\CC$ is realized by a quasi-triangular (quasi-) Hopf algebra $C$. This appears in different language and level of generality since the beginning of quantum groups, see \cite{Lusz93,Maj94, Som96} in a quantum group setting, \cite{HY10,AY13} in a Nichols  algebra setting and \cite{Lau20, LW1} in a terms of relative Drinfeld centers, and notably example of a Lie superalgebra in \cite{SL23}. If the base category has a nontrivial associator category one can similarly obtain a quasi-Hopf algebra, see  \cite{GLO18} Section 5.5.


We will now give for the case of $\CC=\Vect_\Gamma$ with trivial associator and $\Nichols$ a Nichols algebra a direct construction of a realizing Hopf algebra $U$ from the categorical picture. The description matches essentially \cite{LW1} Proposition 5.15 but is suitable for infinite $\Gamma$. 

\begin{example} \label{exm_realizedQuantumGroup}
Assume the braided tensor category $\CC$ is realized as finite-dimensional semisimple modules $\Rep^{wt}(C)$ of a commutative cocommutative Hopf algebra $C$, so $\CC=\Vect_\Gamma$ for a possibly infinite abelian group $\Gamma=\mathrm{Spec}(C)$ with a braiding $\sigma$.
    
    Let $\Nichols=\NicholsOf(X)$ be a finite dimensional Nichols algebra for an object $X=\bigoplus_{k=1}^n \C_{\gamma_k}$ in $\CC$ spanned by $x_1,\ldots, x_n$ in degrees $\gamma_1,\ldots,\gamma_n\in \Gamma$, so $g.x_i=\gamma_i(g)x_i$ for all $g\in C$. Assume there are corresponding grouplike elements $g_i\in C$ with $\sigma(\gamma_i,\gamma)=\gamma(g_i)$ and $\bar{g}_i\in C$ with $\sigma(\gamma,\gamma_i)=\gamma(\bar{g}_i)$. \\

    We define a Hopf algebra $B=\NicholsOf(X)\otimes C$ generated by $C$ and $\NicholsOf(X)$ and together with the commutation relation and modified coproduct and antipode
    \begin{align*}
        gx_i &= \gamma_i(g^{(1)})\cdot x_i g^{(2)}\\
        \Delta(x_i) &= g_i\otimes x_i + x_i\otimes 1\\
        S(x_i)&=-g^{-1}x_i
    \end{align*}
    Here $g^{(1)}$ and $g^{(2)}$ are defined via the co-product, that is $\Delta(g)= g^{(1)}\otimes g^{(2)}$ in Sweedler notation.
    We also define a Hopf algebra $B^*=\NicholsOf(X^*)\otimes C$ where $X^*=X$ with generators $x_i^*$ in degree $\gamma_i(S(-))$ and modified coproduct and antipode
     \begin{align*}
        gx_i^* &= \gamma_i(S(g^{(1)}))\cdot x_i^* g^{(2)}\\
        \Delta(x_i^*) &= \bar{g}_i\otimes x_i^* + x_i^*\otimes 1\\
        S(x_i^*)&=-\bar{g}^{-1}x_i^*
    \end{align*}
    We finally define the Hopf algebra $U=\NicholsOf(X)\otimes \NicholsOf(X^*)\otimes C $ with all previous relations and the {linking relation} 
    \begin{align*}
    x_j^*x_i - \gamma_j(g_i) x_ix_j^* 
    &= \delta_{ij}\big(1- \bar{g}_i g_i\big)\\
    \end{align*}

\begin{lemma}\label{lm_realizedQuantumGroup}
    The relations above define a Hopf algebra and we have an equivalence of braided tensor categories with a nondegenerate braiding
    $$\Rep^{wt}(U)\cong \YD{\Nichols}(\CC)$$
\end{lemma}
\begin{proof}
    We first check that $B, B^*$ and $U$ are Hopf algebras: For $B$ we have to check coassociativity and antipode condition on the generators $x_i$, which clearly holds, and check that $\Delta$ is compatible with the commutation relation, i.e.~$\Delta$ extended by the bialgebra axiom applied to both sides of the relation is equal up to the relation (here using commutativity and cocommutativity)
    \begin{align*}
    \Delta(gx_i)&=g^{(1)}g_i\otimes g^{(2)}x_i+g^{(1)}x_i\otimes g^{(2)} \\
    \Delta(\gamma_i(g^{(1)})x_i g^{(2)}) 
    &=\gamma_i(g^{(1)})g_i g^{(2)}\otimes x_ig^{(3)}
    +\gamma_i(g^{(1)})x_ig^{(2)}\otimes g^{(3)}
    \end{align*}
    Since the commutation relation and the coproduct relation on a product on $1$-dimensional $C$-modules in degrees $\mu,\nu$ reads 
    $$g.(x_i.m_\mu)=(g^{(1)}.x_i)(g^{(2)}.m_\mu) \qquad 
    \Delta(x_i)(m_\mu\otimes n_\nu)=(x_i.m_\mu)\otimes n_\nu+\sigma(\gamma_i,\mu) m_\mu\otimes (x_i.n_\nu)$$
    As a consequence, we have almost by definition an equivalence of tensor categories 
    $$\Rep(B)\cong \Rep(\Nichols)(\CC)$$
    The computation for $B^*$ is completely analogous. 

    For $U$ we have to check that $\Delta$ is compatible with the linking relation, because extending $\Delta$ by the bialgebra axiom to the left-hand side of the linking relation matches the general coproduct for the right hand side (a so-called trivial skew-primitive for $g=\bar{g}_jg_i$):
    \begin{align*}
    \Delta(x_j^* x_i-\gamma_j(g_i) x_ix_j^*)
    &=(\bar{g}_j \otimes x_j^*+x_j^*\otimes 1) 
    (g_i \otimes x_i+x_i\otimes 1)\\
    &-\gamma_j(g_i)(g_i \otimes x_i+x_i\otimes 1)
    (\bar{g}_j \otimes x_j^*+x_j^*\otimes 1)  \\
    &=\bar{g}_jg_i \otimes (x_j^* x_i-\gamma_j(g_i) x_ix_j^*)
    +(x_j^* x_i-\gamma_j(g_i) x_ix_j^*)\otimes 1 \\
    &+(\bar{g}_jx_i-\gamma_j(g_i)x_i\bar{g}_j)\otimes x_j^*
    +(x_j^*g_i-\gamma_j(g_i) g_ix_j^*)\otimes x_i
    \\
    \Delta(1-g)&=g\otimes (1-g)+ (1-g)\otimes 1
    \end{align*}
    
    We now turn to the asserted category equivalence: Define a functor $\YD{\Nichols}(\CC)\to \Rep(U)$ by considering $M\in \Rep(\Nichols)(\CC)$ as a $B$-module and turning the $\Nichols$-coaction $m\mapsto m^{(-1)} \otimes m^{(0)}$ into a $\NicholsOf(X^*)$-action on $M$ using the nondegenerate Hopf pairing $\langle-,-\rangle:\NicholsOf(X^*)\otimes \NicholsOf(X)\to \C$ uniquely determined by $ \langle x_j^*,x_i\rangle=\delta_{i,j}$, whose existence is one of the equivalent characterizations of a Nichols algebra \cite{Lusz93}\cite{HS20}. So in particular on generators 
    $$x_j^*.m =\langle x_j^*,m^{(0)}\rangle m^{(1)}$$
    The coaction being inside the category $\CC$ i.e.~$g.m\mapsto g^{(1)}m^{(-1)} \otimes g^{(2)}m^{(0)}$ is equivalent to the relation between $g,x_j^*$ since
    \begin{align*}
    x_j^*.g.m &=\langle x_j^*,g^{(1)}.m^{(-1)}\rangle g^{(2)}m^{(0)}\\
    &=\langle S(g^{(1)})x_j^*,m^{(-1)}\rangle g^{(2)}.m^{(0)}\\
    &=\langle \gamma_i(g^{(1)})x_j^*,m^{(-1)}\rangle g^{(2)}.m^{(0)}\\
    &=\gamma_j(g^{(1)})g^{(2)}.x_j^*.m
    \end{align*}
    The braided Yetter-Drinfeld condition in Definition \ref{def_YD} is equivalent to the linking relation between $x_i,y_j$ because the two sides of the Yetter-Drinfeld condition evaluated on $x_i\otimes m$ for an element $m$ having degree $\lambda$ and a  generator $x_i$ having the coproduct $\Delta_\NicholsOf(x_i)=1\otimes x_i+x_i\otimes 1$, and the result paired with another generator $x_j^*$, has on both sides of the equation only two contributions for $i=j$  
    \begin{align*}
        \langle x_j^*,x_i\rangle m 
        + \sigma(\gamma_i,\gamma_j) \langle x_j^*,m^{(-1)}\rangle x_i.m^{(0)}
        &= \langle x_j^*,x_i\rangle \sigma(\lambda,\gamma_i)\sigma(\gamma_i,\lambda)m 
        +  \langle x_j^*,(x_i.m)^{(-1)}\rangle (x_i.m)^{(0)}
    \end{align*}
    for braiding factors $\sigma(\gamma_i,\gamma_j)=\gamma_j(g_i)$ and $\sigma(\gamma_i,\lambda)=\lambda(g_i)$ resp.  $\sigma(\lambda,\gamma_i)=\lambda(\bar{g}_i)$ for $m$ in degree $\lambda$. Rewritten as action of $x_j^*$ this is
    \begin{align*}
        \delta_{ij} m 
        + \gamma_j(g_i) x_i.x_j^*.m
        &= \delta_{ij} \lambda(\bar{g}_i)\lambda(g_i) m 
        + x_j^*.x_i.m \\
   \end{align*}
   and this gives the asserted linking relation in $U$. We also check that this functor extends to a tensor functor: For $\Rep(\Nichols)(\CC)\to\Rep(B)$ we have already checked this. The  tensor product of $\Nichols$-comodules is defined using the algebra structure in $\Nichols$  
   $$(m_\mu\otimes n_\nu)^{(-1)}\otimes (m_\mu\otimes n_\nu)^{(0)}
   =
   m_{\mu'}^{(-1)}n_{\nu'}^{(-1)}\sigma(\mu'',\nu')\otimes (m_{\mu''}^{(0)}
   \otimes n_{\nu''}^{(0)})
   $$
   so after applying the Hopf pairing with a generator with primitive coproduct $\Delta_\NicholsOf(x_j^*)=1\otimes x_j^*+x_j^*\otimes 1$
   \begin{align*}
    x_j^* (m\otimes n)
   &=\langle x_j^*, m_{\mu'}^{(-1)}n_{\nu'}^{(-1)}\sigma(\mu'',\nu')\rangle \; (m_{\mu''}^{(0)}
   \otimes n_{\nu''}^{(0)})\\
   &=\big(\langle 1, m_{\mu'}^{(-1)}\rangle\langle x_j^*, n_{\nu'}^{(-1)}\rangle \sigma(\mu'',\nu') 
  +\langle x_j^*, m_{\mu'}^{(-1)}\rangle\langle 1, n_{\nu'}^{(-1)}\rangle \sigma(\mu'',\nu')
  \big)
     \; (m_{\mu''}^{(0)}
   \otimes n_{\nu''}^{(0)})\\
   &=  \sigma(\mu,\gamma_j) \; (m_{\mu}
   \otimes x_j^*.n_{\nu}) +\sigma(\mu-\gamma_j,0) \; (x_j^*.m_{\mu}
   \otimes n_{\nu}) \\
   &= (\bar{g}_j\otimes x_j^*+x_j^*\otimes 1 ).(m\otimes n)
   \end{align*}
    

    The same reasoning gives an inverse functor $\Rep(U)\to \YD{\Nichols}(\CC)$, because the pairing $\langle -,-\rangle$ is invertible for $\Nichols$ finite-dimensional, and the relations between $x_i^ *,g$ and $x_i,x_j^*$ precisely show on generators the coaction being in the category and the Yetter-Drinfeld condition. We thus finally conclude that we have an equivalence of tensor categories.

 A concrete $R$-matrix in some closure of $U$ making this a braided tensor category can easily be read off from the braiding in $\YD{\Nichols}$.
    \end{proof}
\end{example}

\subsection{Example related to quantum groups}\label{sec_quasiquantum}

We review how the quantum groups and the related quasi-quantum groups for even order roots of unity can be obtained as a relative Drinfeld center. In the language of Drinfeld doubles, this view has been around since the beginning of quantum groups:
Let $\g$ be a complex finite-dimensional semisimple Lie algebra and $q\in\C^\times$. Let $\CC=\Vect_\Gamma$ with nondegenerate quadratic form $Q$ be a modular tensor category containing an object $X=\C_{\alpha_1}\oplus \cdots \oplus \C_{\alpha_n}$ with braiding matrix $q^{(\alpha_i,\alpha_j)}$. This can be $\Gamma=\h^*$, but also a finite quotient group $\Gamma$. Note however that for an even root of unity $\CC$ this requires a  associator. 

\begin{example}\label{ex:quantumsl2asrelativecenter}
    Let $\g=\sl_2$ and $q=e^{\pi\i/p}$ a $2p$-to root of unity. Then $\Gamma=\Z_{2p}$ can be endowed with the quadratic form $Q(k)=e^{\pi\i\;k^2/2p}$ which factorizes over $\Z_{2p}$ and the associated bimultiplicative form $B(k,l)=e^ {2\pi\i \;kl /2p}$ is a nondegenerate form. But any representing abelian $3$-cocycle $(\sigma,\omega)$ describing braiding $\sigma$ and associator $\omega$ will have $\omega\neq 1$. For example $\sigma(k,l)=e^ {\pi\i \;kl /2p}$ is not a bimultiplicative map.
\end{example}

Now we have discussed in Example \ref{exm_NicholsQG} the Nichols algebra $\NicholsOf(X)$ matching the quantum Borel part $u_q(\g)$. Let us define 
\begin{align*}
    \BB&=\Rep(\Nichols)(\CC)\\ 
    \intertext{which has an obvious splitting tensor functor to $\CC$, and the relative Drinfeld center}
    \UU&=\cZ_\CC(\BB)
    \end{align*}
    {which is by construction a nondegenerate braided tensor category. Then $\UU$ is the category of representations of the factorizable (quasi-)quantum group $u_q(\g)$ respectively the unrolled quantum group $u^H_q(\g)$}.

From Section \ref{sec:RelativeDrinfeldCentersGiveLocalization} we then know that $\UU$ contains a commutative algebra $A=\Nichols^*$ with regular coaction and adjoint action. In our case, this is the dual of the induced representation $u_q(\g)\otimes_{u_q(\g)^{\geq0}}\C_\epsilon$ (Verma module), which is indeed an algebra in $\UU$.

We also match this with our explicit realization in Example \ref{exm_realizedQuantumGroup}: Let $\sigma(\lambda,\mu)=\sigma(\mu,\lambda)=q^{(\lambda,\mu)}$ and substitute $x_i=E_i,\;x_i^*=K_iF_i(q-q^ {-1}),\;g_i=\bar{g}_i=K_i$ which means $\lambda(K_i)=q^{(\alpha_i,\lambda)}$. Then the relations for $U$ given in Example \ref{exm_realizedQuantumGroup} are the Nichols algebra relations together with 
\begin{align*}
    K_\lambda E_i&=q^{(\alpha_i,\lambda)} E_i K_\lambda \\
    K_\lambda F_i&=q^{-(\alpha_i,\lambda)} F_i K_\lambda\\
    [E_i,F_j]&=\delta_{ij}\frac{K_i-K_i^{-1}}{q-q^{-1}}
\end{align*}
and coproduct
\begin{align*}
    \Delta(K_\lambda) &=K_\lambda\otimes K_\lambda \\
    \Delta(E_i) 
    &=K_i\otimes E_i + E_i\otimes 1 \\
     \Delta(F_i) 
    &=1\otimes F_i + F_i\otimes K_i^{-1} 
\end{align*}

More generally we have
\begin{definition}[Heckenberger quantum group]\label{def_HechenbergerQG}
Let $\CC=\Vect_\Gamma$ with (non-degenerate) braiding given by $\sigma(\lambda,\mu)$ and let $X=\bigoplus_{i=1}^n \C_{\alpha_i}$, then the Nichols algebra $\NicholsOf(X)$ is of diagonal type, depending on the braiding matrix $q_{ij}=\sigma(\alpha_i,\alpha_j)$ as discussed in Section \ref{sec_Nichols}. Possible choices leading to finite-dimensional Nichols algebras were classified in \textup{\cite{Heck09}}. Then we have a tensor category with (non-degenerate) braiding 

$$\UU=\YD{\NicholsOf(q_{ij})}(\CC)$$

Again we can obtain a realization of $\UU=\Rep(U)$ for a (quasi-)Hopf algebra $U$, as discussed in Example \ref{exm_realizedQuantumGroup}.
\end{definition}

This definition includes quantum groups $u_q(\g)$ and unrolled quantum groups $u_q^H(\g)$, quantum groups associated to Lie superalgebras of classic  type  \cite{RS92,SL23} and further examples.

\begin{example}[$u_q^B({\mathfrak{gl}_{1|1}})$]\label{exm_gl11QG}
(compare with e.g. \textup{\cite{RS92, GY22}})
Let $C=\C[A,B,(-1)^F,g]$ with coproducts of $A,B$ primitive and $(-1)^F,g$ grouplike with relations $((-1)^F)^2=1$ and $g=(-1)^F e^{\pi\i\,\hbar B}$ i.e.~
we consider the category $\Rep^{wt}(C)$ of finite-dimensional semisimple representations where the symbol $(-1)^F$ acts by $\pm 1$ and $g$ acts by $(-1)^f q^b$ with $q=e^{\pi\i\,\hbar}$ on the eigenspace of $(-1)^F,B$  with eigenvalue $(-1)^f,b$. Then $\Rep^{wt}(C) =\Vect_{\C^2}\boxtimes \mathrm{sVect}$ and we denote the $1$-dimensional objects by $\C_{(a,b,\pm)}$ according to their eigenvalues under $A,B,(-1)^F$. Let a braiding be given by
$$\C_{(a,b,\pm)}\otimes \C_{(a',b',\pm )}\stackrel{\sigma}{\longrightarrow} \C_{(a',b',\pm)}\otimes \C_{(a,b,\pm)}$$
$$\sigma((a,b,(-1)^f),(a',b',(-1)^{f'}))=
\exp\big(\pi\i\,\begin{pmatrix} a \\ b \\ f \end{pmatrix}^\mathsf{T}
\begin{pmatrix}
0 & -\hbar & 0\\
-\hbar & -\hbar^{2} & 0 \\
0 & 0 & 1
\end{pmatrix}
\begin{pmatrix} a' \\ b' \\ f' \end{pmatrix} \big)
$$

Let $X=x\C_{\gamma}$ with $\gamma=(-1,0,-1)$, then $g$ has been chosen such that $\sigma(\gamma,\lambda)=\sigma(\lambda,\gamma)=\lambda(g)$ holds. Since $\sigma(\gamma,\gamma)=-1$ the associated Nichols algebra is $\Nichols=\C[x]/(x^2)$. Correspondingly the quantum group associated to this situation by Lemma  \ref{exm_realizedQuantumGroup} is generated by $A,B,(-1)^F,g$ and $x,x^*$ with relations above and
\begin{align*}
x^2=(x^*)^2&=0 \\
[A,x]&=-x \\
[B,x]&=0 \\
(-1)^Fx&=-x(-1)^F \\
[A,x^*]&=x^* \\
[B,x^*]&=0 \\
(-1)^Fx^*&=-x^*(-1)^F \\
xx^*+x^*x &=1-g^2\\
\Delta(x)&=g\otimes x+x\otimes 1 \\
\Delta(x^*)&= g\otimes x^*+x^*\otimes 1
\end{align*}
To compare with \textup{\cite{GY22}} we set $G=N=A,E=B,K=g(-1)^F=q^B$ and $X=xK^{-1},Y=x^*(q^{-1}-q)$. Then the relations read
\begin{align*}
XY+YX &= (xK^{-1})(x^*(q^{-1}-q))+x^*(q^{-1}-q)(xK^{-1}) \\
&=K^{-1}(q^{-1}-q)(xx^*+x^*x) \\
&=\frac{K-K^{-1}}{q-q^{-1}}\\
\Delta(X)&=(-1)^F\otimes X+X\otimes K^{-1} \\
\Delta(Y)&= (-1)^FK\otimes Y+Y\otimes 1
\end{align*}
This Hopf algebra is a Radford biproduct between the $2$-dimensional Hopf algebra $\C[\Z_2]$ generated by $(-1)^F$ and the Hopf algebra in the category $\Rep(\C[\Z_2])=\mathrm{sVect}$ (a Hopf superalgebra) generated by the other generators as in \textup{\cite{GY22}}. The Radford biproduct leads to the additional $(-1)^F$ in the first factors of the coproduct of $X,Y$.

\end{example}



We finally consider the parabolic situation in Theorem \ref{exm_parabolic} and Example \ref{exm_parabolic}: Let $X=X_J\oplus X_{\bar{J}}$ and consider again the Nichols algebra $\NicholsOf(\hat{X}_{\bar{J}})$ over $\NicholsOf(X_{{J}})$. Then the braided tensor category we associate to this Nichols algebra $\NicholsOf(\hat{X}_{\bar{J}})$ coincides with the  braided tensor category we associate with the entire Nichols algebra $\NicholsOf(X)$, see \cite{Besp95} Proposition 4.2.3 resp. \cite{BLS14} Theorem 3.12. 

\begin{lemma}\label{lem_parabolic}
In the parabolic situation we have an equivalence of braided tensor categories
$$\YD{\NicholsOf(\hat{X}_{\bar{J}})}\big(\YD{\NicholsOf(X_{{J}})}(\CC)\big)
\;\cong\;
\YD{\NicholsOf(X)}(\CC)
$$
\end{lemma}

\section{The Singlet VOA and related VOAs}

We now consider the braided tensor category given by representations of the singlet vertex algebra $\catM=\Rep(\cM(p))$ for $p\in\mathbb{N}$, see \cite{CMY4, CMY}. Earlier important works on this VOA are \cite{Ad03, AM07, CM, CMR}.
As we will describe first, from free field realizations as kernel of screenings in Heisenberg VOAs $\cM(p)\subset \pi$ we know as in Section \ref{sec_exampleVOA} that there exists a commutative algebra $A$ with $\UU_A^0=\Vect_{\C}$, the  tensor category of $\C$-graded vector spaces with associated quadratic form $Q(k)=e^{\pi i (k/2p)^2}$.


\subsection{Free field realizations}\label{catMA}

Let $\alpha_+ = \sqrt{2p}, \alpha_- = -\sqrt{2/p}$, $\alpha_0 = \alpha_+ + \alpha_-$ and 
\[
\alpha_{r, s} = \frac{1-r}{2} \alpha_+ + \frac{1-s}{2}\alpha_-.
\]
Let $\pi$ be the rank one Heisenberg algebra with generator $\alpha(z)$ satisfying 
\[
\alpha(z)\alpha(w) = (z-w)^{-2}.
\]
Let $\pi_\lambda$ be the Fock module of highest-weight $\lambda$ and let $L = \alpha_+  \mathbb Z$. Then the singlet and triplet algebras are subalgebras of $\pi$ and  the lattice VOA $V_L$ characterized as
\[
\cW(p) = \text{ker}( e_0^{\alpha_-\alpha}: V_L \rightarrow V_{L+\alpha_-}), \qquad \cM(p) = \text{ker}(e_0^{\alpha_-\alpha}: \pi \rightarrow \pi_{\alpha_-}).
\]
Let us denote these embeddings by $\iota$, 
\[
\iota: \cW(p) \hookrightarrow V_L, \qquad \iota: \cM(p) \hookrightarrow \pi. 
\]
There are further singlet modules, characterized as kernels of screenings on Fock modules
\[
M_{r, s} = \text{ker}(e_0^{\alpha_-\alpha}: \pi_{\alpha_{r, s}} \rightarrow \pi_{\alpha_{r+1, p-s}}).
\]
We decide to write $\pi_\lambda$ for the Fock module, while if we view it as a $\cM(p)$-module we write $F_\lambda$.  
The triplet is an extension of the singlet, namely
\begin{equation}\label{eq_TripletAsModule}
\begin{split}
\cW(p) &= \bigoplus_{\lambda \in L } \text{ker}(e_0^{\alpha_-\alpha}: \pi_\lambda \rightarrow \pi_{\lambda+\alpha_-}) \\
&= \bigoplus_{n \in \mathbb Z} \text{ker}(e_0^{\alpha_-\alpha}: \pi_{\alpha_{2n+1, 1}} \rightarrow \pi_{\alpha_{2n+2, p-1}}) \\
&= \bigoplus_{n \in \mathbb Z}  M_{2n+1, 1}.
\end{split}
\end{equation}
Let $A = \pi_{\alpha_{1, 1}} = \pi_0$, as a singlet module it satisfies the non-split exact sequence
$$
0 \rightarrow M_{1, 1} \rightarrow A \rightarrow M_{2, p-1} \rightarrow 0.
$$
$A$ is a commutative algebra in $\catM$ by \cite{HKL}.
We denote by $\cU_A$ and $\cU_A^0$ the category of $A$-modules and local $A$-modules in $\cU$ the the induction and forgetful functor by $\induceA$ and $\forgetA$  as before.

There are a few more interesting VOAs related to the singlet algebra. 
Namely, consider a second Heisenberg VOA with generator $\beta$ and OPE 
\[
\beta(z)\beta(w) = (z-w)^{-2}.
\]
We now denote Fock modules by a superscript $\alpha, \beta$ to indicate the underlying VOA. 
The $\mathcal B(p)$-algebra has been introduced in \cite{CRW}, generalizing the well-known $p=2$ case and the case of $p=3$ of \cite{Ad05}. It is an extension of $\cM(p) \otimes \pi^\beta$ and decomposes as  (see also (5.1) of \cite{CMY2}
\[
\mathcal B(p) \cong \bigoplus_{n \in \Z} M_{1-n, 1} \otimes \pi^\beta_{n \lambda_p}
\]
with $\lambda_p = \sqrt{-\frac{p}{2}}$. $\mathcal B(p)$ are chiral algebras of Argyres-Douglas theories and it is isomorphic to the subregular $W$-algebra of $\mathfrak{sl}_{p-1}$ at level $-(p-1) + \frac{p-1}{p}$  for $p>2$, while in the case $p=2$ it is the well-known $\beta\gamma$-VOA \cite{C17, ACKR, ACGY}. 
$\mathcal B_p$ is by construction a subalgebra of the VOA
\[
\bigoplus_{n \in \Z} \pi^\alpha_{\alpha_{1-n, 1}} \otimes \pi^\beta_{n \lambda_p} \cong \Pi(0) 
\]
where $\Pi(0)$ \cite{Ad1} is a VOA extension of a rank two Heisenberg algebra generated by fields $c, d$ with only non-vanishing OPE
\[
c(z) d(w) =  (z-w)^{-2}
\]
and 
\[
\Pi(0) = \bigoplus_{n \in \Z} \pi^d_n \otimes \pi^c
\]
The isomorphism is given by $d\mapsto \frac{\alpha_-}{2}(\alpha - \sqrt{-1} \beta)$ and 
$c\mapsto \alpha_+(\alpha + \sqrt{-1} \beta)$.
In particular the Fock module $\pi^\alpha_{\alpha_-}\otimes \pi^\beta$ has weight $(-2, \frac{1}{p})$ for $(c, d)$.

For $p>2$, there is a superVOA analogue to $\mathcal B(p)$, dentoted by $\mathcal S(p)$ and this is isomorphic to the principal W-superalgebra of $\mathfrak{sl}_{p-1|1}$ at level $(p-2) + \frac{p}{p-1}$. It is the dual W-algebra via the duality of \cite{CGN21, CL22a}. It decomposes as \cite[Prop.6.1]{CMY2}
\begin{equation}\label{Sp}
\mathcal S(p) \cong \bigoplus_{n \in \Z} M_{1-n, 1} \otimes \pi^\beta_{n \mu_p}
\end{equation}
with $\mu_p = \sqrt{\frac{2-p}{2}}$.
$\mathcal S(p)$ is by construction a subalgebra of the VOA
\begin{equation}\label{Rp}
\mathcal R(p) := \bigoplus_{n \in \Z} \pi^\alpha_{\alpha_{1-n, 1}} \otimes \pi^\beta_{n \mu_p} \cong V_{\epsilon\Z} \otimes \pi^\gamma 
\end{equation}
with $\epsilon^2 =1$  and $\gamma^2= \frac{p}{p-2}$, in particular the lattice VOA $V_{\epsilon\Z}$ is isomorphic to a pair of free fermions. The isomorphism is given by $\epsilon \mapsto \frac{\alpha_+}{2}\alpha + \mu_p \beta$ and 
$\gamma \mapsto -\frac{\alpha_+}{2} \alpha + \frac{p}{2\mu_p}\beta$. 
In particular the Fock module $\pi^\alpha_{\alpha_-}\otimes \pi^\beta$ has weight $(-1, 1)$ for $(\epsilon, \gamma)$.
The braiding is given by
\begin{equation} \label{braidingSp}
(V_{\epsilon\Z}^f \otimes \pi^\gamma_\lambda) \boxtimes (V_\epsilon^g \otimes \pi^\gamma_\mu)  \xrightarrow{(-1)^{fg} q^{(2-p)\lambda \mu }}
(V_{\epsilon\Z}^g \otimes \pi^\gamma_\mu) \boxtimes 
(V_{\epsilon\Z}^f \otimes \pi^\gamma_\lambda) 
\end{equation}
with $\pi^\gamma_\lambda$ the Fock module of weight $\lambda$ and $V_{\epsilon\Z}^\pm$ the even and odd part of $V_{\epsilon\Z} $
and $q = e^{\frac{\pi \i}{p}}$.

The case $p=2$ is a bit different. In this case one has \cite{CRo}
$$
V^k(\mathfrak{gl}_{1|1}) \cong \bigoplus_{n \in \Z} M_{1-n, 1} \otimes \pi^d_n \otimes \pi^c,
$$
here $k \in \C\setminus\{0\}$ as the affine VOAs of $\mathfrak{gl}_{1|1}$ at any two non-zero levels are isomorphic. In this case we have an embedding into 
\begin{equation} \label{gl11}
\bigoplus_{n \in \Z} \pi^\alpha_{\alpha_{1-n, 1}} \otimes \pi^d_n \otimes \pi^c \cong V_{\epsilon\Z} \otimes \pi^B\otimes \pi^A 
\end{equation}
with $\alpha_-=-1$ and with $B$ satisfying inner products $B A = -\hbar^{-1}$ and $A^2=1$ and orthogonal on the others. The isomorphism is given by $\epsilon \mapsto \alpha + c$ and $B \mapsto \hbar^{-1} c$ and $A \mapsto \alpha_-  d$.  In particular the Fock module $\pi^\alpha_{\alpha_-}\otimes \pi^d \otimes \pi^c$ in degree $(\alpha-,0,0)$ for $(\alpha,d,c)$ has degree $(-1,0,  -1)$ for $(\epsilon, B, A)$. With $q=e^{\pi\i\hbar}$ the braiding is 
\begin{equation} \label{braidingGl11}
(V_{\epsilon\Z}^f \otimes \pi^B_{b} \otimes \pi^A_a) \boxtimes (V_\epsilon^g \otimes \pi^B_{b'} \otimes \pi^A_{a'})  \xrightarrow{(-1)^{fg} q^{-b {a'} - a{b'} - \hbar b{b'} }}
(V_\epsilon^g \otimes \pi^B_{b'} \otimes \pi^A_{a'}) \boxtimes (V_{\epsilon\Z}^f \otimes \pi^B_b\otimes \pi^A_a)
\end{equation}

\subsection{Abelian category of the singlet vertex algebra}\label{sec_singletAbelian}

We review the singlet $\cM(p)$ and triplet $\cW(p)$ vertex operator algebras  for integer $p>1$ following \cite{CLR, CMY4, CMY}. We denote by $\UU=\mathcal O^T_{\mathcal M(p)}$ the category of weight modules of the singlet algebra. This category has been finally understood in \cite{CMY4}.

We are interested in three types of modules. 
Firstly a set of simple modules of the singlet algebra  denoted by $M_{r, s}$ for integers $r, s$ and in addition $1\leq s \leq p$. 
Secondly a set of of modules $F_\lambda$ for $\lambda \in \mathbb C$. These are simple, projective and injective  unless $\lambda \in  L = \frac{\alpha_-}{2} \mathbb Z$. Set $F_{r, s}:= F_{\alpha_{r, s}}$. The modules $F_{r, p}$ are simple, projective and injective as well.
For $s \neq  p$ one has an addition type of module denoted by $\overline{F}_{r, s}$ and 
they fit in the non-split short exact sequence 
\begin{equation}\label{ses1}
0 \rightarrow M_{r, s} \rightarrow F_{r, s} \rightarrow M_{r+1, p-s} \rightarrow 0, \qquad 
0 \rightarrow M_{r, s} \rightarrow \overline{F}_{r-1, p-s} \rightarrow M_{r-1, p-s} \rightarrow 0
\end{equation}
As said, the simple $F_\lambda$ are also projective and injective, while the projective cover and injective hull of the $M_{r, s}$ ($s \neq p$) is denoted by $P_{r, s}$ and satisfies
the non-split exact sequences
\begin{equation}\label{ses2}
0 \rightarrow F_{r, s} \rightarrow P_{r, s} \rightarrow F_{r-1, p-s} \rightarrow 0, \qquad 
0 \rightarrow \overline{F}_{r-1, p-s} \rightarrow P_{r, s} \rightarrow \overline{F}_{r, s} \rightarrow 0
\end{equation}
The Loewy diagram of the $P_{r, s}$ is as follows
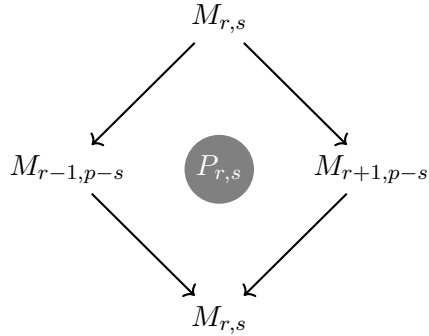
\begin{center}
\begin{tikzpicture}[scale=1]
\node (top) at (0,2) [] {$M_{r, s}$};
\node (left) at (-2,0) [] {$M_{r-1, p-s}$};
\node (right) at (2,0) [] {$M_{r+1, p-s}$};
\node (bottom) at (0,-2) [] {$M_{r, s}$};
\draw[->, thick] (top) -- (left);
\draw[->, thick] (top) -- (right);
\draw[->, thick] (left) -- (bottom);
\draw[->, thick] (right) -- (bottom);
\node (label) at (0,0) [circle, inner sep=2pt, color=white, fill=black!50!] {$P_{r, s}$};
\end{tikzpicture}
\end{center}
and from this it is clear that 
\[
\Ext_{\catM}^1(M_{r, s}, M_{r', s'}) = \begin{cases} \C & \quad \text{for} \ (r', s') = (r\pm 1, p-s) \\ 0 &\quad \text{otherwise} \end{cases} 
\]
We observe that we have a block decomposition
\[
\UU= \bigoplus_{\alpha  \in \C/\Z \, \cup\,  p\Z} \UU_\alpha \oplus \bigoplus_{s= 1}^{p-1} \ \mathcal \UU_s.
 \]
The typical blocks $\UU_\alpha \cong \text{Vect}$ for $\alpha \in \C/\Z \, \cup\,  p\Z$ have each a single simple object, $F_{\frac{\alpha+ p -1}{\sqrt{2p}}}$. We note that this labelling is motivated from \cite{CMR} where the abelian correspondence to the unrolled quantum group was suggested explicitely on the level of modules. 
The block $\UU_s$ has simple objects $M_{r, s}$ and $r$ even as well as $M_{r, p-s}$ for $r$ odd.

We continue with a few preparatory statements:

From \eqref{ses1} and \eqref{ses2} we read of some Hom-spaces
\begin{corollary} \label{cor:hom} For $s, s'\neq  p$
\begin{equation}
\begin{split}
\Hom_{\catM}(F_{r', s'}, F_{r, s}) &= \begin{cases} \mathbb C & \quad \text{for} \ (r', s') = (r-1, p-s) \ \text{and} \ (r', s') = (r, s) \\
0 & \quad \text{otherwise} \end{cases} \\
 \Hom_{\catM}(F_{r', s'}, P_{r, s}) &= \begin{cases} \mathbb C & \quad \text{for} \ (r', s') = (r-1, p-s) \ \text{and} \ (r', s') = (r, s) \\
0 & \quad \text{otherwise} \end{cases} \\ 
\Hom_{\catM}(\overline{F}_{r', s'}, \overline{F}_{r, s}) &= \begin{cases} \mathbb C & \quad \text{for} \ (r', s') = (r+1, p-s) \ \text{and} \ (r', s') = (r, s) \\
0 & \quad \text{otherwise} \end{cases} \\
 \Hom_{\catM}(\overline{F}_{r', s'}, P_{r, s}) &= \begin{cases} \mathbb C & \quad \text{for} \ (r', s') = (r-1, p-s) \ \text{and} \ (r', s') = (r, s) \\
0 & \quad \text{otherwise} \end{cases} \\ 
\Hom_{\catM}(\overline{F}_{r', s'}, F_{r, s}) &= \begin{cases} \mathbb C & \quad \text{for} \ (r', s') = (r, s) \\ 
0 & \quad \text{otherwise} \end{cases} \\ 
\Hom_{\catM}(F_{r', s'}, \overline{F}_{r, s}) &= \begin{cases} \mathbb C & \quad \text{for} \ (r', s') = (r, s) \\
0 & \quad \text{otherwise} \end{cases} \\
\end{split}
\end{equation}
\end{corollary}

\begin{lemma} \label{lem:ext} For $s \neq p$
\[
\Ext_{\catM}^1(F_{r', s'}, F_{r, s}) =  \begin{cases} \mathbb C & \quad \text{for} \ (r', s') = (r-1, p-s) \ \text{and} \ (r', s') = (r-2, s) \\
0 & \quad \text{otherwise} \end{cases}
\]
\end{lemma}
\begin{proof}
The short exact sequence  $0 \rightarrow  F_{r, s} \rightarrow P_{r, s} \rightarrow F_{r-1, p-s} \rightarrow 0 $ gives rise to the usual long exact sequence whose relevant piece is
\begin{equation}
\begin{split}
0 \rightarrow &\Hom_{\catM}(M, F_{r,s}) \rightarrow \Hom_{\catM}(M, P_{r, s}) \rightarrow \\
&\Hom_{\catM}(M, F_{r-1, p-s}) \rightarrow  \Ext_{\catM}^1(M, F_{r, s})  \rightarrow  \Ext_{\catM}^1(M, P_{r, s}) = 0 
\end{split}
\end{equation}
The claim follows immediately using Corollary \ref{cor:hom}.
\end{proof}

\subsection{Some singlet VOA fusion rules}

We need the fusion rules for $F_{r, s}$ with $\overline{F}_{r', s'}$. 
Define 
\[
P_{r', s', r, s} := \bigoplus_{\substack{ \ell =2p+1-s-s' \\  \ell +s + s' \ \text{odd}}}^{p} P_{r+r'-1, \ell}.
\]
\begin{theorem}\label{thm:fusion} \textup{\cite[Thm.9.1]{CLR}} and \textup{\cite[Thm.1.3]{CMY4}}
	The fusion rules 
	\begin{equation}
	\begin{split}
	M_{r', s'} \boxtimes M_{r, s} &=P_{r', s', r, s} \oplus  \bigoplus_{\substack{ \ell = |s-s'| +1 \\  \ell +s + s' \ \text{odd}}}^{\text{min}\{s+s'-1, 2p-1-s-s' \}} M_{r+r'-1, \ell} \\
	M_{r', s'} \boxtimes F_{r, s} &=P_{r', s', r, s} \oplus P_{r', s', r+1, p-s } \oplus  \bigoplus_{\substack{ \ell = |s-s'| +1 \\  \ell +s + s' \ \text{odd}}}^{\text{min}\{s+s'-1, 2p-1-s-s' \}} F_{r+r'-1, \ell} \\
 M_{r', s'} \boxtimes \overline{F}_{r, s} &=P_{r', s', r+1, p-s} \oplus P_{r', s', r, s } \oplus  \bigoplus_{\substack{ \ell = |s-s'| +1 \\  \ell +s + s' \ \text{odd}}}^{\text{min}\{s+s'-1, 2p-1-s-s' \}} \overline{F}_{r+r', \ell} \\
    M_{r, s} \boxtimes F_\lambda &= \bigoplus_{\ell = 0}^{s-1} F_{\lambda + \alpha_{r, s} + \ell \alpha_-}
	\end{split}
	\end{equation}
	hold for $\lambda \notin L$.
\end{theorem}

\newcommand{\forget}{\forgetA} 

\section{The Kazhdan-Lusztig equivalence for the singlet VOAs}\label{sec:KLsinglet}

\subsection{Recognizing the tensor category \texorpdfstring{$\UU$}{U}  knowing the abelian category \texorpdfstring{$\BB$}{B} }

Collecting our results above, we have the following statement assuming we know $\UU$ and $\BB$ are as expected for a quantum group. We will then apply this to the case of the singlet algebra.  

\begin{theorem}\label{thm_CharacterizingUsingB}
Suppose $\UU$ is a braided tensor category and $A\in \UU$ a commutative algebra  fulfilling Assumption \ref{assumption}. Let $\CC=\Vect_\Gamma^Q=\Rep^{wt}(C)$ and 
let $X\in \CC$ with finite-dimensional Nichols algebra $\Nichols$ satisfying the mild degree conditions in Lemma \ref{lm_isNichols}  
and $U_q$ the corresponding quantum group in Lemma \ref{lm_realizedQuantumGroup}. Assume that 
\begin{itemize}
    \item $\UU\cong \Rep^{wt}(U_q)$ as abelian categories.
    \item $\UU_A^0\cong \Vect_\Gamma^Q$ as braided tensor categories.
    \item $\UU_A\cong \Rep(\Nichols)(\CC)$ as abelian categories, compatible with the respective $\CC$-module structures as described in the assumptions of Lemma \ref{lm_infiniteSplittingIsHopfAlgB}.
\end{itemize}
Then in fact these are equivalence of tensor categories and braided tensor categories.
\end{theorem}
\begin{proof}
    \begin{enumerate}[a)]
    \item Since by assumptions on $\UU_A$ all simple modules are trivial $\Nichols$-modules and hence are already local, we have  by Lemma \ref{lm_socle} that the abelian functor forgetting $\Nichols$-action gives rise to a splitting tensor functor $\Localization$.
    \item By Lemma \ref{lm_infiniteSplittingIsHopfAlgB} this implies that $\Nichols$ has the structure of a Hopf algebra such that $\UU_A\cong \Rep(\Nichols)(\CC)$ as tensor categories, as asserted. 
    \item By Lemma \ref{lm_isNichols} $\Nichols$ is in fact the Nichols algebra of $X$.
    \item By Corollary \ref{cor_SchauenburgInfinite} we have $\UU=\cZ_\CC(\UU_A)$ and by construction in  Lemma \ref{lm_realizedQuantumGroup} hence $\UU\cong \Rep^{wt}(U_q)$ as asserted.
    \end{enumerate}
\end{proof}

\subsection{The abelian category \texorpdfstring{$\catM$}{U} }

We want to explain that we have an abelian equivalence of $\UU$ and the category of weight modules of the small unrolled quantum group $\Rep_{\mathrm{wt}}u_q^H(\mathfrak{sl}_2)$. For this we first need to introduce an auxiliary category $\DD$, so that we can describe the block decomposition of $\Rep_{\mathrm{wt}}u_q^H(\mathfrak{sl}_2)$.

Let $\DD$ be an abelian category with simple objects $L_n$ with $n \in \mathbb Z$ and the following properties
\begin{enumerate}
\item 
$
\Ext_\DD^1(L_n, L_m) = \begin{cases} \mathbb C & \quad  \text{if} \ |n-m| =1 \\ 0 &\quad \text{else;} \end{cases}
$

 Let
$ \ses{L_n}{E^\pm_n}{L_{n \pm1}} \ \in  \ \Ext_\DD^1( L_{n \pm 1}, L_n)$
{be a non-zero element.}
Then 
\item
$\Ext_\DD^1(E^+_{n+1}, L_n) = 0 = \Ext_\DD^1(L_n, E^+_{n-2})\ $ and $\ \Ext_\DD^1(E^-_{n-1}, L_n) = 0 = \Ext_\DD^1(L_n, E^-_{n+2})$;
\item 
$\Ext_\DD^1(E^+_{n-1}, L_n) = 0 = \Ext_\DD^1(L_n, E^+_n)\ $ and  $\ \Ext_\DD^1(E^-_{n+1}, L_n) = 0 = \Ext_\DD^1(L_n, E^-_n)$;
\item 
$\Ext_\DD^1(E^+_n, E^+_{n+1}) \neq 0 \ \text{and} \ \Ext_\DD^1(E^-_n, E^-_{n-1}) \neq 0$.
\end{enumerate}
Let $\DD$ be another abelian category with the exact same properties. Then $\DD$ and $\DD'$ are equivalent \cite[Thm.5.2]{ACK}. 
Let now
\[
\ses{E^\pm_{n\pm 1}}{P_n^\pm}{E^\pm_n} \quad \in \ \Ext_\DD^1(E^\pm_n, E^\pm_{n \pm 1})
\]
and assume that $P_n^\pm$ is projective and injective, then the properties 2 and 3 of the list are consequences as the relevant part of the long exact sequence in cohomology is 
\[
\Hom_\DD(M, E^\pm_{n \mp 1}) \rightarrow  \Ext_\DD^1(M, E^\pm_n) \rightarrow 0, \qquad
\Hom_\DD(E^\pm_{n \pm 1}, M) \rightarrow  \Ext_\DD^1(E^\pm_n, M) \rightarrow 0
\]
and since $\Hom_\DD(L_m, E^\pm_n) = \delta_{n, m} \C$ and $\Hom_\DD(E^\pm_n, L_m)= \delta_{m, n \pm 1} \C$.
\begin{example}
The abelian category $\Rep_{\mathrm{wt}}u_q^H(\mathfrak{sl}_2)$ of weight modules of the unrolled quantum group  has been studied in detail in  \textup{\cite{CGP}}
and from that it is clear that it decomposes as 
\[
\Rep_{\mathrm{wt}}u_q^H(\mathfrak{sl}_2) \cong \bigoplus_{\alpha  \in \C/\Z \, \cup\,  p\Z} \UU_\alpha \;\oplus\; \bigoplus_{s= 1}^{p-1} \ \mathcal \UU_s
\]
with $\UU_\alpha \cong \text{Vect}$ for $\alpha  \in \C/\Z \, \cup\,  p\Z$ and $\UU_s \cong \DD$, see \textup{\cite[Section 2]{ACK}}. 
\end{example}

On the other hand, we have obtained the block decomposition of $O^T_{\mathcal M(p)}$ in Section \ref{sec_singletAbelian}. If we set 
\begin{equation}
\begin{split}
L_n &= \begin{cases} M_{n, s} &\quad n \ \text{even} \\ M_{n, p-s} &\quad n \ \text{odd} \\ \end{cases} \\ 
E^+_n &= \begin{cases} F_{n, s} &\quad n \ \text{even} \\ F_{n, p-s} &\quad n \ \text{odd} \\ 
\end{cases} \\ 
E^-_n &= \begin{cases} \overline{F}_{n-1, p-s} &\quad n \ \text{even} \\ \overline{F}_{n-1, s} &\quad n \ \text{odd} \\ \end{cases} \\ 
P_n &= \begin{cases} P_{n, s} &\quad n \ \text{even} \\ P_{n, p-s} &\quad n \ \text{odd} \\ \end{cases} \\
\end{split}
\end{equation}
then $\UU_s$ has exactly the properties of the atypical block of last subsection and so in particular due to  \cite[Thm.5.2]{ACK} we have that $\UU_s \cong \DD$ as an abelian category. Hence

\begin{theorem}\label{thm_SingletAbelian}
We have an equivalence $\Rep_{\mathrm{wt}}(u_q^H(\mathfrak{sl}_2)) \cong \mathcal  O^T_{\mathcal M(p)}$ as abelian categories.
\end{theorem}

\subsection{The abelian category \texorpdfstring{$\catM_A$}{U\_A} }

We need to understand the category $\catM_A$. In particular we need to show that it is equivalent to the category of Nichols algebra modules in $\text{Vect}_\C^Q$ for the Nichols algebra of Example \ref{exm_NicholsRank1}.

Recall that $A= F_{1, 1}$ is the commutative algebra in the category $\UU$  corresponding to the Heisenberg VOA $\pi$. 

Using Theorem \ref{thm:fusion} and exactness of tensor product (thanks to rigidity \cite{CMY4})
we get the fusion rules
\begin{equation}\label{eq:somefusion}
\begin{split}
A \boxtimes F_\mu &= F_\mu \oplus F_{\mu + \alpha_-} \oplus \dots \oplus F_{\mu + (p-1)\alpha_-} \\
A \boxtimes F_{r'. s'} &= P_{2, p-1, r', s'} \oplus P_{2, p-1, r'+1, p-s'} \oplus F_{r'+1, p-s'} \oplus  F_{r', s'} \\
&= \bigoplus_{\substack{\ell = p - (s'-2) \\ \ell = p  -s' \mod 2}}^p P_{r'+1, \ell} \oplus
\bigoplus_{\substack{\ell = s'+2 \\ \ell = s' \mod 2}}^p P_{r'+2, \ell} \oplus F_{r'+1, p-s'} \oplus  F_{r', s'} \\
A \boxtimes \overline{F}_{r'. s'} &= P_{2, p-1, r', s'} \oplus P_{2, p-1, r'+1, p-s'} \oplus P_{r'+1, p-s'}\\
&= \bigoplus_{\substack{\ell = p - s' \\ \ell = p  -s' \mod 2}}^p P_{r'+1, \ell} \oplus
\bigoplus_{\substack{\ell = s'+2 \\ \ell = s' \mod 2}}^p P_{r'+2, \ell}
\end{split}
\end{equation}
The last line used Corollary 2.3 of \cite{CLR} and the second line used 
\[
	M_{2, p-1} \boxtimes F_{r'. s'} =  P_{2, p-1, r', s'} \oplus P_{2, p-1, r'+1, p-s'} \oplus F_{r'+1, p-s'}
	\]
 together with $\text{Ext}^1(F_{r+1, p-s}, F_{r, s})= 0$ by Lemma \ref{lem:ext}.
\begin{corollary}\label{cor:someHoms} Set $G_\lambda = F_\lambda$ for $\lambda \notin L$, $G_{\lambda} = F_{r, p}$ for $\lambda = \alpha_{r, p}$ and $G_\lambda = \overline{F}_{r, s}$ for $\lambda = \alpha_{r, s}$ and $r\in \Z$ and $s = 1, \dots, p-1$.
	\begin{equation}
 \begin{split} 
 \Hom_{\catM}(G_\lambda, A \boxtimes G_\mu) = \begin{cases} \mathbb C &\quad \lambda = \mu + \ell \alpha_-, \qquad \ell = 0, \dots, p-1 \\ 0 & \quad \text{else} \end{cases} \\
 \Hom_{\catM}(A \boxtimes G_\mu, F_\lambda) = \begin{cases} \mathbb C &\quad \lambda = \mu + \ell \alpha_-, \qquad \ell = 0, \dots, p-1 \\ 0 & \quad \text{else} \end{cases} 
 \end{split}
 \end{equation}
\end{corollary}
\begin{proof}
	Immediate from \eqref{eq:somefusion} together with Corollary \ref{cor:hom}.
\end{proof}
\begin{lemma}\label{lm_singletAbelianSteps} For $\lambda \in \C$
\begin{enumerate}
\item $\induceA(G_\lambda)$ is indecomposable.
\item Let $X$ be a simple object in $\catM_A$ with the property that $\forgetA(X) \cong F_\lambda$, then $X \cong \pi_\lambda$.
\item The simple composition factors of $\induceA(G_\lambda)$ are $\{ \pi_\lambda, \pi_{\lambda- \alpha_-}, \dots, \pi_{\lambda+(p-1)\alpha_-}\}$
    \item Let $X$ be a simple object in $\catM_A$ then $X \cong \pi_\mu$ for some $\mu$. In particular every simple module is local. 
    \item The top (the maximal semisimple quotient) of $\induceA(G_\lambda)$ is $\pi_\lambda$.
    \item $\induceA(G_\lambda)$ is projective and $\induceA(G_\lambda) \otimes  \pi_\mu \cong \induceA(G_{\lambda+\mu})$ for all $\lambda, \mu \in C$.
 \item    Set $H_\lambda = \induceA(M_{0, p-1}) \otimes \pi_\lambda$. Then there are non-split short exact sequences
 \[
 \ses{H_\lambda}{\induceA(G_{\lambda- p\alpha_-})}{\pi_{\lambda-p\alpha_-}},\qquad\ses{\pi_\lambda}{\induceA(G_{\lambda- (p-1)\alpha_-})}{H_\lambda}
 \]
 in particular the top of $H_\lambda$ is $\pi_{\lambda - (p-1)\alpha_-}$ and so $\Hom_{\Rep(\catM)}(H_\lambda, \pi_\mu) = \C \delta_{\lambda- (p-1)\alpha_-, \mu}$.
 \item $\Ext^1_{\Rep(\catM)}(\pi_\lambda, \pi_\mu) = \C \delta_{\lambda+\alpha_-, \mu}$.
 \item  
	$\induceA(G_\lambda)$ is indecomposable with Loewy diagram 
	\[
	\pi_\lambda \rightarrow \pi_{\lambda+ \alpha_-} \rightarrow \dots \rightarrow \pi_{\lambda+ (p-1) \alpha_-}
	\]
 \item Any object in $\catM_A$ is a quotient of a finite direct sum of the $\induceA(G_\lambda)$,  $\induceA(G_{\lambda_1}) \oplus \dots \oplus
 \induceA(G_{\lambda_m})$. In particular the $\{\induceA(G_\lambda) | \lambda \in \mathbb C\}$ form a complete list of indecomposable projective objects in $\catM_A$.
 \item $\catM_A$ is rigid.
\end{enumerate}

\end{lemma}
\begin{proof} We are using Frobenius reciprocity $\Hom_{\catM_A}(\induceA(X), Y) = \text{Hom}_{\catM}(X, \forgetA(Y))$.
\begin{enumerate}
\item  Firstly we get from the previous Corollary that
\begin{equation}\nonumber
    \begin{split}
    \Hom_{\catM_A}(\induceA(G_\lambda), \induceA(G_\mu)) &= \text{Hom}_{\catM}( G_\lambda, \forgetA(\induceA(G_\mu)))\\ &= \text{Hom}_{\catM}( G_\lambda, A \boxtimes G_\mu) \\ &= \begin{cases} \mathbb C &\quad \lambda = \mu + \ell \alpha_-, \qquad \ell = 0, \dots, p-1 \\ 0 & \quad \text{else} \end{cases}    
    \end{split}
\end{equation}
so each $\induceA(G_\lambda)$ is indecomposable.
\item We have 
\[
\mathbb C = \Hom_{\catM}(G_\lambda, F_\lambda) = \Hom_{\catM}(G_\lambda, \forget(\pi_\lambda)) =
\Hom_{\catM_A}(\induceA(G_\lambda), \pi_\lambda) 
\]
as well as 
\[
\mathbb C = \Hom_{\catM}(G_\lambda, F_\lambda) = \Hom_{\catM}(G_\lambda, \forget(X)) =
\Hom_{\catM_A}(\induceA(G_\lambda), X) 
\]
Let $f \in  \Hom_{\catM}(G_\lambda, F_\lambda)$, this means the induction $\induceA(f)$ is a surjective morphism to $\pi_\lambda + X$.
But a morphism from $\induceA(G_\lambda)$ to $X$ or $\pi_\lambda$ must in particular be a morphism in the underlying category $\catM$, that is in $\Hom_{\catM}(A \otimes G_\lambda, F_\lambda)$. 
This Hom-space is however only one-dimensional by the previous Corollary and so $\pi_\lambda + X = \pi_\lambda$ and in particular $X \cong \pi_\lambda$.
\item
 This is seen inductively: Let $\mu^i = \lambda + (p-i)\alpha_-$ for $i = 1, \dots, p$.  
	It follows that
	\[
	\Hom_{\catM_A}(\induceA(G_{\mu^i}), \induceA(G_\lambda)) = \Hom_{\catM}(G_{\mu^i}, A \boxtimes G_\lambda) = \mathbb C 
	\] 
	from the previous Corollary. Also from Frobenius reciprocity $\pi_\nu$ is a simple quotient of $\mathcal F(G_\nu)$. 
	The image of a non-trivial morphism $\induceA(G_{\mu^1}) \rightarrow  \induceA(G_\lambda)$ must in particular be an object in $\catM$.  However we inspect from the fusion rules that the  only such morphism has $F_{\mu^1}= \forgetA(\pi_{\mu^1})$ as image and hence by the previous point $\pi_{\mu^1}$ must be a simple submodule of $\induceA(G_\lambda)$. 
	Set $Q^0 = \induceA(G_\lambda)$ and $Q^1 = Q^0/ \pi_{\mu^1}$. The induction hypothesis is now that $Q^{i-1}$ has $\pi_{\mu^i}$ as simple submodule so that we can set $Q^i = Q^{i-1}/\pi_{\mu^i}$.  Using Frobenius reciprocity we verify that 
	$\Hom_{\catM_A}(\induceA(G_{\mu^i}), Q^{i-1}) = \mathbb C$ and again one observes that the only possible non-zero morphism needs to have  $F_{\mu^i}$ as image so that $\pi_{\mu^i}$ is a simple submodule of $Q^{i-1}$. 
 \item Let now $X$ be a simple object in $\catM_A$ then $\forgetA(X)$ is an object in $\catM$ and so in particular it has a simple submodule $M$. But for every simple module in $\catM$ there exists precisely one $\mu \in \C$ such that $\Hom_{\catM}(G_\mu, M)$ is non-zero. Hence $\Hom_{\catM}(G_\mu, \forgetA(X))$ is non-zero and  by Frobenius reciprocity 
 $\Hom_{\catM_A}(\induceA(G_\lambda), X)$ is non-zero as well, in particular $X \cong \pi_\nu$ for some $\nu$ by the previous statement. 
 \item This is now clear since $\delta_{\lambda, \mu} \C = \Hom_{\catM}(G_\lambda, F_\mu) =\Hom_{\catM_A}(\induceA(G_\lambda), \pi_\mu)$.
 \item Since $\pi_{\lambda + \ell \alpha_-} \otimes \pi_\mu = \pi_{\lambda+\mu + \ell\alpha_-}$
 the module $\induceA(G_\lambda) \otimes  \pi_\mu$  has the same composition factors as $\induceA(G_{\lambda+\mu})$.
 Since $\pi_\mu$ is a simple current Proposition 2.5 of \cite{CKLR} applies\footnote{Proposition 2.5 of \cite{CKLR} is stated for local modules of a VOA, but the proof works in any tensor category (with right exact tensor product, which holds in $\cU_A$  \cite[Remark 2.64]{CKM}).}, in particular  $\induceA(G_\lambda) \otimes  \pi_\mu$ is indecomposable and the top is $\pi_\lambda \otimes \pi_\mu = \pi_{\lambda+\mu}$. 
 We have 
 \[
 \Hom_{\catM_A}(\induceA(G_{\lambda+\mu}), \induceA(G_\lambda) \otimes  \pi_\mu) =
 \Hom_{\catM}(G_{\lambda+ \mu}, \forgetA (\induceA(G_\lambda) \otimes  \pi_\mu)).
 \]
 There are two cases. First, if $\lambda+\mu$ is typical, i.e.~not equal to $\alpha_{r', s'}$ for some $r'\in \Z$ and $s'=1, \dots, p-1$, then
 \[
 \forgetA (\induceA(G_\lambda) \otimes  \pi_\mu) = F_{\lambda+\mu} \oplus F_{\lambda+\mu +\alpha_-} \oplus \dots \oplus F_{\lambda+\mu + (p-1)\alpha_-} 
 \]
and so $\Hom_{\catM_A}(\induceA(G_{\lambda+\mu}), \induceA(G_\lambda) \otimes  \pi_\mu) = \C$. Since the top level of both modules is $\pi_{\lambda+\mu}$ the only possibility is that $\induceA(G_\lambda) \otimes  \pi_\mu \cong \induceA(G_{\lambda+\mu})$.
Now, assume that $\lambda+ \mu = \alpha_{r', s'}$ for some $r'\in \Z$ and $s'=1, \dots, p-1$, then
$\induceA(G_\lambda) \otimes  \pi_\mu$ must have the same composition factors as 
\[
\bigoplus_{\substack{\ell = p - s' \\ \ell = p  -s' \mod 2}}^p P_{r'+1, \ell} \oplus
\bigoplus_{\substack{\ell = s'+2 \\ \ell = s' \mod 2}}^p P_{r'+2, \ell}.
\] 
The only summand that belongs to the same block as $\overline{F}_{r', s'}$ is $P_{r'+1, p-s'}$.
The functor $\forgetA$ is exact and since $\forgetA(\pi_\lambda) = F_\lambda$, the only possibilities for the decomposition of $\induceA(G_\lambda) \otimes  \pi_\mu$ is 
\[
\bigoplus_{\substack{\ell = p - s' \\ \ell = p  -s' \mod 2}}^p R_{r'+1, \ell} \oplus
\bigoplus_{\substack{\ell = s'+2 \\ \ell = s' \mod 2}}^p R_{r'+2, \ell},
\]
with $R_{r, s} \in \{ P_{r, s}, F_{r, s }\oplus F_{r-1, p-s}\}$. Since $$\Hom_{\catM}(\overline{F}_{r', s'}, P_{r'+1, p-s'}) = \C = \Hom_{\catM}(\overline{F}_{r', s'}, F_{r'+1, p-s'} \oplus F_{r', s'})$$ we have 
\[
 \Hom_{\catM_A}(\induceA(G_{\lambda+\mu}), \induceA(G_\lambda) \otimes  \pi_\mu) =
 \Hom_{\catM}(G_{\lambda+ \mu}, \forgetA (\induceA(G_\lambda) \otimes  \pi_\mu)) = \C
 \]
 and the same argument as in the typical case applies. 

 Consider $\lambda \notin L$ in this case $G_\lambda$ is projective and so is $\induceA(G_\lambda)$, see e.g. Remark 2.64 of \cite{CKM}. The $\pi_\mu$ are rigid and since the tensor product of a rigid object with a projective one is projective \cite[Prop. 4.2.12]{EGNO} the projectivity holds for all $\lambda$.
\item  The $M_{r, 1}$ are simple currents and in particular $M_{r, 1} \otimes A = \overline{F}_{r, 1}$. Using Frobenius reciprocity,  that is $ \Hom_{\catM_A}(\induceA(M_{r, 1}), \pi_\lambda)= \Hom_\catM(M_{r, 1}, F_{\lambda}) = \C \delta_{\alpha_{r, 1}, \lambda}$ we observe that 
$\induceA(M_{r, 1}) \cong \pi_{\alpha_{r, 1}}$.
We apply the exact functor $\induceA$ to
\[
0 \rightarrow M_{1, 1} \rightarrow \overline{F}_{0, p-1} \rightarrow M_{0, p-1} \rightarrow 0, \qquad 
0 \rightarrow M_{0, p-1} \rightarrow \overline{F}_{-1, 1} \rightarrow M_{-1, 1} \rightarrow 0
\]
and use that $\induceA(M_{1, 1}) = \pi_0$, $\overline{F}_{0, p-1} = G_{-(p-1)\alpha_-}$, $\overline{F}_{-1, 1} = G_{-p\alpha_-}$ and $\induceA(M_{-1, 1}) = \pi_{-p\alpha_-}$
giving 
\[
0 \rightarrow \pi_0 \rightarrow \induceA(G_{-(p-1)\alpha_-}) \rightarrow \induceA(M_{0, p-1}) \rightarrow 0
\]
and
\[
0 \rightarrow \induceA(M_{0, p-1}) \rightarrow \induceA(G_{-p\alpha_-}) \rightarrow \pi_{-p\alpha_-} \rightarrow 0
\]
the result follows from exactness of tensoring with simple currents together with the previous point. 
\item 
The short exact sequence 
\[
\ses{H_{\lambda+ p \alpha_-}}{\induceA(G_\lambda)}{\pi_\lambda}
\]
induces the long-exact sequence
\begin{equation}\nonumber
    \begin{split}
 0 &\rightarrow \Hom_{\catM_A}(\pi_\lambda, \pi_\mu) 
 \rightarrow \Hom_{\catM_A}(\induceA(G_\lambda), \pi_\mu)\\ 
&\rightarrow \Hom_{\catM_A}(H_{\lambda + p\alpha_-}, \pi_\mu)
\rightarrow \Ext^1_{\catM_A}(\pi_\lambda, \pi_\mu) \\
 &\rightarrow \Ext^1_{\catM_A}(\induceA(G_\lambda), \pi_\mu)
    \end{split}
\end{equation}
Since $\induceA(G_\lambda)$ is projective $\Ext^1_{\catM_A}(\induceA(G_\lambda), \pi_\mu) = 0$, while \newline  $\Hom_{\catM_A}(H_{\lambda + p\alpha_-}, \pi_\mu) = \C \delta_{\lambda+\alpha_-, \mu}$ by the previous point and\newline  $\Hom_{\catM_A}(\pi_\lambda, \pi_\mu)  = \C \delta_{\lambda, \mu} = \Hom_{\catM_A}(\induceA(G_\lambda), \pi_\mu)$ by point 5 of this Lemma. 
The statement follows. 
\item This point is clear: from point 1 $\induceA(G_\lambda)$ is indecomposable. From point 3, its composition factors are $\pi_\lambda, \pi_{\lambda+\alpha_-}, \dots, \pi_{\lambda+(p-1)\alpha_-}$ and from the previous point, the only possible arrows are from $\pi_{\lambda + s \alpha_-}$ to $\pi_{\lambda+(s+1)\alpha_-}$. If one such arrow were missing, than the module wouldn't be indecomposable. 
\item  $\catM$ is of finite length \cite{CMY4}. By \eqref{ses2} every projective object in $\catM$ has a composition series with all composition factors being $G_\lambda$ for suitable $\lambda$. By exactness of $\induceA$ together with projectivity of the $\induceA(G_\lambda)$ this means that the induction of any projective module is a direct sum of suitable $\induceA(G_\lambda)$. 
Let $X$ be an arbitrary object in $\catM_A$, the surjection $\forgetA(X) \rightarrow \forgetA(X)$ induces a surjection of $\induceA(\forgetA(X)$ onto $X$, i.e.~$X$ is a quotient of an induced module and so in particular of the induction of the projective cover of $\forgetA(X)$, but the lattere, as just remarked, is a direct sum of suitable $\induceA(G_\lambda)$. 
\item Recall \eqref{ses1} 
\begin{equation}\nonumber
0 \rightarrow M_{r, s} \rightarrow F_{r, s} \rightarrow M_{r+1, p-s} \rightarrow 0, \qquad 
0 \rightarrow M_{r, s} \rightarrow \overline{F}_{r-1, p-s} \rightarrow M_{r-1, p-s} \rightarrow 0
\end{equation}
Let $\lambda = \alpha_{r-1, p-s}$, so that $G_\lambda = \overline{F}_{r-1, p-s}$.
Since $\cU$ rigid \cite{CMY4} induction is exact and so $\induceA(M_{r, s})$ is a submodule of $\induceA(G_\lambda)$. Since $\Hom_{\cU_A}(\induceA(M_{r, s}), \pi_\mu) \cong \Hom_\cU(M_{r, s}, F_\mu) = \delta_{\mu, \alpha_{r, s}} \mathbb C$ the only possibility is that $\induceA(M_{r, s})$ is the indecomposable submodule of $\induceA(G_\lambda)$ whose top is $\pi_{\alpha_{r, s}}$. We compute that $\alpha_{r, s} = \alpha_{r-1, p-s} + (p-s) \alpha_-= \lambda + (p-s) \alpha_-$ and so the Loewy diagram of $\induceA(M_{r, s})$ is 
\[
\pi_{\alpha_{r, s}} \rightarrow \pi_{\alpha_{r, s} + \alpha_-} \rightarrow  \dots \rightarrow 
\pi_{\alpha_{r, s} + (s-1)\alpha_-}. 
\]
In particular $\induceA(M_{r, s})$ has length $s$ and as an induced module it is rigid. 
By exactness of $\induceA$ the sequence
\[
0 \rightarrow M_{r+1, p-s} \rightarrow \overline{F}_{r, s} \rightarrow M_{r, s} \rightarrow 0
\]
induces to 
\[
0 \rightarrow \induceA(M_{r+1, p-s}) \rightarrow \induceA(G_{\alpha_{r, s}}) \rightarrow \induceA(M_{r, s}) \rightarrow 0
\]
and tensoring with the simple current $\pi_\mu$ for any $\mu$ gives
\[
0 \rightarrow \pi_\mu \otimes \induceA(M_{r+1, p-s}) \rightarrow \induceA(G_{\alpha_{r, s}+\mu }) \rightarrow \pi_\mu \otimes\induceA(M_{r, s}) \rightarrow 0.
\]
Since $\pi_\mu$  and $\induceA(M_{r, s})$ are rigid the same is true for $\pi_\mu \otimes\induceA(M_{r, s})$ and so any quotient of $\induceA(G_\nu)$ for any $\nu$ is rigid.
The statement follows from the previous point.\qedhere
\end{enumerate}
\end{proof}

The information above suffices to establish an explicit equivalence of abelian categories $$F:\catM_A\to \Rep(\Nichols)(\Vect^{Q}_\C)$$
where $\Nichols=\C[x]/x^p$ is an algebra in $\Vect^{Q}_\C$ described explicitly in Example \ref{exm_NicholsRank1}, with $x$ in degree $\alpha$ having a self-braiding $Q(\alpha)=e^{2\pi\i/p}$. Namely, the simple modules in $\catM_A$ are by Lemma \ref{lm_singletAbelianSteps} (2) the Heisenberg modules $\pi_\lambda$ (which are also the local modules) and the simple modules in $\Rep(\Nichols)(\Vect^{Q}_\C)$ are the modules $\C_{\lambda,1}$, which are the $1$-dimensional graded vector space  $\C_\lambda$ with the zero action of $x$, and we have an equivalence of braided tensor categories $\CC=\Rep(\pi)\cong \Vect^{Q}_\C$ with $F:\pi_\lambda \mapsto \C_{\lambda}$. So we set
$$F:\pi_\lambda \mapsto \C_{\lambda,1}$$
The indecomposable projective modules in $\catM_A$ are by Lemma \ref{lm_singletAbelianSteps} (10) the induced modules $\induceA(G_\lambda)$ and the indecomposable projective modules in $\Rep(\Nichols)(\Vect^{Q}_\C)$ are the modules $\C_{\lambda,p}$ of dimension $p$, and by  by Lemma \ref{lm_singletAbelianSteps} (9) these modules have coinciding Loewy diagram. To reduce choices in the upcoming proof we will set
$$F:\;\C_{\lambda,p}\boxtimes \induceA(G_0)\boxtimes_{\catM_A} \pi_\lambda$$
which is (non-canonically) isomorphic to $\induceA(G_\lambda)$ by Lemma \ref{lm_singletAbelianSteps} (6). 

\begin{theorem}\label{thm_SingletTwistedAbelian}
    There is an equivalence of abelian categories 
    $$F:\,\Rep(\Nichols)(\CC)\to \catM_A$$    
    for $\CC=\Rep(\pi)\cong \Vect_{\C}$ compatible with the splitting tensor functor $\Localization:\catM_A\to \Rep(\pi)$ from the socle filtration and the forgetful functor to $\CC$, and compatible with the right $\CC$-module structures in the sense described in Lemma \ref{lm_infiniteSplittingIsHopfAlgB} or the proof below. 
\end{theorem}
By Lemma \ref{lm_infiniteSplittingIsHopfAlgB} this implies
\begin{corollary}\label{cor:84}
On the algebra $\Nichols$ there exists the structure of a Hopf algebra in $\CC$ such that $F$ is an equivalence of tensor categories.     
\end{corollary}
\begin{proof}[Proof of Theorem \ref{thm_SingletTwistedAbelian}]
We have described in the previous paragraph a bijection $F$ between simple objects and indecomposable projective objects, and coinciding Loewy diagrams (including the information which extensions split). This is sufficient to establish an abelian equivalence, because the abelian category is fixed by the algebra of endomorphisms of a projective generator. In our case, more directly, all indecomposable objects are quotients of indecomposable projective objects, so we can map them to the respective quotients. We now discuss the functor in more details and fix choices, in particular an explicit action on $Hom$-spaces, and the additional functors and constraints.

Since the non-strictness is mainly on the vertex algebra side, it is in the following more convenient to fix $\CC=\Rep(\pi)$ and regard $\Vect_\C^Q$ as equivalent tensor category 
$$e:\Vect_\C^Q\stackrel{\sim}{\longrightarrow} \Rep(\pi)$$
by setting $e(\C_\lambda)=\pi_\lambda$ and 
with  tensor structure $e^T_{\lambda,\mu}: \pi_{\lambda+\mu}\to \pi_\lambda\boxtimes_\CC \pi_\mu$ given by the  standard choice of intertwiners $\Y_{\lambda,\mu}(a,z)b$.

The functors $\varepsilon:\CC\to \Rep(\Nichols)(\Vect_\C^Q)$ and $\forgetNichols:\Rep(\Nichols)(\Vect_\C^Q)\to \CC$ are the functors endowing a graded vector space with the trivial action resp. forgetting the action, composed with the tensor equivalence $e:\Vect_\C^Q\to \Rep(\pi)$. In particular $\varepsilon(\pi_\lambda)=\C_{\lambda,1}$ and $\forgetNichols(\C_{\lambda,l})=\bigoplus_{k=0}^{l-1} \pi_{k\alpha_- +\lambda}$ and $\forgetNichols^T=e^T$. We suppress the natural transformation $\forgetNichols\to \varepsilon=\id$.

The tensor functor $\iota:\CC=\UU_0^ A \to \UU_A$ the embedding of local modules into all modules, with trivial tensor structure $\iota^T=\id$. Since we are in the case of an extension by a commutative algebra, the embedding lifts to an embedding into the center by Definition \ref{def_cent2}. For the duration of this proof we will not suppress this, but denote $\pi_\lambda\in \UU$ and $\iota\pi_\lambda\in \UU_A$. 

The tensor functor $\Localization:\catM_A\to \CC$ sends a module in $\UU_A$ to its associated graded module (with respect to the socle filtration). The image of $\induceA(G_\lambda)=A\boxtimes_\catM G_\lambda$ is isomorphic to $\bigoplus_{k=0}^{p-1} \pi_{k\alpha_- +\lambda}$ by the Loewy series in Lemma \ref{lm_singletAbelianSteps} (9), the respective $\Hom$-spaces were obtained in Corollary \ref{cor:someHoms}. Note that on the top component $k=0$ there is an obvious preferred isomorphism, but not for general $k$, and we will choose such a family of isomorphisms $\eta_\bullet$ below. As vector spaces we have $M=\Localization(M)$ and $\Localization^T$ is trivial. We suppress the natural transformation $\Localization\circ \iota=\id$.

We have a  commuting diagram of functors 
    \begin{align*}
          \xymatrix{   
          \Rep(\Nichols)(\CC)\ar[rr]^F \ar[dr]_{\forgetNichols}
          &&
          \catM_A \ar[dl]^{\Localization} 
          \ar@{}[dll]^(.25){}="a"^(.55){}="b" 
          \ar@{=>}_\eta "a";"b"\\
          &
          \CC
          &
          }
    \end{align*}
    which requires us to fix a natural isomorphism $\eta_\bullet: \Localization\circ F\to \forgetNichols$. The central choice we have to make arbitrarily is for $A^*=\overline{F_{1,1}}=G_0$ in $\UU$, a coalgebra in $\UU$, which is the image under $F$ of $\Nichols=\C_{0,p}$, the regular representation of $\Nichols$, a coalgebra in $\Rep(\Nichols)(\CC)$. 

    $$\eta_{\Nichols}:\;\Localization(F(\Nichols))= \Localization(\induceA(A^*)) \longrightarrow \forgetNichols(\Nichols)=\bigoplus_{k=0}^{p-1} \pi_{k\alpha_-}$$

    From this we obtain the other projective indecomposables $\eta_{\C_{\lambda,p}}$ by right-multiplication with $\pi_\lambda$ as follows:
\begin{align*}
\eta_{\C_{\lambda,p}}:\;\Localization(\induceA(A^*)\boxtimes_{\catM_A} \iota\pi_\lambda)
&\stackrel{\Localization^T}{\longrightarrow}
\Localization(\induceA(A^*))\boxtimes_{\CC} \pi_\lambda\\
&\stackrel{\eta_{\Nichols}\otimes \id}{\longrightarrow}
\big( \bigoplus_{k=0}^{p-1} \pi_{k\alpha_-} \big)
\boxtimes_\CC \pi_\lambda\\
&\stackrel{ (e^T_{k\alpha_-, \lambda})^{-1}}{\longrightarrow}
\bigoplus_{k=0}^{p-1} \pi_{k\alpha_-+\lambda+\mu} 
\end{align*}
Having this commuting diagram of functors we can define the functor $F$ on Hom-spaces, i.e.~define $F(f):F(M)\to F(N)$ for $f:M\to N$ by
 \begin{align*}
          \xymatrix{   
            \Localization(F(M))
            \ar[rr]^{\Localization(F(f))} 
            \ar[d]^ {\eta_M}
            &&
            \Localization(F(N))
            \ar[d]^ {\eta_N} \\
            \forgetNichols(M)
            \ar[rr]^{\forgetNichols(f)} 
            &&
            \forgetNichols(N)
          }
\end{align*}
With this choice $\eta_\bullet$ is natural by definition.

We now discuss the equivalence of module categories and the compatibility with $\Localization,\forgetNichols$ in Lemma \ref{lm_infiniteSplittingIsHopfAlgB}:  The right $\CC$-module category action on $\Rep(\Nichols)(\CC)$ is given by 
$\C_{\lambda,l}.\pi_\mu=\C_{\lambda+\mu,l}$, the right $\CC$-module category action on $\catM_A$ is is given by right-tensoring in $\catM_A$ with the local module $\pi_\mu$. We extend the functor $F$  to a module category equivalence by fixing a tensor structure $F^T$
\begin{align*}
F(\oneside{\forgetNichols(M)\otimes_\CC V})
&\stackrel{F^T}{\longrightarrow} F(M)\boxtimes_{\catM_A} \iota(V)\\
\intertext{for which again it suffices to fix the case $M=\C_{\lambda,p}$, $V=\C_\mu$, which read explicitly }
F(\C_{\lambda+\mu,p})
&\stackrel{F^T}{\longrightarrow} F(\C_{\lambda,p})\boxtimes_{\catM_A} \iota\pi_\mu\\
\end{align*}
With our choice $F(\C_{\lambda,p})=\induceA(A^*) \boxtimes_{\catM_A} \iota\pi_\lambda$ there is the following canonical choice (note that if we would have taken $F(\C_{\lambda,p})=\induceA(G_\lambda)$  we would require instead a coherent choices for the isomorphism in Lemma \ref{lm_singletAbelianSteps} (6))

\begin{align*}
F^T:\;F(\C_{\lambda+\mu,p})
&=\induceA(A^*)\boxtimes_{\catM_A} \iota\pi_{\lambda+\mu}\\
& \stackrel{\iota^ T e^T_{\lambda,\mu}}{\longrightarrow} 
\induceA(A^*)\boxtimes_{\catM_A} \big(\iota\pi_{\lambda}\boxtimes_{\catM_A}\iota\pi_{\mu}\big)\\
&\stackrel{assoz}{\longrightarrow}  \big(\induceA(A^*))\boxtimes_{\catM_A} \iota\pi_{\lambda}\big)\boxtimes_{\catM_A}\iota\pi_{\mu}
=F(\C_{\lambda,p})\boxtimes_{\catM_A} \iota\pi_\mu
\end{align*}


We check naturality of $F^T$ with the definition above via $\eta$ of $F(f)$ for $f:\C_{\lambda,p}\to \C_{\lambda',p}$ a morphism in $\Rep(\Nichols)(\CC)$ - of which there exists a nontrivial one for $\lambda'-\lambda=0,\alpha_-,\cdots, (p-1)\alpha_-$.
\begin{equation}\nonumber
\resizebox{\textwidth}{!}{
$\begin{split}
\xymatrix{  
    &
            \bigoplus\limits_{k=0}^{p-1} \pi_{k\alpha_-+\lambda+\mu}
            \ar[dd]^{\substack{\vphantom{\bigoplus_0^\infty} \\ \vphantom{\bigoplus\limits_0^\infty} \\ \forgetNichols(f\otimes \id) }}
    && 
            \bigoplus\limits_{k=0}^{p-1} \pi_{k\alpha_-+\lambda} \boxtimes_{\CC}\pi_{\mu}
            \ar[dd]^{\forgetNichols(f)\otimes \id}
    \\
    \Localization(\induceA(A^*)\boxtimes_{\catM_A} \iota\pi_{\lambda+\mu})
    \ar[rr]^{\Localization^T\circ\Localization(F^T_{\C_{\lambda,p},\C_\mu})\qquad\qquad\qquad}
    \ar@{-->}[dd]^{\Localization(F(f\otimes \id))}
    \ar[ur]^{\eta_{\C_{\lambda+\mu,p}}}
    &&
    \Localization(\induceA(A^*) \boxtimes_{\catM_A}\iota\pi_{\lambda})\boxtimes_{\CC}\pi_{\mu}
    \ar@{-->}[dd]^{\Localization(F(f))}
    \ar[ur]^{\eta_{\C_{\lambda,p}}\otimes \id}
    &
    \\
    &
            \bigoplus\limits_{k=0}^{p-1} \pi_{k\alpha_-+\lambda'+\mu}
    && 
            \bigoplus\limits_{k=0}^{p-1} \pi_{k\alpha_-+\lambda'} \boxtimes_{\CC}\pi_{\mu}
    \\    
    \Localization(\induceA(A^*)\boxtimes_{\catM_A} \iota\pi_{\lambda'+\mu})
    \ar[rr]^{\Localization^T\circ\Localization(F^T_{\C_{\lambda',p},\C_\mu})\qquad\qquad\qquad}
    \ar[ur]^{\eta_{\C_{\lambda'+\mu,p}}}
    &&
    \Localization(\induceA(A^*) \boxtimes_{\catM_A}\iota\pi_{\lambda'})\boxtimes_{\CC}\pi_{\mu}
    \ar[ur]^{\eta_{\C_{\lambda',p}}\otimes \id}
    & \\
}
\end{split}$
}
\end{equation}

 Following the diagram from $\Localization(\induceA(A^*)\boxtimes_{\catM_A} \iota\pi_{\lambda+\mu})$ to 
 $\big(\Localization(\induceA(A^*) \boxtimes_{\catM_A}\iota\pi_{\lambda'})\big)\boxtimes_{\CC}\pi_{\mu}$ in the two possible ways gives in each cases only various $e^T$ which are overall equal by the pentagon axiom of a tensor structure.

    We finally check that $F^T$ becomes trivial on $\CC$ after application of $\Localization,\forgetNichols$ in the sense of condition (\ref{ModCatCondition}), which is precisely the missing part in the previous diagram. 
\begin{equation}\nonumber
\resizebox{\textwidth}{!}{
$\begin{split}
\xymatrix{  
        &
                \bigoplus\limits_{k=0}^{p-1} \pi_{k\alpha_-+\lambda+\mu}        
                \ar[rr]^{\forgetNichols^T}
        &&
                \left(\bigoplus\limits_{k=0}^{p-1} \pi_{k\alpha_-+\lambda}\right)\boxtimes_\CC \pi_\mu
        \\
        \Localization(\induceA(A^*)\boxtimes_{\catM_A} \iota\pi_{\lambda+\mu})          \ar[rr]^{\Localization^T \circ \Localization(F^T_{\C_{\lambda,p},\C_\mu})} 
          \ar[ur]^{\eta_{\C_{\lambda+\mu,p}}}
          &&
        \Localization(\induceA(A^*) \boxtimes_{\catM_A}\iota\pi_{\lambda})\boxtimes_{\CC}\pi_{\mu}
        \ar[ur]^{\eta_{\C_{\lambda,p}}\otimes \id}
        &
        \\
}
\end{split}$
}
\end{equation}
This diagram commutes again because the both ways give essentially the following morphisms wich again agree by the pentagon axiom 
$$e^T_{k\alpha_-+\lambda,\mu}\circ (e^T_{k\alpha_-,\lambda+\mu})^{-1}
=(e^T_{k\alpha_-,\lambda})^{-1} \circ e^T_{\lambda,\mu}. $$\qedhere
\end{proof}

\begin{remark}
We remark that the only actual choice we made in the proof above was the identification of the underlying coalgebras 
  $$\eta_{\Nichols}:\;\Localization(F(\Nichols))= \Localization(\induceA(A^*)) \longrightarrow \forgetNichols(\Nichols)=\bigoplus_{k=0}^{p-1} \pi_{k\alpha_-}$$
and all the other structures followed readily from the definition of the projectives $\UU_A$ as $\induceA(A^*)\otimes \pi_\lambda$. Since this is in general the shape of the projectives for a nilpotent algebra, we expect that the proof carries over to more complicated situations.
\end{remark}

\subsection{Recognizing the tensor structures.}

We now apply the characterization result Theorem \ref{thm_CharacterizingUsingB}  and the explicit calculations in the present section for certain categories of modules $\UU=\mathcal O^T_{\mathcal M(p)}$ of the singlet vertex algebra $\mathcal{M}(p)$ to prove as our main application. 

\begin{theorem}\label{thm_singletEquivalence}
We have equivalences of braided tensor categories
\begin{align*}
\mathcal O^T_{\mathcal M(p)} &\cong \Rep_{\mathrm{wt}}(u_q^ H(\sl_2))
\end{align*}
\end{theorem}
\begin{proof} We proceed in the following steps, that could be similary undertaken for more complicated examples.  
\begin{enumerate}
    \item We have a free field realization in the Heisenberg algebra $\pi\supset \mathcal{M}(p)$. 
    Correspondingly we have a commutative algebra $A$ in $\UU$, namely $\pi$ over $\mathcal{M}(p)$, such that $\UU_A^0=\Rep(\pi)=\Vect^Q_\C$.
    \item In Theorem \ref{thm_SingletAbelian} we determined explicitly for the singlet that the abelian category $\UU$ is equivalent to the abelian category of representations of the unrolled quantum group of $\sl_2$ as an algebra, finite-dimensional and with semisimple action of the Cartan part. 
    \item In Theorem \ref{thm_SingletTwistedAbelian} we have used Frobenius reciprocity to determine explicitly for the singlet that the abelian category $\UU_A$ is equivalent to the abelian category $\Rep(\Nichols)(\CC)$ for the algebra $\Nichols=\C[x]/x^p$ in $\CC=\Vect^Q_\C$.
   \item Thus Theorem \ref{thm_CharacterizingUsingB}  shows the assertion. \qedhere
\end{enumerate}
\end{proof}

\begin{remark}\label{rem_triplet}
    We will obtain in the next section the analogous result 
    $$\mathcal O^T_{\mathcal W(p)} \cong \Rep(u_q(\sl_2))$$
    by a simple current extension, but it is instructive to briefly discuss how the previous proof would change if applied directly  to the triplet vertex algebra $\cW(p)$: 
    \begin{enumerate}
    \item[1']
    The triplet  has a free field realization in a lattice vertex algebra $\cW(p)\subset \V_{\sqrt{p}\Z}$. In consequence hus the category of representations $\UU=\Rep(\cW(p))$ contains a commutative algebra $A$ with $\UU_A^0=\Vect_{\Z_{2p}}$.
    \item[2'] The computation analogous to Theorem \ref{thm_SingletAbelian} can be found in \textup{\cite{McRY}}.
    \item[3'] The computations in Theorem \ref{thm_SingletTwistedAbelian} could be easily repeated for the triplet vertex algebra, but some care has to be taken because of cyclicity.
    \item[4'] Then again Theorem \ref{thm_CharacterizingUsingB}  shows the assertion. Note that we now really need the stronger Lemma \ref{lm_isNichols} instead of Lemma \ref{lm_isNicholsFree} to recognize the Nichols algebra, because it is not a-priori $\Z$-graded but only $\Z_{2p}$-graded. 
    \end{enumerate}
\end{remark}


\section{Characterization of general quantum groups}

In this section we prove a general version of the characterization of $u_q(\sl_2)$ in the last section for general $\g$, and beyond for arbitrary diagonal Nichols algebras, see Definition \ref{def_HechenbergerQG}. For simplicity we now restrict ourselves to the non-quasi-case, since the quasi-Hopf algebra cases can be recovered by uprolling as we will explain in Section \ref{sec_uprolling}.

\subsection{Relative Hopf modules}

\renewcommand{\c}{c}

The theory of Hopf modules, developed notably in \cite{Tak79} 
and \cite{Skry06}, can be used to compute the category of modules $\UU_A$ for a certain class of algebras $A$ in  the tensor category $\UU$ of representations of a Hopf algebra $U$. 

More precisely, assume $A$ is of the following form: Let $\pi: U\to A^*$ be a coalgebra surjection, which turns $A^*$ into a coalgebra in the category $\Rep(U)$, and let $A$ be the dual algebra, which we assume to be finite-dimensional. There is an adjunction between the linear categories 
$\UU_A$ and $\Rep(B)$ for the algebra $B$ of $\pi$-coinvariants
$$B=U^{coin(\pi)}=\{h\in U\mid \pi(h^{(1)})h^{(2)}=1\otimes h\}$$
Here, an object  $X\in \UU_A$ is sent to the subspace of $A$-invariants $X^A$, which is a module over $B \subset U$. Conversely, an object $Y\in \Rep(B)$ is sent to the induced module $U\otimes_B Y$.\\

The main technical assertion in \cite{Skry06} is the following: Recall that an algebra $U$ is weakly finite if a surjective endomorphism of a free module of finite rank is automatically surjective, as in Corollary \ref{cor_SchauenburgInfinite}. Assume that $U$ is weakly finite and $A$ is simple, then by \cite{Skry06} Thm 7.6 the locally finite modules in $\UU_A$ are projective (and frequently free) as $A$-modules. Using this, the dual statement to \cite{Skry06} Lemma 6.2 (see also \cite{Tak79} Theorem 2) shows that the previous adjoint functors in fact give an equivalence of abelian  categories:

\begin{theorem}\label{thm_skry}
Let $U$ be weakly finite and $A$ a finite-dimensional simple algebra in $\Rep(U)$ given by a surjection of coalgebras $\pi: U\to A^*$ and $B$ the algebra of coinvariants as above, then the category of locally finite modules is
$$\UU_A^{loc. fin.}\cong \Rep(B)$$
\end{theorem}

We need a relative version: Assume there is a commutative Hopf subalgebra $C\subset U$ of finite index, such that $\pi|_C=\epsilon_C$ and so $B\supset C$. Denote the restriction of the one-sided quotient $U/UC^+$ with the adjoint $C$-action to a $C$-module by $\c(U/UC^+)$ and the restriction of $A$ to a $C$-module by $\c(A)$. Denote on the other hand the restriction of the one-sided quotient $B/BC^+$ by $\c(B/BC^+)$. We interprete $\c(X)$ as a $C$-graded dimension. From the induction functor in the previous abelian equivalence we have $A\cong U\otimes_B \C_\epsilon$, so 

\begin{corollary}\label{cor_gradeddim}
We have an isomorphism of $C$-modules
$$\c(U/UC^+) \cong \c(B/BC^+) \otimes \c(A)$$
\end{corollary}

Note also that $C$-semisimple representations in $\UU_A$ are mapped to $C$-semisimple representations in $\Rep(B)$.

\begin{problem}[Takeuchi-Skryabin categorically]
    We restrict ourselves to $\UU=\Rep(U)$ for a Hopf algebra, but the ideas should actually apply in a general situation where $\UU\to \CC$ is some tensor functor to the local modules of $A$ in $\UU$. However, the method relies on the theorems for relative Hopf moduels by Takeuchi and Skryabin in the next section, which are very algebraically formulated (similar to our relative finiteness), which have to our knowledge never been worked out in a categorical context, although this seems very promising. A good strarting point should be the characterization in \cite{Bi23}.

    In particular, the different niceness conditions on an algebra $A$ in $\UU$ should be related under sufficient finiteness asusmptions: $A$ Frobenius, $A$ simple algebra, $\forgetA(X)$ projective implies $X$ projective.  
\end{problem}


\subsection{Fixing the twist}

Let $\UU=\Rep(U)$ for a Hopf algebra $U$. Let $A^*$ be a coalgebra in $\UU$, which as a left $U$-module admits a surjection $\pi:U\to A^*$ (typically an induced module). 

\begin{lemma}\label{lm_fixingtwist}
    There exists an invertible element $J\in U\otimes U$ such that the image under the tensor $F^J:\Rep(U)\to \Rep(U^J)$, given by the identity on objects and morphisms and the tensor structure $J$, the surjection $\pi$ is a surjection of coalgebras:
    $$F^J(\pi):U^J\to F^J(A^*)$$  
\end{lemma}
\begin{proof}

    We define $\bar{J}=\Delta_{A^*}(\pi(1_U))$ and claim that this element generates $A^*\otimes A^*$ as $U\otimes U$-module: Indeed we have the codimension 1 submodule $\ker(\epsilon_{A^*})\otimes A^* + A^*\otimes \ker(\epsilon_{A^*})$ which does not contain $\bar{J}$ because of counitality of $A^*$, so $\bar{J}$ is a generator.
    
    Choose $J\in U\otimes U$ with $J.\bar{J}=\pi(1_U) \otimes \pi(1_U)$. Then the coproduct on $F^J(A^*)$, by definition of $F^J$, is $\Delta_{A^*}^J(a)=J.\Delta_{A^*}(a)$, in particular we have by construction 
    \[
    \Delta_{A^*}^J(\pi(1_U))= J.\Delta_{A^*}(a) =  J.\bar{J} =  \pi(1_U)\otimes \pi(1_U).
    \]
 In consequence 
    $$\Delta_{A^*}^J(\pi(u))=\Delta^{U^J}(u).\Delta_{A^*}^J(\pi(1_U))
    =\Delta^{U^J}(u).(\pi(1_U)\otimes\pi(1_U))=\pi(u^{(1)})\otimes \pi(u^{(2)})$$
\end{proof}


\subsection{Main theorem}

   
   
    

 Let $\NicholsOf(X)$ be a finite-dimensional Nichols algebra in $\CC=\Rep^{wt}(C)$, the finite-dimensional semisimple representations of a (possibly infinite) Hopf algebra $C$. Let $U_q\supset C$ be the corresponding quantum group constructed in Example \ref{exm_realizedQuantumGroup} with the property that the category of finite-dimensional $C$-semisimple representations $\Rep^{wt}(U_q)$ is equivalent to $\UU_q=\YD{\Nichols}(\CC)$, which is a braided tensor category with trivial M\"uger center and locally finite.
 
\begin{theorem}\label{thm_characterizeQG}
  Let $U\supset C$ be a Hopf algebra with $\UU=\Rep^{wt}(U)$ a braided tensor category with trivial M\"uger center and $A$ is a finite-dimensional commutative simple algebra object in $\UU$. Assume that 
  \begin{itemize}
  \item $U=U_q$ as an algebra.
  \item $A$ as an algebra in $\Rep^{wt}(U)$ is the dual Verma module $\bV_0^*$, see Definition \ref{def_Verma}.
  \item $\NicholsOf(X)\in\CC$ is sufficiently unrolled in the sense of Definition \ref{def_sufficientlyUnrolled}.
  \item the category $\UU_A^0$ of local $A$-modules is equivalent as a braided tensor category to $\CC$, compatible with the inclusion $C\subset U$ in the sense that the following  commutes
  \begin{center}
$$
 \begin{tikzcd}[row sep=10ex, column sep=15ex]
   \Rep^{wt}(U) \arrow{r}{\induceA}  
   \arrow{rd}{\mathrm{res}_C}
   & 
   \Rep(A)(\Rep^{wt}(U))  \arrow{d}{A-\mathrm{inv}}\\
   & 
   \Rep^{wt}(C)
    \end{tikzcd}
$$
\end{center}

  \item the categories $\UU$ and $\UU_A$ are rigid.
  \end{itemize}
  Then we have $B\cong \NicholsOf(X)\rtimes C$ and $U\cong U_q$ as Hopf algebras, and accordingly an equivalence of braided tensor categories $\UU=\UU_q$.
\end{theorem}
\begin{corollary}\label{cor_fixingtwist}
With Lemma \ref{lm_fixingtwist} it is sufficient to assume that  $A$ as a $U$-module comes from a module surjection $\pi:U\to A^*$, up to a possible $2$-cocycle twist of $U$. The condition does then not depend on knowing the coalgebra structure of $U$.
\end{corollary}
\begin{proof}[Proof of Theorem \ref{thm_characterizeQG}]~

a) We apply Theorem \ref{thm_skry} to compute $\UU_A$ in terms of $U$ with unknown coalgebra structure: The prerequisite that $U$ is weakly finite holds, because $C\subset U$ has finite index. The prerequisite that the algebra $A$ comes from a coalgebra surjection $\pi:U\to A^*$ is true by assumption. Hence we can apply Theorem \ref{thm_skry} which yields $\Rep(U)_A\cong \Rep(B)$ for the Hopf algebra $B=U^{coinv(\pi)}$ and by Corollary  \ref{cor_gradeddim} we have $C\subset B$ and $\Rep^{wt}(U)_A\cong \Rep^{wt}(B)$ and $B/BC^+$ under the adjoint action of $C$ (to be interpreted as a graded dimension) is isomorphic to  $\NicholsOf(X)$.

b) From the assumed sufficient unrolling we have by Corollary \ref{cor_isNicholsRadfordFree} that  $C$-module structure on $B$ (to be interpreted as strictly positive grading) implies the existence of a splitting functor, so that $B\cong \Nichols\rtimes C$ for some Hopf algebra $\Nichols$ in $\CC$ which is isomorphic to $\NicholsOf(X)$ as a $C$-module. 

c) From the assumed sufficient unrolling we have by Lemma \ref{lm_isNicholsFree} that the isomorphism of $\Nichols\cong \NicholsOf(X)$ as $C$-modules (to be interpreted as a matching graded dimension) implies that $\Nichols\cong \NicholsOf(X)$ as Hopf algebras in $\CC$.

d) The braided tensor category equivalence is a consequence of Lemma \ref{lm_infiniteSplittingIsHopfAlgB}. The rigidity holds by assumption, being haploid holds for the Verma module by construction, the relative finiteness conditions hold since $U\supset C$ has finite index. 
\end{proof}

 \begin{problem}[Skryabin with tensor structure]
     Suppose now that $U$ is quasitriangular and $A$ is a commutative algebra. Are there conditions under which $B\subset H$ is in fact a Hopf subalgebra and $\UU_A\cong \Rep(B)$ as tensor categories?  
 \end{problem}

 \subsection{Application to Vertex algebras and free-field realizations}

        Theorem \ref{thm_characterizeQG} and the improvement in Corollary \ref{lm_fixingtwist} is taylored to the situation of a free-field realization of a vertex algebra $\V\supset \cW$ with known tensor category $\Rep(\V)=\CC$ and known abelian category $\UU$. We now summarize our general assumptions:

        \begin{assumption}
            Let $\V$ be a vertex operator algebra, which is simple over itself and its own contragradient dual, and let $\CC$ be a vertex tensor category of $\V$-modules, which is rigid and has trivial M\"uger center. Assume that as tensor category $\CC\cong \Rep^{wt}(C)$, the finite-dimensional semisimple modules over a commutative Hopf algebra $C$.

            Let $\cW\subset \V$ be a vertex operator subalgebra, which is simple over itself and its own contragradient dual, and let $\UU$ be a vertex tensor category of $\W$-modules, which is rigid and has trivial M\"uger center. Assume that as abelian  category $\UU\cong \Rep^{wt}(U)$, the finite-dimensional $C$-semisimple modules over an algebra $U\supset C$.
            
            Assume that $\V$ as $\cW$-module is in $\UU$ and that $A=\V$ is simple, haploid and finite-dimensional as an algebra in $\UU$ and assume the tensor category $\UU_A$ is rigid. 
        \end{assumption}

        \begin{theorem}
            In this situation, assume that $U=U_q$ as an algebra, so $\UU\cong \UU_q$ as abelian categories, and $A=\bV_0^*$ as an object in $\UU$, as in Theorem \ref{thm_characterizeQG}. Assume that we have a quasi-tensor functor $\c:\UU\to \CC$, meaning there is a weights-space associated to every object, which is compatible under fusion product. Assume further that we have a second commutative algebra $\bar{A}$ fulfilling all the assumptions above, with $\bar{A}=\bar{\bV}^*_0$ the Verma module for the opposite Borel. Assume that the tensor product $\bV_\lambda \otimes \bar{\bV}_0$ in $\UU$ contain the projective cover of the simple module $\mathbb{L}_\lambda$ in $\UU$ for all weights $\lambda$. Then $\UU\cong\UU_q$ as braided tensor categories.
        \end{theorem}
        \begin{proof}
        Using Theorem \ref{thm_characterizeQG} and Corollary \ref{lm_fixingtwist} it is sufficient to prove that  $\c:\UU\to \CC$ is an actual tensor functor for which we proceed as we did in \cite{CLR} Proposition 3.6: We have an oplax tensor functor 
        $$\coVerma:\,\CC=\UU_A^0\stackrel{\iota}{\longrightarrow} \UU_A \stackrel{\forgetA}{\longrightarrow} \UU$$
        Since $1_\CC$ is a coalgebra in $\CC$ its image $\coVerma(1_\CC)=A^*$ is a coalgebra $(A^*,\delta)$ in $\UU$. The coassociativity reads
        $$A^*\stackrel{\delta}{\longrightarrow} A^*\otimes_\UU A^*
        \stackrel{\delta\otimes \id}{\longrightarrow} (A^*\otimes_\UU A^* )\otimes_\UU A^*
        \xrightarrow{a_{A^*,A^*,A^*}} A^*\otimes_\UU (A^* \otimes_\UU A^*)
        \stackrel{\id \otimes \delta}{\longleftarrow} A^*\otimes_\UU A^*
        \stackrel{ \delta}{\longleftarrow} A^*
        $$
        From Corollary \ref{lm_fixingtwist} we can already assume up to twist that $\delta(1)=1\otimes 1$ where $1$ denotes the highest-weight vector of $\bV_0$. But since this coproduct is strictly coassociative, we have in this case $a_{A^*,A^*,A^*}=\id$.\\

        We now repeat the steps above for larger coalgebras in $\UU$: By assumption we have the coalgebras $A^*,\,\bar{A}^*$, and since $\UU$ is braided the tensor product $A^* \otimes_\UU \bar{A}^*$ becomes again a coalgebra, and up to twist we can assume $\delta(1\otimes \bar{1})=(1\otimes \bar{1})\otimes (1\otimes \bar{1})$ and hence we find that $a$ on a triple of these modules is trivial.\\

        We now consider the family of $\CC$-objects $\C_\lambda$ indexed by $\lambda\in\mathrm{Spec}(C)$ and $\CC$-morphisms   $\C_\lambda\to \C_{\lambda'}\otimes_\CC \C_{\lambda''}$  for all $\lambda=\lambda'+\lambda''$. The family fulfills a graded version of coassociativity - however for infinite $C$ is is not a coalgebra in the strict sense, but in an ind-completion sense. Nevertheless, applying the oplax tensor functor $\coVerma$ we obtain a 
        family of $\UU$-objects $\bV_\lambda$ and $\UU$-morphisms $\delta_{\lambda',\lambda''}:\,\bV_\lambda^*\to \bV_{\lambda'}^*\otimes_\UU \bV_{\lambda''}^*$ again with a version of coassociativity. Again we can assume up to twist that on the highest weight vectors $\delta(1_\lambda)=1_{\lambda'}\otimes 1_{\lambda''}$, which is strictly coassociative, so the argument above proves in this case  for all $\lambda,\mu,\nu$
        $$a_{\bar{\bV}_\lambda^*,\,
        \bar{\bV}_\mu^*,\,
        \bar{\bV}_\nu^*}=\id$$    
        Similarly we can again tensor 
        with $\bar{\bV}_0^*=\bar{A}^*$ and the argument above proves for all $\lambda,\mu,\nu$
        $$a_{\bar{\bV}_\lambda^*\otimes \bar{\bV}_0^*,\,
        \bar{\bV}_\mu^*\otimes \bar{\bV}_0^*,\,
        \bar{\bV}_\nu^*\otimes \bar{\bV}_0^*}=\id$$ 
        
        But since we assumed that ${\bV}_\lambda^*\otimes \bar{\bV}_0^*$ contains the projective covers of all simple modules, this means that $a_{\bullet,\bullet,\bullet}$ is trivial altogether and $\c$ is a tensor functor.
        \end{proof}
        \begin{example}\label{ex_singlet}
        For $u_q^H(\sl_2)$ we have checked   the assumptions on the fusion product \textup{\cite[Section 2.1]{CLR}} and the two free-field realizations \textup{\cite[Section 2.2]{CLR}} (and then proceeded to determine the tensor category for $p=2$ in a more ad-hoc fashion). These calculations together with the preceding theorem give now an independent proof of Theorem \ref{thm_singletEquivalence}.
        \end{example}

    Under certain convergence conditions, the algebra of screenings is the Nichols algebra \cite{Len21}. The question which lattices realize $\CC$ in a way compatible with reflection operators was treated in \cite{FL22}, which also addreses the questions for which choices the convergence conditions holds.

\begin{problem}[\cite{FL22}]
For any vertex algebra $\V$ and a set of non-local screening operators, what is the braided category of representations of the kernel of screenings $\cW$? A good candidate in general would be the relative center of the representations of the algebra of screenings $\Nichols$ inside $\CC=\Rep(\V)$.
\end{problem}

\section{Uprolling}\label{sec_uprolling}

We want to describe how the results of our articles commute with a second extension with an algebra $S$, in applications typically a simple-current extension.

\subsection{Commutative algebra statement}

Let $A,S$ be commutative algebras in a braided tensor category, such that $c_{A,S}\circ c_{S,A}=\id$, in particular $A\otimes S$ is again a commutative algebra in $\UU$. Moreover, the induction functors $\mathrm{ind}_A$ resp. $\mathrm{ind}_S$ map $S$ resp. $A$ to the local modules and hence to commutative algebras in $\mathrm{ind}_A(S)\in \UU_A^0$ resp $\mathrm{ind}_S(A)\in \UU_S^0$. Clearly

\begin{lemma}
    We have an equivalence of braided tensor categories 
    $$\big(\UU_A^0\big)_{\mathrm{ind}_A(S)}^0 \cong 
    \big(\UU_S^0\big)_{\mathrm{ind}_S(A)}^0 \cong
    \UU_{A\otimes S}^0 $$
\end{lemma}
In the notation of the paper for the $A$-extension this could be rewritten 
$$\CC=\UU_A^0 \quad \Longrightarrow \quad \tilde{\CC}=\tilde{\UU}_{\tilde{A}}^0$$ 
where we use tilde for the $S$-extension $\tilde{\CC}=\CC_{\mathrm{ind}_A(S)}^0$ and $\tilde{\UU}=\UU_S^0$ and $\tilde{A}=\mathrm{ind}_S(A)$.

\subsection{Yetter-Drinfeld module statememt}

Let $\Nichols$ be a Hopf algebra in $\CC$ and $R$ be a commutative algebra in $\CC$, such that $c_{\Nichols,R}\circ c_{R,\Nichols}=\id$. Then since induction is a tensor functor it defines a Hopf algebra  
$$\tilde{\Nichols}:=\mathrm{ind}_R(\Nichols) \in \CC_R^0$$
On the other hand we can endow $R$ with trivial $\Nichols$-action and -coaction via the counit $\epsilon:\Nichols\to 1$, which is by Definition \ref{def_YD} precisely a Yetter-Drinfeld module iff $c_{\Nichols,R}\circ c_{R,\Nichols}=\id$. Denote this object by 
$$S:=R_\epsilon^1$$

\begin{theorem}\label{thm_uprolling}
    We have an equivalence of braided tensor categories
    $$\big(\YD{\Nichols}(\CC)\big)_{S}^0
    \cong \YD{\tilde{\Nichols}}(\CC_R^0) $$
\end{theorem}
\begin{proof}
An object on the left-hand side is a $\Nichols$-Yetter Drinfeld module in $\CC$ with a  $\Nichols$-linear and  $\Nichols$-colinear action of $S$, which is local. By definition of $S$ this is a local $R$-action commuting with  $\Nichols$-action and  $\Nichols$-coaction.

An object on the right-hand side is a local $R$-module $N$ with an $R$-linear action and coaction by $\tilde{\Nichols}=\mathrm{ind}_R(\Nichols)$. By definition of the tensor product this means
$$\mathrm{ind}_R(\Nichols)  \otimes_R M  \to M 
\qquad M  \to \mathrm{ind}_R(\Nichols)  \otimes_R  M $$
Composing these with the natural transformation in  Lemma \ref{lm_trivialActionToInducedModule} gives equivalently 
 $\CC$-morphisms 
$$\Nichols  \otimes M  \to M,\qquad M  \to \Nichols  \otimes  M $$
which are $R$-linear with respect to the action on the tensor factor $M$ only, or differently spoken: $\Nichols$-action and $\Nichols$-coaction must commute with the $R$-action on $M$. The Yetter-Drinfeld condition in $\CC_R^0$ in this way becomes the usual Yetter-Drinfeld condition, because the braiding in $\CC_R^0$ comes from the braiding on $\CC$.

But now we have unpacked the definition on both sides to the very same objects, namely objects $M\in \CC$ with local $R$-action and $\Nichols$-action and  $\Nichols$-coaction fulfilling the Yetter-Drinfeld-condition, such that the $R$-action commutes with the  $\Nichols$-action and  $\Nichols$-coaction.
\end{proof}

\begin{corollary}\label{cor:uprolling}
    Suppose we identified $\UU\cong \YD{\Nichols}(\CC)$ and we have a commutative algebra $S\in \UU$, whose image in $\YD{\Nichols}(\CC)$ is of the form $S=R_\epsilon^1$ for some $R\in\CC$ with $c_{\Nichols,R}\circ c_{R,\Nichols}=\id$, then the extension of $\UU$ can be described in terms of the extension $\tilde{\CC}=\CC_R^0$ as
    $$\UU_S^0=\YD{\tilde{\Nichols}}(\tilde{\CC})$$ 
    where $\tilde{\Nichols}$ is the induction of $\Nichols$ in $\CC$ to $\tilde{\CC}$.
\end{corollary}

\begin{remark}
In general it is difficult for a commutative algebra $S$ in $\UU=\Rep(U)$ to determine a realizing Hopf algebra $U_S^0$ for $\UU_S$, even for a simple current extension $S=\bigoplus_i \C_{\chi_i}$ with representations $\chi_i:U\to \C$. Note that typically we get only a quasi-splitting tensor functor (and hence a quasi-Hopf algebra) or no splitting tensor functor at all. However, a realizing algebra for the abelian category $\Rep(S)(\UU)$ can be determined again by Skryabins result in Theorem  \ref{thm_skry}, namely
$$U_S=U^{coin(S)}=\{u\in U\mid \chi_i(u^{(1)})u^{(2)}=1\otimes u\}$$
and the local modules are clearly those modules factorizing over the Hopf algebra quotient 
$$U_S^0=U_S/I^+U_S,\quad I=\bigoplus_i(\chi_i\otimes\id)R^2$$
\end{remark}
\begin{problem}
Does the commutativity of $S$ imply that $U^{coin(S)}$ is in fact a Hopf algebra and that $U_S\cong\Rep(U^{coin(S)})$ as tensor categories? Or does this at least hold in presence of a splitting functor?
\end{problem}

\subsection{Example the triplet \texorpdfstring{$\mathcal W(p)$}{W(p)}}

In Theorem \ref{thm_singletEquivalence} we have proven for the singlet vertex algebra $\cM(p)$ for $p\in\Z_{\geq 2}$ that there is an  equivalence of braided tensor categories 
$$\UU=\mathcal{O}_{\cM(p)}^T \cong \YD{\Nichols}(\CC)$$
where $\CC=\Vect_\C^Q$ with quadratic form $Q(\lambda)=e^{\pi\i\, \lambda^2}$ and $\Nichols=\C[x]/x^p$ is the Nichols algebra generated by $x$ in degree $\alpha_-=-\sqrt{2/p}$. The category $\UU$ is realized by the unrolled quantum group $u_q^H(\sl_2)$ at $q=e^{\pi\i/ p}$ .\\

The singlet vertex algebra has a simple current extension by $S=R_1^\epsilon$ described in equation (\ref{eq_TripletAsModule}) to the triplet vertex algebra $\cW(p)$, where 
$$R=\bigoplus_{n\in \Z} \C_{n\alpha_+},\qquad \alpha_+=-p\alpha_-=\sqrt{2p} $$
Hence by Corollary \ref{cor:uprolling} we have 
\begin{corollary}\label{cor_triplet}
There is an an equivalence of braided tensor categories 
$$\tilde{\UU}=\mathcal{O}_{\cW(p)}^T \cong \YD{\tilde{\Nichols}}(\tilde{\CC})$$
where the simple current extension of $\tilde{\CC}=\CC_R^0$ corresponding to the lattice vertex algebra of $\alpha_+\Z$  is explicitly
$$\tilde{\CC}=\CC_R^0=\Vect_{(\alpha_+\Z)'/(\alpha_+\Z)}^{\tilde{Q}}=\Vect_{\frac{1}{\sqrt{2p}}\Z/\sqrt{2p}\Z}^{\tilde{Q}}=\Vect_{\Z_{2p}}^{\tilde{Q}}$$ 
and $\tilde{Q}$ is the corresponding quadratic form on $\Z_p$ and $\tilde{\Nichols}=\C[\tilde{x}]/\tilde{x}^p$ with $\tilde{x}$ in degree $\bar{\alpha_-}\in (\alpha_+\Z)'/(\alpha_+\Z)$ resp. $-2\in \Z_{2p}$.
\end{corollary}
Note however that while $\tilde{Q}(\bar{\lambda})=e^{\pi\i\, \bar{\lambda}^2}$ and the associated bimultiplicative form $\tilde{B}(\bar{\lambda},\bar{\mu})=e^{2\pi\i\, \bar{\lambda}\bar{\mu}}$ is well-defined up to $\frac{1}{\sqrt{2p}}$, the braiding 
$\tilde{\sigma}(\bar{\lambda},\bar{\mu})=e^{\pi\i\, \bar{\lambda}\bar{\mu}}$ is not multiplicative any more and depends on choices. This means that $\tilde{\CC}$ requires a non-trivial associator, a ubiquitous phenomenon for modular tensor categories associated to quadratic spaces, see for example \cite{FRS02} Section 2.5.1. In turn $\tilde{\UU}$ cannot be realized by a Hopf algebra as in  Example \ref{exm_realizedQuantumGroup}, but as a quasi-Hopf algebra. Such a factorizable quasi-variant of a quantum group $u_q(\sl_2)$ for an even-order root of unity $q=e^{\pi\i/ p}$ has been constructed in \cite{CGR20} using precisely the simple-current extension above, and the corresponding construction for $u_q(\g)$ using Yetter-Drinfeld modules in $\tilde{\CC}$ can be found in \cite{GLO18}.

\subsection{Example \texorpdfstring{$\mathcal S(p)$}{S(p)} and \texorpdfstring{$\mathfrak{gl}_{1|1}$}{gl(1|1)}}

We consider the vertex algebra $\mathcal S(p)$ for $p \in \Z_{\geq 3}$ in Section \ref{catMA}. 
Its associated free field algebra is a Heisenberg VOA times a pair of free fermions and so the category $\tilde{\CC}$ is $\text{sVect} \boxtimes \Vect$. The corresponding Hopf algebra is $\C[(-1)^F, K, K^{-1}, H]$ with the relation $K = q^{(2-p)H}$ on modules and $q= e^{\frac{\pi i }{p}}$. 
The braiding is given by \eqref{braidingSp} and so in particular the $R$-matrix is $R = q^{(2-p) H \otimes H} (-1)^{F \otimes F}$.
The Nichols algebra is $\C[x]/x^p$ with $x$ having degree $(-1, 1)$ with self-braiding $q^{(2-p)}(-1)=q^2$.  We have  $\sigma((-1,+1),(\pm1,\lambda))=\pm q^{(2-p)\lambda}$  so $g=\bar{g}=(-1)^F K$. 

Using the method of Example \ref{exm_realizedQuantumGroup}, the corresponding Hopf algebra $U(S(p))$ is generated by the previous ones together with $x, x^*$, subject to the relations 
\begin{equation}\label{usp}
\begin{split}
    x^p &= 0 = (x^*)^p \\
    (-1)^F x (-1)^F &= -x, \qquad  (-1)^F x^* (-1)^F = -x^* \\
    [H,x] &= x,\qquad [H,x^*]=-x^* \\
    KxK^{-1} &= -q^{2} x, \qquad Kx^*K^{-1} =  -q^{-2}x^*\\
    xx^*-q^2 x^*x &= 1- K^{2} \\
    &\\
    \Delta((-1)^F)&=(-1)^F\otimes (-1)^F \\
    \Delta(H)&=1\otimes H+H\otimes 1 \\
    \Delta(x)&=(-1)^F K\otimes x+x\otimes 1 \\
    \Delta(x^*)&=(-1)^F K\otimes x^*+x^*\otimes 1 \\
\end{split}
\end{equation}

\begin{lemma}\label{lem_sp}
 The vertex tensor category of representations of the vertex superalgebra $\mathcal S(p)$ that are in $\mathcal O^T_{\mathcal M(p)}$ is braided equivalent to the category of weight modules of $U(S(p))$
\end{lemma}
\begin{proof}
 $\mathcal S(p)$ is a simple current extension of $\cM(p) \otimes \pi^\beta$ \eqref{Sp}.  By Theorem \ref{thm_singletEquivalence} the latter is equivalent to    $\YD{\Nichols}(\CC)$ for $\CC = \Vect_{\C^2}^Q$ and $\Nichols = \C[x]/x^p$ with $x$ in degree $(1, 0)$ and quadratic form 
 $$
 Q(\lambda)=q^{\lambda^T\begin{pmatrix} \frac{2}{p} & 0 \\ 0 & 1\end{pmatrix}\lambda}
 $$
 $\mathcal S(p) = \mathcal R(p)_1^\epsilon$ for $\mathcal R(p)$ in \eqref{Rp} by construction. The simple current extension $\CC_R^0$ is precisely the previous braided tensor category $\tilde{\CC}=\text{sVect} \boxtimes \Vect$ after the coordinate change $(\epsilon,\gamma)= (\frac{\alpha_+}{2}\alpha + \mu_p \beta,    \;-\frac{\alpha_+}{2} \alpha + \frac{p}{2\mu_p}\beta)$. 
 The degree of $x$ is the same as the one of the Fock module $\pi^\alpha_{\alpha_-} \otimes \pi^\beta$ and so in these new coordinates it is $(-1, 1)$.
We are exactly in the situation of Corollary \ref{cor:uprolling} and so the Lemma follows. 
\end{proof}

We also consider the vertex algebra $V^k(\mathfrak{gl}_{1|1})$, which can be considered as the degenerate case of the previous family for $p=2$. The free-field realization in Section \ref{catMA} in $\cM(2)\otimes \pi^c \otimes \pi^d$ has a braided tensor category equivalent to $\Rep^{wt}(u_\i^H)(\sl_2)\boxtimes \Vect_{\C^2}$, or equivalently $\YD{\Nichols}(\CC)$ for $\CC=\Vect_{\C^3}$ with $\Nichols=\C[x]/x^2$ with $x$ in degree $(-1,0,0)$. The vertex algebra $V^k(\mathfrak{gl}_{1|1})$ is a simple current extension thereof by $S=R_1^\epsilon$ with $R$ given in equation (\ref{gl11}). With the coordinate transformation given after the cited equation we have $\tilde{\CC}=\CC_R^0=\mathrm{sVect}\boxtimes \Vect_{\C^2}$ with braiding given by equation {\ref{braidingGl11}}. Hence by Corollary \ref{cor:uprolling} we have
\begin{corollary}\label{cor_gl11}
The braided tensor category of weight representations of $V^k(\mathfrak{gl}_{1|1})$ is equivalent to $\YD{\tilde{\Nichols}}(\tilde{\CC})$. The realizing quantum group for this braided tensor category is $u_q^B({\mathfrak{gl}_{1|1}})$ constructed in Example \ref{exm_gl11QG}.
\end{corollary}

\section{Appendix: Some categorical assertions}

\subsection{Simples, indecomposables and projectives over the tensor product of rings}

Suppose that we are over an algebraically closed field

\begin{lemma} \label{From Nikolaus for Thomas} 
Let $A, B$ be algebras, and $M$ an $A$-module, $N$ a $B$-module. If $M,N$ have both the following property  over $A,B$, then we assert that $M \otimes N$ has the same property over $A\otimes B$: 
\begin{itemize}
\item[a)] simple, under the assumption that $M,N$ are finite-dimensional. Conversely, every finite-dimensional simple is of the form $M\otimes N$.  
\item[b)] indecomposable, under the assumption that a) holds and that the Jordan H\"older length is finite.
\item[c)] free projective and injective 
\item[d)] projective cover and injective hull
\end{itemize}
\end{lemma}
\begin{proof}
a)  For finite-dimensional modules the statment is \cite{EGHLSVYG11} Section 3.10. Note the standard counterexmaple $A=B=M=N=\mathbb{C}(x)$, where any polynomial not in product form cannot be inverted, for example $(x-y)$ is an ideal in $\mathbb{C}(x)\otimes \mathbb{C}(y)$. 

b) We proceed by induction along the length of the composition series of $M,N$. If both are simple, then $M\otimes N$ is simple by a). Suppose now $M$ is an indecomposable extension
$$0 \to M'\stackrel{f}{\to} M\stackrel{g}{\to} M''\to 0$$
Assume that our assertion were wrong, that is $M\otimes N$ decomposes into a direct sum $V\oplus W$ as a $A\otimes B$-module, but by induction $M'\otimes N$ and $M''\otimes N$ are indecomposable. We have a short exact sequence of $A\otimes B$-modules
$$0 \to M'\otimes N \stackrel{f\otimes \id}{\to} M \otimes N \stackrel{g\otimes \id}{\to} M''\otimes N\to 0$$
The image under $g\otimes \id$ of the decomposition $V\otimes W$ is a decomposition of $M''\otimes N$, which is by assumption indecomposable, so the image of $V$ (or $W$) is zero and the kernel of $g\otimes \id$ is $V\oplus W'$ for some $W'\subset W$. By exactness and since $M''\otimes N$ is indecomposable by assumption, we have $W'=0$ and our decomposition $V\oplus W$ coicides with our exact sequence, which hence splits
$$0 \to \underbracket{M'\otimes N}_{V} \stackrel{f\otimes \id}{\to} M \otimes N \stackrel{s}{\leftrightarrows} \underbracket{M''\otimes N}_{W}\to 0$$
Since $s$ is a section of $g\otimes \id_N$ it preserves the projection to $N$ (note that $s$ does not have to be of the form $q\otimes \id_N$). For any fixed $\nu\in N^*$ we have hence a split short exact sequence
$$0 \to M' \stackrel{f\otimes \id}{\to} M \stackrel{s|\nu}{\leftrightarrows} M''\to 0$$
which is a contradiction to the assumed indecomposablility of $M$.

c) Suppose $M,N$ are free with an $A,B$-basis $m_i,n_i$ of any cardinality, then $m_i\otimes n_i$ is a $A\otimes B$-basis of $M\otimes N$.

Now we turn to the tensor product of projective $A,B$-modules.
 An $A$-module $P$ is projective iff for every surjection of $A$-modules $X\stackrel{f}{\to} Y$ and every map $P\stackrel{g}{\to} X$ we have a lift to $P\stackrel{\tilde{g}}{\to} Y$

$$\begin{tikzcd}
   & P \arrow{d}{g}\arrow[dotted, swap]{dl}{\tilde{g}} \\
    X \arrow{r}{f} & Y \arrow{r} & 0
  \end{tikzcd}$$
  
Now suppose that the $B$-module $P'$ is projective, meaning that for any $B$-module surjection $X'\stackrel{f'}{\to} Y'$ and morphism $P'\stackrel{g'}{\to} X'$ we have a lift $\tilde{g}$. 

We now prove that $P\otimes P'$ is a projective $A\otimes B$-module: Assume that we have a  $A\otimes B$-module surjection $X''\stackrel{f''}{\to} Y''$ and morphism $P\otimes P'\stackrel{g''}{\to} X''$. By restriction, $g''$ is a surjection of $A$-modules and of $B$-modules. Hence by the assumed projectivity we have lifts $P\stackrel{\tilde{g}}{\to} Y''$ and $P'\stackrel{\tilde{g}'}{\to} Y''$. Simply taking their tensor product gives a lift $P\otimes P'\stackrel{\tilde{g}\otimes \tilde{g}'}{\to} Y''$.

e) Completely analogously we see that the tensor product of injective $A,B$-modules $I,I'$ is an injective $A\otimes B$-module, via the universal property 

$$\begin{tikzcd}
   & I\otimes I' &\\
    0\arrow{r} & X''\arrow{u}{g''}\arrow{r}{f}  & Y'' \arrow[dotted, swap]{ul}{\tilde{g}\otimes \tilde{g}'} 
  \end{tikzcd}$$
and again using the lift $\tilde{g}\otimes \tilde{g}'$

f) This follows from b) and c).
\end{proof}

 The finiteness theorem for comodules states that any finitely generated comodule over a finite field is finite-dimensional,
 hence any simple comodule is finite-dimensional.

\begin{corollary}
The tensor product of simple $A,B$-comodules $M,N$ is a simple $A\otimes B$-comodule, and the  projective cover and injective hull is the tensor product of the projective covers or injective hulls of $M,N$. 
\end{corollary}

\subsection{A criterion for fully faithfulness}

\begin{lemma}\label{ff}
Let $\UU$ be an abelian category with the property that for any object $X$, there exists a projective object $P_X$ and an injective object $I_X$ with
a surjection $\pi_X : P_X \twoheadrightarrow X$ and an embedding $\iota_X:  X \hookrightarrow I_X$. 
 Let $\mathrm{Proj}$ and $\mathrm{Inj}$ denote the sets of projective and injective objects in $\UU$. Let $\mathcal F : \UU \rightarrow \mathcal V$ be an exact functor with the property that $\mathcal F: \text{Hom}_\UU(X, Y) \xrightarrow{\cong} \text{Hom}_{\mathcal V}(\mathcal F(X), \mathcal F(Y))$ for all $X, Y  \in \mathrm{Proj} \cup \mathrm{Inj}$. 
Then $\mathcal F$ is fully faithfull.
\end{lemma}
\begin{proof}

    Let $X, Y$ be objects in $\UU$  and
    consider $g: \mathcal F(X) \rightarrow \mathcal F(Y)$. 
Then
\[
\mathcal F(P_X) \xrightarrow{\mathcal F(\pi_X)} \mathcal F(X) \xrightarrow{g} 
\mathcal F(Y) \xrightarrow{\mathcal F(\iota_Y)} \mathcal F(I_Y) \quad \in \ \text{Hom}_{\mathcal V}\left(\mathcal F(P_X), \mathcal F(I_Y)\right).
\]
By assumption there exists $h \in \text{Hom}_\UU(P_X, I_Y)$ with $\mathcal F(h) = \mathcal F(\iota_Y) \circ g \circ \mathcal F(\pi_X)$.

Let $K_X$ be the kernel of $\pi_X$. We now prove that $K_X$ is in the kernel of $h$, so that $h$ factors through $X$. 
Since $\mathcal F$ is exact, the short exact sequence
\[
0 \rightarrow K_X \xrightarrow{\iota} P_X \xrightarrow{\pi_X} X \rightarrow 0
\]
gives the short exact sequence 
\[
0 \rightarrow \mathcal F(K_X) \xrightarrow{\mathcal F(\iota)} \mathcal F(P_X) \xrightarrow{\pi_X} \mathcal F(X) \rightarrow 0.
\]
Since $\mathcal F(h) = \mathcal F(\iota_Y) \circ g \circ \mathcal F(\pi_X)$ the object $\mathcal F(K_X)$ must be in the kernel of $\mathcal F(h) \circ \mathcal F(\iota)$.
We show that  $K_X$ not being in the kernel of $h$ is impossible: If $K_X$ were not in the kernel of $h$, then the map
\[
P_{K_X} \xrightarrow{\pi_{K_X}} K_X \xrightarrow{\iota} P_X \xrightarrow{h} I_Y 
\]
would be non-zero and hence 
\[
\mathcal F(P_{K_X}) \xrightarrow{\mathcal F(\pi_{K_X})} \mathcal F(K_X) \xrightarrow{\mathcal F(\iota)} \mathcal F(P_X) \xrightarrow{\mathcal F(h)} \mathcal F(I_Y) 
\]
would be non-zero as well, a contradiction. 

Next we show that the image of $h$ is contained in $Y$, so that $h$ factors through $Y$ as well. Let $coK_Y$ be the cokernel of $\iota_Y$, so that there is the short exact sequence
\[
0 \rightarrow  Y \xrightarrow{\iota_Y} I_Y \xrightarrow{\pi} coK_Y \rightarrow  0.
\]
We show that the image of $h$ not contained in $Y$ is impossible: If the image of $h$ is not in contained in $Y$, then the morphism
\[
P_X \xrightarrow{h} I_Y \xrightarrow{\pi} coK_Y 
\]
would be non-zero and hence
\[
P_X \xrightarrow{h} I_Y \xrightarrow{\pi} coK_Y  \xrightarrow{\iota_{coK_Y}} I_{coK_Y}
\]
as well. But then
\[
\mathcal F(P_X) \xrightarrow{\mathcal F(h)} \mathcal F(I_Y) \xrightarrow{\mathcal F(\pi)} \mathcal F(coK_Y)  \xrightarrow{\mathcal F(\iota_{coK_Y})} \mathcal F(I_{coK_Y})
\]
would also be non-zero, a contradiction. 

We have shown that there exists $\hat{h} \in \text{Hom}_\UU(X, Y)$ such that
$h = \iota_Y \circ \hat h \circ \pi_X$ and hence
\[
\mathcal F(\iota_Y) \circ \mathcal F(\hat h) \circ \mathcal F(\pi_X) = \mathcal F(h) = 
\mathcal F(\iota_Y) \circ g \circ \mathcal F(\pi_X).
\]
The functor $\mathcal F$ is exact and so $\mathcal F(\iota_Y)$ is an embedding and $\mathcal F(\pi_X)$ a surjection. Hence $\mathcal F(\hat h) =g$ and thus 
\[
\mathcal F: \text{Hom}_\UU(X, Y) \rightarrow \text{Hom}_{\mathcal V}(\mathcal F(X), \mathcal F(Y))
\]
is surjective. 
The same reasoning gives injectivity: Consider $f, g \in \text{Hom}_\UU(X, Y)$ with $\mathcal F(f) = \mathcal F(g)$, which  implies that 
\[
\mathcal F(\iota_Y) \circ \mathcal F(f) \circ \mathcal F(\pi_X)=
\mathcal F(\iota_Y) \circ \mathcal F(g) \circ \mathcal F(\pi_X)
\]
but then there exists $h \in \text{Hom}_\UU(P_X, I_Y)$ with $\mathcal F(h) = \mathcal F(\iota_Y) \circ \mathcal F(f) \circ \mathcal F(\pi_X)$ as before. The same argument as before then gives that $f=g$. 
\end{proof}

\begin{remark}\label{rem:ff} We actually proved the slightly stronger statement: 

    Let $\UU$ be an abelian category with 
    a subcategory $\CC$ with 
    the property that for any object $X$, there exists objects $P_X, I_X$ with
a surjection $\pi_X : P_X \twoheadrightarrow X$ and an embedding $\iota_X:  X \hookrightarrow I_X$. 
  Let $\mathcal F : \UU \rightarrow \mathcal V$ be an exact functor with the property that $\mathcal F: \text{Hom}_\UU(X, Y) \xrightarrow{\cong} \text{Hom}_{\mathcal V}(\mathcal F(X), \mathcal F(Y))$ for all $X, Y  \in \CC$. 
Then $\mathcal F$ is fully faithfull.
\end{remark}

\newcommand\arxiv[2]{\href{http://arXiv.org/abs/#1}{#2}}


\begin{thebibliography}{EGHLSVYG11}

\bibitem[AA20]{AA20} N. Andruskiewitsch, I. Angiono, I, On Nichols algebras over basic Hopf algebras, \href{https://doi.org/10.1007/s00209-020-02493-w}{Math. Z. 296, 1429–1469 (2020).} 

\bibitem[Ad03]{Ad03}
D.~Adamovi\'c, Classification of irreducible modules of certain subalgebras of free boson vertex algebra, \href{https://doi.org/10.1016/j.jalgebra.2003.07.011}{J. Algebra 270 (2003), no. 1, 115-132.}

\bibitem[Ad05]{Ad05}
D.~Adamovi\'c,  A construction of admissible $A^{(1)}_1$-modules of level $-\frac{4}{3}$,  \href{https://doi.org/10.1016/j.jpaa.2004.08.007}{J. Pure and Applied Algebra 196 (2005), no. 2-3, 119-134.}

\bibitem[Ad19]{Ad1}
D.~Adamovi\'c,
Realizations of Simple Affine Vertex Algebras and Their Modules: The Cases ${\widehat{sl(2)}}$ and ${\widehat{osp(1,2)}}$, \href{https://doi.org/10.1007/s00220-019-03328-4}{Commun. Math. Phys. \textbf{366} (2019) no.3, 1025-1067.}

 \bibitem[ABBGM05]{ABBGM} S. Arkhipov, A. Braverman, R. Bezrukavnikov, D.  Gaitsgory, I. Mirković, Modules over the small quantum group and semi-infinite flag manifold, \href{https://doi.org/10.1007/s00031-005-0401-5}{Transformation Groups 10 (2005), no. 3-4, 279–362.}

\bibitem[AI19]{AI19} I. Angiono, A. Garcia Iglesias, Liftings of Nichols algebras of diagonal type II: all liftings are cocycle deformations, \href{https://doi.org/10.1007/s00029-019-0452-4}{Sel. Math. New Ser. 25, 5 (2019).}

\bibitem[ACG]{ACG}
D.~Adamovic, T.~Creutzig and N.~Genra,
Relaxed and logarithmic modules of $\widehat{\mathfrak{sl}_3}$,
\href{https://arxiv.org/abs/2110.15203}{[arXiv:2110.15203 [math.RT]].}

\bibitem[ACGY21]{ACGY}
D.~Adamovic, T.~Creutzig, N.~Genra and J.~Yang,
The Vertex Algebras $\mathcal {R}^{(p)}$ and $\mathcal {V}^{({p})}$,
\href{https://doi.org/10.1007/s00220-021-03950-1}{Commun. Math. Phys. \textbf{383} (2021) no.2, 1207-1241.}

\bibitem[ACK]{ACK}
T. Arakawa, T. Creutzig, K. Kawasetsu, in preparation.

\bibitem[ACKR20]{ACKR} J. Auger, T. Creutzig, S. Kanade, M. Rupert, Braided Tensor Categories related to Bp Vertex Algebras, \href{https://doi.org/10.1007/s00220-020-03747-8}{Commun. Math. Phys. \textbf{378}, 219–260 (2020).}

\bibitem[AFGV10]{AFGV10} N. Andruskiewitsch, Fantino, Grana, Vendramin: On Nichols algebras associated to simple racks, \href{http://www.ams.org/books/conm/537/}{Cont. Math. Vol. 537, 2011.}

\bibitem[AG03]{AG}  S. Arkhipov, D. Gaitsgory, Another realization of the category of modules over the small quantum group,  \href{https://doi.org/10.1016/S0001-8708(02)00016-6}{Adv. Math. 173 (2003), no. 1, 114-143.}

\bibitem[AKM15]{AKM15} I. Angiono, M. Kochetov, M. Mastnak,
On rigidity of Nichols algebras, \href{https://doi.org/10.1016/j.jpaa.2015.05.032}{J. Pure and Applied Algebra, Volume 219/12 (2015).}

\bibitem[AM07]{AM07}
D.~Adamovic and A.~Milas,
Logarithmic intertwining operators and W(2,2p-1)-algebras, \href{https://doi.org/10.1063/1.2747725}{J. Math. Phys. \textbf{48} (2007), 073503.}

\bibitem[AM08]{AM08}
D.~Adamovic and A.~Milas,
On the triplet vertex algebra W(p),
\href{https://doi.org/10.1016/j.aim.2007.11.012}{Adv. Math. \textbf{217} (2008), 2664-2699.}

\bibitem[AM22]{AM22}
D.~Adamovi\'c, A.~Milas and Q.~Wang,
On parafermion vertex algebras of \ensuremath{\mathfrak{s}}\ensuremath{\mathfrak{l}}(2) and \ensuremath{\mathfrak{s}}\ensuremath{\mathfrak{l}}(3) at level -$\frac{3}{2}$, \href{https://doi.org/10.1142/S0219199720500868}{Commun. Contemp. Math. \textbf{24} (2022) no.01, 2050086.}

\bibitem[AGI19]{AGI19} I. Angiono, A. Garcia Iglesias, Pointed Hopf algebras: a guided tour to the liftings, \href{https://doi.org/10.15446/recolma.v53nsupl.83958}{Rev. Colomb. Math. 53 (2019) p. 1-44.}  

\bibitem[AHS10]{AHS10} N. Andruskiewitsch, I. Heckenberger, H.-J. Schneider, The Nichols algebra of a semisimple Yetter-Drinfeld module. \href{https://www.jstor.org/stable/40931047}{American Journal of Mathematics 132 (2010).}

\bibitem[AAH18]{AAH18} 
N. Andruskiewitsch, I.  Angiono, I. Heckenberger, Liftings of Jordan and Super Jordan Planes, \href{doi:10.1017/S0013091517000402}{Proceedings of the Edinburgh Mathematical Society, 61(3) (2018), 661-672.} 

\bibitem[AAH22]{AAH22} N. Andruskiewitsch, I.  Angiono, I. Heckenberger, Corrigendum: Liftings of the Jordan plane,  arXiv:2203.03350.

\bibitem[AS10]{AS10} N.~Andruskiewitsch, H.-J. Schneider, On the classification of finite-dimensional pointed Hopf algebras, \href{http://doi.org/10.4007/annals.2010.171.375}{Ann. of Math. 171/1 (2010), 375--417.}

\bibitem[ALSW21]{ALSW21}
R.~Allen, S.~Lentner, C.~Schweigert and S.~Wood,
Duality structures for module categories of vertex operator algebras and the Feigin Fuchs boson,
\href{https://arxiv.org/abs/2107.05718}{[arXiv:2107.05718 [math.QA]].}

\bibitem[ALSW23]{ALSW23}
R.~Allen, S.~Lentner, C.~Schweigert and S.~Wood,
in preparation.

\bibitem[AW22]{AW}
R.~Allen and S.~Wood,
Bosonic Ghostbusting: The Bosonic Ghost Vertex Algebra Admits a Logarithmic Module Category with Rigid Fusion,
\href{https://doi.org/10.1007/s00220-021-04305-6}{Commun. Math. Phys. \textbf{390} (2022) no.2, 959-1015.}

\bibitem[AY13]{AY13}
I.~Angiono and H.~Yamane, The R-matrix of quantum doubles of Nichols algebras of diagonal type,
\href{https://doi.org/10.1063/1.4907379}{J. Math. Phys. 56/2 (2015).}

\bibitem[Bea03]{Bea03} M. Beattie, Duals of pointed Hopf algebras, \href{https://doi.org/10.1016/S0021-8693(03)00034-6}{J. Algebra 262 (2003), 54-76.}

  \bibitem[Besp95]{Besp95}
    Yu. N. Bespalov:
    Crossed modules, quantum braided groups and ribbon structures,
    \href{https://doi.org/10.1007/BF02065863}{Theor. Math. Phys. 103 (1995)  368-387.}

\bibitem[BCDN23]{BCDN}
A.~Ballin, T.~Creutzig, T.~Dimofte and W.~Niu,
3d mirror symmetry of braided tensor categories,
\href{https://arxiv.org/abs/2304.11001}{[arXiv:2304.11001 [hep-th]].}

\bibitem[BF18]{BF18}
A.~Braverman and M.~Finkelberg,
Coulomb branches of 3-dimensional gauge theories and related structures,
\href{https://doi.org/10.1007/978-3-030-26856-5_1}{Lect. Notes Math. \textbf{2248} (2019), 1-52.}

\bibitem[BFN21]{BFN21}
A.~Braverman, M.~Finkelberg and H.~Nakajima,
Line bundles over Coulomb branches,
\href{https://doi.org/10.4310/ATMP.2021.v25.n4.a2}{Adv. Theor. Math. Phys. \textbf{25} (2021) no.4, 957-993.}

\bibitem[BLS14]{BLS14}
A. Barvels, S. Lentner and C. Schweigert, Partially dualized Hopf algebras have equivalent Yetter-Drinfel'd modules, \href{https://doi.org/10.1016/j.jalgebra.2015.02.010}{J. Algebra 430 (2014).}

\bibitem[BN11]{BN10} A. Brugui\`eres, S. Natale, Exact sequences of tensor categories,  \href{https://doi.org/10.1093/imrn/rnq294}{IMRN 2011, no. 24, 5644–5705.}

\bibitem[Bi23]{Bi23} J. Bichon, Faithful flatness of Hopf algebras over coideal subalgebras with a conditional expectation, Preprint (2023), \href{https://arxiv.org/abs/2301.05480}{[arXiv:2301.05480 [math.QA]].} 

\bibitem[BM21]{BM21} N. Bortolussi, M. Mombelli: (Co)ends for representations of tensor categories, Theory and Applications of Categories, Vol. 37, No. 6, (2021), 144–188.

\bibitem[C17]{C17}
	T.~Creutzig,  W-algebras for Argyres-Douglas theories, \href{https://doi.org/10.1007/s40879-017-0156-2}{Eur. J. Math. 3 (2017), no. 3, 659-690.}

\bibitem[C19]{C1}
	T.~Creutzig, Fusion categories for affine vertex algebras at admissible levels. \href{https://doi.org/10.1007/s00029-019-0479-6}{Selecta Math. (N.S.) 25 (2019), no. 2, Paper No. 27, 21 pp.}

 \bibitem[C23]{C2}
	T.~Creutzig, in preparation.

\bibitem[CDGG21]{CDGG}
T.~Creutzig, T.~Dimofte, N.~Garner and N.~Geer,
A QFT for non-semisimple TQFT,
\href{https://arxiv.org/abs/2112.01559}{[arXiv:2112.01559 [hep-th]].}

\bibitem[CGN21]{CGN21}
T.~Creutzig, N.~Genra and S.~Nakatsuka,
Duality of subregular $W$-algebras and principal $W$-superalgebras,
\href{https://doi.org/10.1016/j.aim.2021.107685}{Adv. Math. \textbf{383} (2021), 107685.}

	\bibitem[CGP15]{CGP} 
	F.~Costantino, N.~Geer, B.~Patureau-Mirand,
	{\it Some remarks on the unrolled quantum group of sl(2)},
	\href{https://doi.org/10.1016/j.jpaa.2014.10.012}{J.\ Pure and Applied Algebra {\bf 219} (2015) 3238--3262}

\bibitem[CGR20]{CGR20}
T.~Creutzig, A.~M.~Gainutdinov and I.~Runkel,
A quasi-Hopf algebra for the triplet vertex operator algebra,
\href{https://doi.org/10.1142/S021919971950024X}{Commun. Contemp. Math. \textbf{22} (2020) no.03, 1950024.}

\bibitem[CH09]{CH09} M. Cuntz, I.  Heckenberger, Finite Weyl groupoids of rank three, \href{https://www.jstor.org/stable/41407821}{Transactions of the American Mathematical Society 364 (2009).}

\bibitem[CHY18]{CHY}
T.~Creutzig, Y.~Z.~Huang and J.~Yang,
Braided tensor categories of admissible modules for affine Lie algebras,
\href{https://doi.org/10.1007/s00220-018-3217-6}{Commun. Math. Phys. \textbf{362} (2018) no.3, 827-854.}

\bibitem[CJORY21]{CJORY}
T.~Creutzig, C.~Jiang, F.~Orosz Hunziker, D.~Ridout and J.~Yang,
Tensor categories arising from the Virasoro algebra, \href{https://doi.org/10.1016/j.aim.2021.107601}{Adv. Math. \textbf{380} (2021), 107601.}

\bibitem[CMOY]{CMOY}
T.~Creutzig, R. McRae, F.~Orosz Hunziker,  and J.~Yang,
in preparation.

\bibitem[CKL20]{CKL}
T.~Creutzig, S.~Kanade and A.~R.~Linshaw,
Simple current extensions beyond semi-simplicity, \href{https://doi.org/10.1142/S0219199719500019}{Commun. Contemp. Math. \textbf{22} (2020) no.01, 1950001.}

\bibitem[CKLR19]{CKLR}
T.~Creutzig, S.~Kanade, A.~R.~Linshaw and D.~Ridout,
Schur-Weyl Duality for Heisenberg Cosets, \href{https://doi.org/10.1007/s00031-018-9497-2}{Transform. Groups \textbf{24} (2019), 301-354.}

\bibitem[CKM]{CKM}
	T.~Creutzig, S.~Kanade and R.~McRae,
	tensor categories for vertex operator superalgebra extensions, to appear in Memoirs of the AMS, \href{https://arxiv.org/abs/1705.05017}{[arXiv:1705.05017 [math.QA]].}

\bibitem[CKM22]{CKM2}
T.~Creutzig, S.~Kanade and R.~McRae, Gluing vertex algebras,
\href{https://doi.org/10.1016/j.aim.2021.108174}{Adv. Math. \textbf{396} (2022), 108174.}

\bibitem[CL17]{CL17} M. Cuntz and  S. Lentner, A simplicial complex of Nichols algebras, \href{https://doi.org/10.1007/s00209-016-1711-0}{Mathematische Zeitschrift 285 (2017).} 

\bibitem[CL22a]{CL22a}
T.~Creutzig and A.~R.~Linshaw, Trialities of W-algebras, \href{https://dx.doi.org/10.4310/CJM.2022.v10.n1.a2}{Camb. J. Math. 10 (2022), no. 1, 69–194.}

\bibitem[CL22b]{CL22b}
T.~Creutzig and A.~R.~Linshaw,
Trialities of orthosymplectic W-algebras,
\href{https://doi.org/10.1016/j.aim.2022.108678}{Adv. Math. \textbf{409} (2022), 108678.}

\bibitem[CLR]{CLR}
T.~Creutzig, S.~Lentner and M.~Rupert,
Characterizing braided tensor categories associated to logarithmic vertex operator algebras,
\href{https://arxiv.org/abs/2104.13262}{[arXiv:2104.13262 [math.QA]].}

 \bibitem[CM17]{CM}
T.~Creutzig and A.~Milas,
False Theta Functions and the Verlinde formula,
\href{https://doi.org/10.1016/j.aim.2014.05.018}{Adv. Math. \textbf{262} (2014), 520-545.}

 \bibitem[CM17]{CM17}
T.~Creutzig and A.~Milas,
Higher rank partial and false theta functions and representation theory,
\href{https://doi.org/10.1016/j.aim.2017.04.027}{Adv. Math. \textbf{314} (2017), 203-227.}

\bibitem[CMR18]{CMR} T.Creutzig, A. Milas, M. Rupert, Logarithmic Link Invariants of $\overline{U}_q^H(\mathfrak{sl}_2)$ and Asymptotic Dimensions of Singlet Vertex Algebras, \href{https://www.sciencedirect.com/science/article/abs/pii/S0022404917302876?via%3Dihub}{J. Pure and Applied Algebra 222,  10, (2018),  3224-3247}.

\bibitem[CMY21]{CMY}
		T.~Creutzig, R.~McRae and J.~Yang,
		On ribbon categories for singlet vertex algebras,
		\href{https://doi.org/10.1007/s00220-021-04097-9}{Commun. Math. Phys. 387, 865–925 (2021).}

\bibitem[CMY22a]{CMY3}
T.~Creutzig, R.~McRae and J.~Yang,
Direct limit completions of vertex tensor categories,
\href{https://doi.org/10.1142/S0219199721500334}{Commun. Contemp. Math. \textbf{24} (2022) no.02, 2150033.}

\bibitem[CMY22b]{CMY5}
		T.~Creutzig, R.~McRae and J.~Yang, Tensor structure on the Kazhdan-Lusztig category for affine gl(1$|$1). \href{https://doi.org/10.1093/imrn/rnab080}{Int. Math. Res. Not. IMRN 2022, no. 16, 12462–12515.}


\bibitem[CMY23a]{CMY4}
T.~Creutzig, R.~McRae and J.~Yang,
Ribbon tensor structure on the full representation categories of the singlet vertex algebras,
\href{https://doi.org/10.1016/j.aim.2022.108828}{Adv. Math. \textbf{413} (2023), 108828.}

\bibitem[CMY23b]{CMY2}
T.~Creutzig, R.~McRae and J.~Yang,
Rigid tensor structure on big module categories for some $W$-(super)algebras in type $A$,
\href{https://arxiv.org/abs/2210.04678}{[arXiv:2210.04678 [math.QA]].}

\bibitem[CNS]{CNS}
T. Creutzig, S. Nakatsuka, S. Sugimoto, in preparation.  

\bibitem[CRo09]{CRo}
T.~Creutzig and P.~B.~Ronne,
The GL(1$|$1)-symplectic fermion correspondence,
\href{https://doi.org/10.1016/j.nuclphysb.2009.02.013}{Nucl. Phys. B \textbf{815} (2009), 95-124.}

\bibitem[CRR23]{CRR}
T.~Creutzig, D.~Ridout and M.~Rupert,
A Kazhdan\textendash{}Lusztig Correspondence for $L_{-\frac{3}{2}}(\mathfrak {sl}_3)$, \href{http://dx.doi.org/10.1007/s00220-022-04602-8}{Commun. Math. Phys. \textbf{400} (2023) no.1, 639-682.}

\bibitem[CRW14]{CRW}
T.~Creutzig, D.~Ridout and S.~Wood,
Coset Constructions of Logarithmic (1, p) Models, \href{https://doi.org/10.1007/s11005-014-0680-7}{Lett. Math. Phys. \textbf{104} (2014), 553-583.}

\bibitem[CY21]{CY} T. Creutzig, J. Yang, Tensor categories of affine Lie algebras beyond admissible levels, \href{https://doi.org/10.1007/s00208-021-02159-w}{Math. Ann. 380 (2021), no. 3-4, 1991-2040.}

\bibitem[DMNO13]{DMNO} A. Davydov, M. M\"uger, D. Nikshych, V. Ostrik,  The Witt group of non-degenerate braided fusion categories, \href{https://doi.org/10.1515/crelle.2012.014}{J. Reine Angew. Math. 677 (2013), 135-177.} 

\bibitem[DNO13]{DNO} A. Davydov, M. M\"uger, D. Nikshych, V. Ostrik, On the structure of the Witt group of braided fusion categories, \href{https://doi.org/10.1007/s00029-012-0093-3}{Selecta Math. (N.S.) 19 (2013), no. 1, 237-269.} 


\bibitem[EGHLSVYG11]{EGHLSVYG11} 
P. Etingof, O. Golberg, S. Hensel, T. Liu, A. Schwendner, D. Vaintrob, E. Yudovina, S.Gerovitch, Introduction to representation theory, Student Mathematical Library 59 (2011), AMS. 

\bibitem[EGNO]{EGNO} P. Etingof, S. Gelaki, D. Nikshych, V. Ostrik, {\it Tensor Categories}, American Mathematical Society, Mathematical Surveys and Monographs, Vol. 205, 2015.

\bibitem[ENOM09]{ENOM09} P. Etingof, D. Nikshych, V. Ostrik, E. Meir, Fusion categories and homotopy theory, \href{http://dx.doi.org/10.4171/QT/6}{Quantum Topol. 1 (2010), no. 3, pp. 209–273.}

\bibitem[FF90a]{FF90} B. Feigin, E. Frenkel, Affine Kac-Moody algebras and semi-infinite flag manifolds, \href{https://doi.org/10.1007/BF02097051}{Comm. Math. Phys., 128(1):161-189, 1990.}

\bibitem[FF90b]{FF90b} B. Feigin, E. Frenkel, Quantization of the Drinfel'd-Sokolov reduction, \href{https://doi.org/10.1016/0370-2693(90)91310-8}{Phys. Lett. B, 246(1-2):75-81, 1990.}

\bibitem[FGST06a]{FGST06}
B.~L.~Feigin, A.~M.~Gainutdinov, A.~M.~Semikhatov and I.~Y.~Tipunin,
Modular group representations and fusion in logarithmic conformal field theories and in the quantum group center,
\href{https://doi.org/10.1007/s00220-006-1551-6}{Commun. Math. Phys. \textbf{265} (2006), 47-93.}

\bibitem[FGST06b]{FGST06b}
B.~L.~Feigin, A.~M.~Gainutdinov, A.~M.~Semikhatov and I.~Y.~Tipunin,
Kazhdan-Lusztig correspondence for the representation category of the triplet W-algebra in logarithmic CFT,
\href{https://doi.org/10.1007/s11232-006-0113-6}{Theor. Math. Phys. \textbf{148} (2006), 1210-1235.}

\bibitem[FGST07]{FGST07}
B.~L.~Feigin, A.~M.~Gainutdinov, A.~M.~Semikhatov and I.~Y.~Tipunin,
Kazhdan-Lusztig-dual quantum group for logarithmic extensions of Virasoro minimal models,
\href{https://doi.org/10.1063/1.2423226}{J. Math. Phys. \textbf{48} (2007), 032303.}

\bibitem[FL22]{FL22} I. Flandoli, S. Lentner, Algebras of Non-Local Screenings and Diagonal Nichols Algebras, \href{https://doi.org/10.3842/SIGMA.2022.018}{SIGMA 18 (2022).}

\bibitem[FRS02]{FRS02}
J.~Fuchs, I.~Runkel and C.~Schweigert,
TFT construction of RCFT correlators 1. Partition functions,
\href{https://doi.org/10.1016/S0550-3213(02)00744-7}{Nucl. Phys. B \textbf{646} (2002), 353-497.}

\bibitem[FT10]{FT10}
B.~L.~Feigin and I.~Y.~Tipunin,
Logarithmic CFTs connected with simple Lie algebras,
\href{https://arxiv.org/abs/1002.5047}{[arXiv:1002.5047 [math.QA]].}

\bibitem[Ge17]{Gen} N. Genra, Screening operators for W-algebras, \href{https://doi.org/10.1007/s00029-017-0315-9}{Selecta Math. (N.S.) 23 (2017), no. 3, 2157-2202.}

\bibitem[Gu93]{Gu93}
V.~Gurarie,
Logarithmic operators in conformal field theory,
\href{https://doi.org/10.1016/0550-3213(93)90528-W}{Nucl. Phys. B \textbf{410} (1993), 535-549.}

\bibitem[GK96]{GK96}
M.~R.~Gaberdiel and H.~G.~Kausch,
A Rational logarithmic conformal field theory,
\href{https://doi.org/10.1016/0370-2693(96)00949-5}{Phys. Lett. B \textbf{386} (1996), 131-137.}

\bibitem[GLO18]{GLO18}
		A.~Gainutdinov, S.~Lentner, and T.~Ohrmann, Modularization of small quantum groups, \href{https://arxiv.org/abs/1809.02116}{[arXiv:1809.02116 [math.QA]].}

\bibitem[GLOR23]{GLR23} A. Gainutdinov, S. Lentner, T. Ohrmann, M. Rupert, Commutative algebras over Hopf algebras and Drinfeld centers, in preparation.

\bibitem[GN]{GN}
T.~Gannon and C.~Negron,
Quantum SL(2) and logarithmic vertex operator algebras at (p,1)-central charge,
\href{https://arxiv.org/abs/2104.12821}{[arXiv:2104.12821 [math.QA]].}

\bibitem[GNN09]{GNN09} 
S. Gelaki, D. Naidu, D.  Nikshych, Centers of graded fusion categories, \href{http://dx.doi.org/10.2140/ant.2009.3.959}{Algebra \& Number Theory 3  (2009).} 

\bibitem[GR19]{GR}
D.~Gaiotto and M.~Rap\v{c}\'ak, Vertex Algebras at the Corner,
\href{https://doi.org/10.1007/JHEP01(2019)160}{JHEP \textbf{01} (2019), 160.}

\bibitem[GRW09]{GRW09}
M. Gaberdiel, I. Runkel, S. Wood, Fusion rules and boundary conditions in the c=0 triplet model. \href{https://doi.org/10.1088/1751-8113/42/32/325403}{Journal of Physics A: Mathematical and Theoretical 42 (2009).}

\bibitem[GY22]{GY22}
N.~Geer and M.~B.~Young,
Three dimensional topological quantum field theory from $U_q(\mathfrak{gl}(1 \vert 1))$ and $U(1 \vert 1)$ Chern--Simons theory,
\href{https://arxiv.org/abs/2210.04286}{[arXiv:2210.04286 [math.QA]].}

\bibitem[H09]{H09}
Y.~Z.~Huang, Cofiniteness conditions, projective covers and the logarithmic tensor product theory, \href{https://doi.org/10.1016/j.jpaa.2008.07.016}{J. Pure Appl. Algebra 213 (2009), no. 4, 458–475.}

\bibitem[HKL]{HKL}
Y.~Z.~Huang, A.~Kirillov and J.~Lepowsky,
Braided tensor categories and extensions of vertex operator algebras,
\href{https://doi.org/10.1007/s00220-015-2292-1}{Commun. Math. Phys. \textbf{337} (2015) no.3, 1143-1159.}

\bibitem[HLZ1]{HLZ0}
Y.-Z. Huang, J, Lepowsky and L. Zhang, A logarithmic generalization of tensor product theory for modules for a vertex operator algebra,
\href{https://doi.org/10.1142/S0129167X06003758}{{\it Internat. J. Math.} {\bf 17} (2006), 975--1012.}

\bibitem[HLZ2]{HLZ1}
Y.-Z. Huang, J, Lepowsky and L. Zhang, Logarithmic tensor category theory for generalized modules for a conformal vertex algebra, I: Introduction and strongly graded algebras and their generalized modules, \href{https://doi.org/10.1007/978-3-642-39383-9_5}{Conformal Field Theories and Tensor Categories.} Mathematical Lectures from Peking University. Springer, Berlin, Heidelberg.

\bibitem[HLZ3]{HLZ2}
Y.-Z. Huang, J, Lepowsky and L. Zhang, Logarithmic tensor category theory for generalized modules for a conformal vertex algebra, II: Logarithmic formal calculus and properties of logarithmic intertwining operators, \href{https://arxiv.org/abs/1012.4196}{[arXiv:1012.4196].}

\bibitem[HLZ4]{HLZ3}
Y.-Z. Huang, J, Lepowsky and L. Zhang, Logarithmic tensor category theory for generalized modules for a conformal vertex algebra, III: Intertwining maps and tensor product bifunctors, \href{https://arxiv.org/abs/1012.4197}{[arXiv:1012.4197].}

\bibitem[HLZ5]{HLZ4}
Y.-Z. Huang, J, Lepowsky and L. Zhang, Logarithmic tensor category theory for generalized modules for a conformal vertex algebra, IV: Construction of tensor product bifunctors and the compatibility conditions, \href{https://arxiv.org/abs/1012.4198}{[arXiv:1012.4198].}

\bibitem[HLZ6]{HLZ5}
Y.-Z. Huang, J, Lepowsky and L. Zhang, Logarithmic tensor category theory for generalized modules for a conformal vertex algebra, V: Convergence condition for intertwining maps and the corresponding compatibility condition, \href{https://arxiv.org/abs/1012.4199}{[arXiv:1012.4199].}

\bibitem[HLZ7]{HLZ6}
Y.-Z. Huang, J, Lepowsky and L. Zhang, Logarithmic tensor category theory for generalized modules for a conformal vertex algebra, VI: Expansion condition, associativity of logarithmic intertwining operators, and the associativity isomorphisms, \href{https://arxiv.org/abs/1012.4202}{[arXiv:1012.4202].}

\bibitem[HLZ8]{HLZ7}
Y.-Z. Huang, J, Lepowsky and L. Zhang, Logarithmic tensor category theory for generalized modules for a conformal vertex algebra, VII: Convergence and extension properties and applications to expansion for intertwining maps, \href{https://arxiv.org/abs/1110.1929}{[arXiv:1110.1929].}

\bibitem[HLZ9]{HLZ8}
Y.-Z. Huang, J, Lepowsky and L. Zhang, Logarithmic tensor category theory for generalized modules for a conformal vertex algebra, VIII: Braided tensor category structure on categories of generalized modules for a conformal vertex algebra, \href{https://arxiv.org/abs/1110.1931}{[arXiv:1110.1931].}

\bibitem[HR21]{HR21}
J.~Hilburn and S.~Raskin,
Tate's thesis in the de Rham Setting,
\href{https://www.ams.org/journals/jams/2023-36-03/S0894-0347-2022-01010-5/home.html}{J. Amer. Math. Soc. 36 (2023), 917-1001.}

\bibitem[Heck09]{Heck09} I. Heckenberger, Classification of arithmetic root systems, \href{https://doi.org/10.1016/j.aim.2008.08.005}{Adv. Math. \textbf{220}, no.1 (2009), 59-124.}

\bibitem[HS10]{HS10} I. Heckenberger, H.-J. Schneider, Root systems and Weyl groupoids for Nichols
algebras, \href{https://doi.org/10.1112/plms/pdq001}{Proc. Lond. Math. Soc. 101 (2010) 623-654.}

\bibitem[HS13]{HS13} I. Heckenberger, H.-J. Schneider, Right coideal subalgebras of Nichols algebras and the Duflo order on the Weyl groupoid, \href{https://doi.org/10.1007/s11856-012-0180-3}{Isr. J. Math. 197, 139–187 (2013).}

\bibitem[HS20]{HS20} I. Heckenberger, H.-J. Schneider, Hopf algebras and root systems, Mathematical Surveys and Monographs, Vol. 247, American Mathematical Society, Providence, RI (2020).

\bibitem[HV14]{HV14} I. Heckenberger, L. Vendramin, A classification of Nichols algebras of semi-simple Yetter-Drinfeld modules over non-abelian groups, \href{https://doi.org/10.4171/jems/667}{Journal of the European Mathematical Society 19 (2014).}

\bibitem[HY10]{HY10} I. Heckenberger and H. Yamane, Drinfel'd doubles and Shapovalov determinants, \href{https://dialnet.unirioja.es/ejemplar/280978}{Revista de la Unión Matemática Argentina 51 (2010), 107–146.}

\bibitem[Ka91]{Ka91}
H.~G.~Kausch,
Extended conformal algebras generated by a multiplet of primary fields, \href{https://doi.org/10.1016/0370-2693(91)91655-F}{Phys. Lett. B \textbf{259} (1991), 448-455.}

\bibitem[KK]{KK}
V. Kac and D. Kazhdan, Structure of representations with highest weight of infinite dimensional Lie algebras, \href{https://doi.org/10.1016/0001-8708(79)90066-5}{{ Adv. Math.} {\bf 34} (1979), 97--184.}

\bibitem[KL1]{KL1}D. Kazhdan and G. Lusztig, Affine Lie algebras and quantum groups, \href{https://doi.org/10.1155/S1073792891000041}{IMRN {\bf 2} (1991), 21--29.}

\bibitem[KL2]{KL2}D. Kazhdan and G. Lusztig, Tensor structure arising from affine Lie algebras, I, \href{https://doi.org/10.2307/2152745}{{\em J. Amer. Math. Soc.} {\bf 6} (1993), 905--947.}

\bibitem[KL3]{KL3}D. Kazhdan and G. Lusztig, Tensor structure arising from affine Lie algebras, II, \href{https://doi.org/10.2307/2152746}{{\em J. Amer. Math. Soc.} {\bf 6} (1993), 949--1011.}

\bibitem[KL4]{KL4}D. Kazhdan and G. Lusztig, Tensor structure arising from affine Lie algebras, III, \href{https://doi.org/10.2307/2152762}{{\em J. Amer. Math. Soc.} {\bf 7} (1994), 335--381.}

\bibitem[KL5]{KL5}D. Kazhdan and G. Lusztig, Tensor structure arising from affine Lie algebras, IV, \href{https://doi.org/10.2307/2152763}{{\em J. Amer. Math. Soc.} {\bf 7} (1994), 383--453.}
 
\bibitem[KO]{KO} A. Kirillov Jr., V. Ostrik, On a q-analogue of the McKay correspondence and the ADE classification of sl2 conformal field theories. \href{https://doi.org/10.1006/aima.2002.2072}{Adv. Math. \textbf{171} (2002), no. 2, 183-227}. 

\bibitem[KW04]{KW04} V. G. Kac, M. Wakimoto,  Quantum reduction and representation theory of superconformal algebras. \href{https://doi.org/10.1016/j.aim.2003.12.005}{Adv. Math. 185 (2004), no. 2, 400-458.}

\bibitem[Lusz93]{Lusz93} G. Lusztig, Introduction to Quantum Groups, Birkhäuser 1993.

\bibitem[Len13]{Len13} S. Lentner, New Large-Rank Nichols Algebras Over Nonabelian Groups With Commutator Subgroup $\mathbb{Z}_2$. \href{https://doi.org/10.1016/j.jalgebra.2014.07.017}{J. Algebra 419 (2014).} 

\bibitem[Len21]{Len21} S.~D.~Lentner, Quantum groups and Nichols algebras acting on conformal field theories, \href{https://doi.org/10.1016/j.aim.2020.107517}{Adv. Math \textbf{378} (2021) 107517}.

\bibitem[Lau20]{Lau20} R. Laugwitz, The relative monoidal center and tensor products of monoidal categories,  \href{https://doi.org/10.1142/S0219199719500688}{Commun. in Cont. Math. Vol. 22, No. 8 (2020). }

            \bibitem[LW1]{LW1} R. Laugwitz, C. Walton, Constructing Non-Semisimple Modular Categories With Relative Monoidal Centers, \href{https://doi.org/10.1093/imrn/rnab097}{IMRN 10 Volume 2022, 15826–15868}

            \bibitem[LW2]{LW2} R. Laugwitz, C. Walton, Constructing Non-Semisimple Modular Categories With Local Modules \href{https://arxiv.org/pdf/2202.08644.pdf}{[arXiv:2202.08644]}

  \bibitem[LV21]{LV21} S. Lentner, K. Vocke, A Family of New Borel Subalgebras of Quantum Groups, \href{https://doi.org/10.1007/s10468-020-09956-y}{Algebras and Representation Theory. 24 (2021)}
  
\bibitem[Mac52]{Mac52} S. MacLane, Cohomology theory of Abelian groups, Proceedings of the International Congress of
Mathematicians, Cambridge, Mass., 1950, vol. 2, pages 8–14. Amer. Math. Soc., Providence, R. I., 1952.

\bibitem[Maj91]{Maj91} S. Majid, Representations, duals and quantum doubles of monoidal categories, Rend. Circ. Mat. Palermo (2) Suppl., 26 (1991), 197-206.

\bibitem[Maj94]{Maj94} S. Majid, Algebras and Hopf algebras in braided categories, Advances in Hopf Algebras (Chicago, IL, 1992), Lecture Notes in Pure and Appl. Math., Vol. 158, New York: Dekker, 1994, 55-105.

\bibitem[McR20]{McR2}
R.~McRae,
On the tensor structure of modules for compact orbifold vertex operator algebras,
\href{https://doi.org/10.1007/s00209-019-02445-z}{Math. Z. \textbf{296}, pg 409–452, (2020).}

\bibitem[McR21a]{McR1}
R.~Mcrae,
Twisted modules and $G$-equivariantization in logarithmic conformal field theory,
\href{https://doi.org/10.1007/s00220-020-03882-2}{Commun. Math. Phys. \textbf{383} (2021) no.3, 1939-2019.}

\bibitem[McR21b]{McR3} R. McRae, A general mirror equivalence theorem for coset vertex operator algebras, \href{https://arxiv.org/abs/2107.06577}{[arXiv:2107.06577].}

\bibitem[MS00]{MS00} A. Milinski, H.-J. Schneider, Pointed indecomposable Hopf algebras over Coxeter groups, \href{http://dx.doi.org/10.1090/conm/267}{Contemp. Math. 267 (2000), 215-236.}

\bibitem[McL50]{McL} S. MacLane, Cohomology theory of Abelian groups, Proceedings of the International Congress of Mathematicians, Cambridge, Mass., 1950, vol. 2, pp. 8–14. 

\bibitem[McRY]{McRY} R. McRae, J. Yang, Structure of Virasoro tensor categories at central charge $13-6p-6p^{-1}$ for integers $p > 1$,  \href{https://arxiv.org/abs/2011.02170}{[arXiv:2011.02170].}

\bibitem[Mo22]{Mo22} M. Mombelli, Relative Adjoint Algebras, Preprint (2022), \href{https://arxiv.org/abs/2212.07390}{[arXiv:2212.07390].}

\bibitem[NZ89]{NZ} W. Nichols, M. B. Zoeller, A Hopf algebra freeness theorem, \href{https://doi.org/10.2307/2374514}{Amer. J. Math. 111 (1989), no. 2, 381–385.}

\bibitem[O08]{O08} V.Ostrik, Module categories over representations of SLq(2) in the non-semisimple case, \href{http://dx.doi.org/10.1007/s00039-007-0637-4}{Geom. Funct. Anal., 17(6):2005–2017, 2008.}

\bibitem[QS07]{QS07}
T.~Quella and V.~Schomerus, Free fermion resolution of supergroup WZNW models,
\href{http://dx.doi.org/10.1088/1126-6708/2007/09/085}{JHEP \textbf{09} (2007), 085.}

\bibitem[Sch]{Sch} P. Schauenburg, The monoidal center construction and bimodules, \href{https://doi.org/10.1016/S0022-4049(00)00040-2}{J. Pure Appl. Algebra
158 (2001), no. 2-3, 325-346.}


\bibitem[S11]{S11}
A.~M.~Semikhatov,
Virasoro central charges for Nichols algebras, In: Bai, C., Fuchs, J., Huang, YZ., Kong, L., Runkel, I., Schweigert, C. (eds) \href{https://doi.org/10.1007/978-3-642-39383-9_3}{Conformal Field Theories and Tensor Categories.} Mathematical Lectures from Peking University. Springer, Berlin, Heidelberg.

\bibitem[S12]{S12}
A.~M.~Semikhatov,
Fusion in the entwined category of Yetter--Drinfeld modules of a rank-1 Nichols algebra,
\href{https://doi.org/10.1007/s11232-012-0118-2}{Theor. Math. Phys. \textbf{173} (2012), 1329-1358.}

\bibitem[ST12]{ST12}
A.~M.~Semikhatov and I.~Y.~Tipunin,
The Nichols algebra of screenings,
\href{https://doi.org/10.1142/S0219199712500290}{Commun. Contemp. Math. \textbf{14} (2012) no.04, 1250029.}

\bibitem[ST13]{ST13}
A.~M.~Semikhatov and I.~Y.~Tipunin,
Logarithmic $\widehat{s\ell}(2)$ CFT models from Nichols algebras. I,
\href{https://iopscience.iop.org/article/10.1088/1751-8113/46/49/494011}{J. Phys. A \textbf{46} (2013), 494011.}

\bibitem[Rad85]{Rad85} D. E. Radford, The Structure of Hopf Algebras with a Projection, \href{https://doi.org/10.1016/0021-8693(85)90124-3}{J. Algebra 92(1985), 322–374.}

\bibitem[RS92]{RS92}
L.~Rozansky and H.~Saleur,
S and T matrices for the superU(1,1) WZW model: Application to surgery and three manifolds invariants based on the Alexander-Conway polynomial, \href{https://doi.org/10.1016/0550-3213(93)90326-K}{Nucl. Phys. B \textbf{389} (1993), 365-423.}

\bibitem[Sh19]{Sh19}  K. Shimizu, Non-degeneracy conditions for braided finite tensor categories, \href{https://doi.org/10.1016/j.aim.2019.106778}{Adv. Math. 355 (2019), 106778, 36 pp.}
  
  \bibitem[Skry06]{Skry06} Skryabin: Projectivity and freeness over comodule algebras, \href{https://www.jstor.org/stable/20161694}{Trans. AMS 359 / 6 (2007) 2597-2623.}

\bibitem[Som96]{Som96}
Y. Sommerh\"auser, Deformed enveloping algebras, \href{https://nyjm.albany.edu/j/1996/2-3.html}{New York J. Math. 2 (1996), 35–58.}

\bibitem[Su21]{Su21}
S.~Sugimoto,
On the Feigin\textendash{}Tipunin conjecture,
\href{https://doi.org/10.1007/s00029-021-00662-1}{Selecta Math. \textbf{27} (2021) no.5, 86.}

\bibitem[Su22]{Su22} S.~Sugimoto,
 Simplicity of higher rank triplet W-algebras, 
\href{https://doi.org/10.1093/imrn/rnac189}{Int. Math. Res. Not., rnac189, (2022).}

\bibitem[SL23]{SL23}
R.~Laugwitz and G.~Sanmarco,
Finite-dimensional quantum groups of type Super A and non-semisimple modular categories,
\href{https://arxiv.org/abs/2301.10685}{[arXiv:2301.10685 [math.QA]].}

\bibitem[SS06]{SS06}
V.~Schomerus and H.~Saleur,
The GL(1$|$1) WZW model: From supergeometry to logarithmic CFT,
\href{https://doi.org/10.1016/j.nuclphysb.2005.11.013}{Nucl. Phys. B \textbf{734} (2006), 221-245.}

\bibitem[SW21]{SW21} C. Schweigert, L. Woike,
Homotopy coherent mapping class group actions and excision for Hochschild complexes of modular categories, \href{https://doi.org/10.1016/j.aim.2021.107814}{Advances in Mathematics, 386, (2021), 107814}.


\bibitem[Tak79]{Tak79} Takeuchi, Relative Hopf modules - Equivalence and Freeness Criteria, \href{https://doi.org/10.1016/0021-8693(79)90093-0}{J. Algebra, Vol. 60, Issue 2, 1979, Pg 452-471.}

\bibitem[TW13]{TW}
A.~Tsuchiya and S.~Wood,
The tensor structure on the representation category of the $W_{p}$ triplet algebra,
\href{https://iopscience.iop.org/article/10.1088/1751-8113/46/44/445203/meta}{J. Phys. A \textbf{46} (2013), 445203.}

\bibitem[V69]{V69} W. Vasconcelos, On finitely generated flat modules,
\href{http://dx.doi.org/10.2307/1994928}{Trans. Amer. Math. Soc. 138 (1969), 505-512.}

\bibitem[Vay19]{Vay19} C. Vay, On Hopf algebras with triangular decomposition, \href{https://doi.org/10.1090/conm/728/14662}
{Contemporary Mathematics
728 (2019), 181-199.}
.

\bibitem[W89]{Wor} S. L. Woronowicz,  Differential calculus on compact matrix pseudogroups (quantum groups), \href{https://doi.org/10.1007/BF01221411}{Commun. Math. Phys. 122 (1989), no. 1, 125-170.}

\bibitem[Web19]{W19} B. Webster, Coherent sheaves and quantum Coulomb branches I: tilting bundles from integrable systems, \href{https://arxiv.org/abs/1905.04623}{[arXiv:1905.04623 [math.AG]].}
  
\end{thebibliography}
\end{document}